\documentclass[11pt]{amsart}

\usepackage{amsmath,amssymb,amsthm}
\usepackage{amsbsy}
\usepackage[latin1]{inputenc}
\usepackage{mathrsfs}
\usepackage{version, tabularx, multicol}
\usepackage{graphicx,float,psfrag}
\usepackage{subfigure}
\usepackage{epstopdf}
\usepackage{color,latexsym,amsfonts}
\usepackage{mathtools}
\usepackage[pagewise]{lineno}
\usepackage{pifont} 
\usepackage[shortlabels]{enumitem}

\newcommand{\cuthere}{%
\noindent
\raisebox{-2.8pt}[0pt][0.95\baselineskip]{\ding{34}}
\unskip{\tiny\dotfill}
}
\newcolumntype{M}[1]{>{\centering\arraybackslash}m{#1}}

\newcommand{\ca}{{\mathcal A}}
\newcommand{\cb}{{\mathcal B}}
\newcommand{\cc}{{\mathcal C}}

\newcommand{\cf}{{\mathcal F}}
\newcommand{\cg}{{\mathcal G}}
\newcommand{\ch}{{\mathcal H}}
\newcommand{\ci}{{\mathcal I}}

\newcommand{\cn}{{\mathcal N}}
\newcommand{\cm}{{\mathcal M}}
\newcommand{\cmt}{{\boldsymbol{\mathcal M}}}
\newcommand{\cp}{{\mathcal P}}
\newcommand{\cq}{{\mathcal Q}}
\newcommand{\cR}{{\mathcal R}}

\newcommand{\Tau}{{\mathcal T}}
\newcommand{\ct}{{\mathcal T}}

\newcommand{\Bz}{{\mathrm B}_0}
\newcommand{\Bnz}{{\mathrm B}}

\newcommand{\E}{{\mathbb E}}

\newcommand{\M}{{\mathbb M}}
\newcommand{\N}{{\mathbb N}}

\renewcommand{\P}{{\mathbb P}}

\newcommand{\R}{{\mathbb R}}

\newcommand{\T}{{\mathbb T}}

\newcommand{\rN}{{\rm N}}

\newcommand{\rd}{{\rm d}}

\newcommand{\rh}{{\rm h}}

\newcommand{\bt}{{\mathbf t}}
\newcommand{\bv}{\mathbf v}
\newcommand{\bL}{\mathbf L}
\newcommand{\bw}{\mathbf w}
\newcommand{\bT}{{\mathbf T}}

\newcommand{\bTun}{{\mathbf T^\mathrm{unif}_n}}
\newcommand{\bTunn}{{\mathbf T^\mathrm{unif}_{n+1}}}

\newcommand{\fd}{{f_\mathrm{dens}}}
\newcommand{\fint}{f_{\mathrm{int}}}

\newcommand{\fT}{\mathfrak{T}}
\newcommand{\fTs}{\fT^{\mathrm{ske}}}

\newcommand{\ind}{{\bf 1}}

\renewcommand{\root}{{\varrho}}
\newcommand{\leaf}{{\mathrm{Lf}}}
\newcommand{\Br}{{\mathrm {Br}}}
\newcommand{\length}{{\mathscr L}}

\newcommand{\Span}{{\mathrm {Span}}}
\newcommand{\Split}{{\mathrm {Split}}}
\newcommand{\graft}{{\mathrm {Graft}}}
\newcommand{\tree}{{\mathrm {Tree}}}
\newcommand{\dgh}{{d_\mathrm {GH}}}
\newcommand{\dlgh}{{d_\mathrm {LGH}}}
\newcommand{\dlghz}{{d^{*}_\mathrm {LGH}}}
\newcommand{\dlghn}{{d^{(n)}_\mathrm {LGH}}}
\newcommand{\dlghzer}{{d^{(0)}_\mathrm {LGH}}}
\newcommand{\dlghnk}{{d^{(n+k)}_\mathrm {LGH}}}
\newcommand{\dlghk}{{d^{(k)}_\mathrm {LGH}}}
\newcommand{\dlghd}{{d^{[2]}_\mathrm {LGH}}}
\newcommand{\dghn}{{d^{(n)}_\mathrm{GH}}}
\newcommand{\dghnk}{{d^{(n+k)}_\mathrm{GH}}}
\newcommand{\dghk}{{d^{(k)}_\mathrm{GH}}}
\newcommand{\dghz}{{d^{(0)}_\mathrm{GH}}}

\newcommand{\dghd}{{d^{[2]}_\mathrm{GH}}}

\newcommand{\TL}{{\T_{\mathrm{loc-K}}}}
\newcommand{\TLd}{{\T^{[2]}_{\mathrm{loc-K}}}}
\newcommand{\TLz}{{\T^{*}_{\mathrm{loc-K}}}}
\newcommand{\TLb}{{\T^{0}_{\mathrm{loc-K}}}}
\newcommand{\TLzb}{{\T^{0,*}_{\mathrm{loc-K}}}}
\newcommand{\TLn}{{\T^{(n)}_{\mathrm{loc-K}}}}
\newcommand{\TLzer}{{\T^{(0)}_{\mathrm{loc-K}}}}
\newcommand{\TLnb}{{\T^{(n),0}_{\mathrm{loc-K}}}}
\newcommand{\TLunnb}{{\T^{(1),0}_{\mathrm{loc-K}}}}
\newcommand{\TLnn}{{\T^{(n+1)}_{\mathrm{loc-K}}}}

\newcommand{\TLnm}{{\T^{(n-1)}_{\mathrm{loc-K}}}}
\newcommand{\TLnk}{{\T^{(n+k)}_{\mathrm{loc-K}}}}
\newcommand{\TLnz}{{\T^{(n),*}_{\mathrm{loc-K}}}}
\newcommand{\TLk}{{\T^{(k)}_{\mathrm{loc-K}}}}
\newcommand{\TLkk}{{\T^{(k+1)}_{\mathrm{loc-K}}}}
\newcommand{\TLu}{{\T^{(1)}_{\mathrm{loc-K}}}}

\newcommand{\TLuzb}{{\T^{(1),0,*}_{\mathrm{loc-K}}}}
\newcommand{\TLnl}{{\T^{\mathrm{no\, leaf}}_{\mathrm{loc-K}}}}
\newcommand{\TLsb}{{\T^{\mathrm{spine},0}_{\mathrm{loc-K}}}}
\newcommand{\TLs}{{\T^{\mathrm{spine}}_{\mathrm{loc-K}}}}
\newcommand{\TLnzb}{{\T^{(n),0,*}_{\mathrm{loc-K}}}}

\newcommand{\TK}{{\T_{\mathrm{K}}}}
\newcommand{\TDn}{{\T^{(n)}_{\mathrm{dis}}}}
\newcommand{\TDk}{{\T^{(k)}_{\mathrm{dis}}}}
\newcommand{\TPn}{{\T^{(n)}_{\mathrm{plan}}}}
\newcommand{\TPk}{{\T^{(k)}_{\mathrm{plan}}}}
\newcommand{\TPu}{{\T^{(1)}_{\mathrm{plan}}}}

\newcommand{\TDnn}{{\T^{(2^n-1)}_{\mathrm{dis}}}}
\newcommand{\TDo}{{\T^{(n+1)}_{\mathrm{dis}}}}
\newcommand{\TDu}{{\T^{(1)}_{\mathrm{dis}}}}
\newcommand{\TDkk}{{\T^{(k+1)}_{\mathrm{dis}}}}
\newcommand{\TKd}{{\T^{[2]}_{\mathrm{K}}}}
\newcommand{\TKn}{{\T^{(n)}_{\mathrm{K}}}}

\newcommand{\TKk}{{\T^{(k)}_{\mathrm{K}}}}
\newcommand{\TKz}{{\T^{(0)}_{\mathrm{K}}}}
\newcommand{\TKu}{{\T^{(1)}_{\mathrm{K}}}}

\newcommand{\Tu}{\mathrm{T}_1}
\newcommand{\Tut}{\tilde{ \mathrm{T}}_1}
\newcommand{\Tz}{\mathrm{T}_0}

\newcommand{\ImM}{{\tilde \M(E)}}

\newcommand{\Card}{{\mathrm {Card}}\;}
\newcommand{\diam}{{\mathrm {diam}}\;}
\newcommand{\dist}{{\mathrm {dist}}\;}

\renewcommand{\Im}{{\mathrm {Im}}\,}

\newcommand{\val}[1]{\mathop{\left| #1 \right|}\nolimits}
\newcommand{\inv}[1]{\mathop{\frac{1}{ #1}}\nolimits}
\newcommand{\expp}[1]{\mathop {\mathrm{e}^{ #1}}}

\newcommand{\lb}{[\![}
\newcommand{\rb}{]\!]}

\theoremstyle{plain}
\newtheorem{theo}{Theorem}[section]
\newtheorem*{theo*}{Theorem}

\newtheorem{cor}[theo]{Corollary}
\newtheorem*{cor*}{Corollary}
\newtheorem{prop}[theo]{Proposition}
\newtheorem{lem}[theo]{Lemma}

\theoremstyle{remark}
\newtheorem{rem}[theo]{Remark}

\title{Brownian continuum random trees conditioned to be large}

\date{\today}

\author{Romain Abraham}
\address{Romain Abraham,
Institut Denis Poisson,
Universit\'{e} d'Orl\'{e}ans,
Universit\'e de Tours,
CNRS,
France}
\email{romain.abraham@univ-orleans.fr}

\author{Jean-Fran\c{c}ois Delmas}
\address{Jean-Fran\c{c}ois Delmas, CERMICS,
Ecole des Ponts,  France}
\email{delmas@cermics.enpc.fr}

\author{Hui He}
\address{Hui He,
Laboratory of Mathematics and Complex Systems, School of Mathematical Sciences, Beijing Normal University, P.R. China}
\email{hehui@bnu.edu.cn}
\thanks{The reseach of this work is supported by The National Key R\&D Program of China (No. 2020YFA0712900)}

\begin{document}
\begin{abstract}
  We  consider  a  Feller   diffusion  $(Z_s,s\ge  0)$  (with  diffusion
  coefficient  $\sqrt{2\beta}$   and  drift  $\theta\in  \R$)   that  we
  condition on  $\{Z_t=a_t\}$, where $a_t$ is  a deterministic function,
  and we study the limit in  distribution of the conditioned process and
  of  its genealogical  tree as  $t\to  +\infty$.  When  $a_t$ does  not
  increase too rapidly, we recover the standard size-biased process (and
  the associated  genealogical tree given  by the Kesten's  tree).  When
  $a_t$   behaves  as   $\alpha\beta^2  t^2$   when  $\theta=0$   or  as
  $\alpha \expp{2\beta|\theta| t}$  when $\theta\ne 0$, we  obtain a new
  diffusion,  as  already  proved  by  Overbeck  in  1994  in  the  case
  $\theta=0$. We  give a new  representation of this diffusion  using an
  elementary  SDE   with  a  Poisson  immigration.    The  corresponding
  genealogical tree is described by an infinite discrete skeleton (which
  does  not  satisfy the  branching  property)  decorated with  Brownian
  continuum random trees given by a Poisson point measure.

As a by-product of this study, we introduce several sets of trees endowed with a Gromov-type distance which are of independent interest and which allow here to define in a formal and measurable way the decoration of a backbone with a family of continuum random trees.
\end{abstract}



\maketitle

\section{Introduction}

\subsection{The discrete case motivation}

In  \cite{abd:llfgt},  for  the   geometric  reproduction  law,  and  in
\cite{ad:apegwt},  for  general  super-critical  reproduction  laws  with
finite mean and some special  sub-critical reproduction laws, the authors
consider  the limit  of a  Galton-Watson (GW)  process $(Z_n,  n\in \N)$
started at $Z_0=1$  conditionally on $Z_n=a_n$ as $n$  goes to infinity,
provided  the event  $\{Z_n=a_n\}$ has  positive probability.  They also
consider  more generally  the  local  limit of  the  GW  tree, which  in
particular  allows  to study  condensation  phenomenon  (on this  latter
subject, see  \cite{js:cnt,j:sgtcgwrac,ad:llcgwtcc}).  According  to the
different growth  rate of $a_n$  as $n$  goes to infinity,  they observe
different regimes for the limiting random tree: if $a_n=0$ for $n$ large,
the  limiting  tree  corresponds  to  the GW  tree  conditioned  on  the
extinction  event;  if  $a_n$  is strictly  positive  but  grows  slowly
(including the  case $a_n$  bounded), then  the limit  is the  so-called
Kesten tree,  which consists in an  infinite spine decorated with  independent GW
trees with  the initial reproduction law;  if $a_n$ grows at  a moderate
speed  (given  in   the  super-critical  case  of   finite  variance  by
$a_n\sim\alpha m^n$ with $\alpha>0$ and $m$ the mean of the reproduction
law), then the limit is a skeleton given by an immigration process decorated again
with independent  GW trees  with the  initial reproduction  law; if
$a_n$       grows      faster       than       $m^n$      (that       is
$\lim_{n\rightarrow  \infty }  m^{-n}  a_n=\infty $)  then 
results are known only for the geometric reproduction law (the limit exhibits
a condensation at the root, that is,  the root has an infinite number of
children, and then those children generate independent trees) and for
bounded reproduction  laws (the limit  is the regular $b$-ary tree,  with $b$
the possible maximum number of children).

We  mention that  the  local limit  distributions of  the  GW tree  with
geometric  reproduction also  appear when  considering the local limit
of trees  having $n$ 
vertices with a Gibbs distribution where the energy is the height of the
tree, see~\cite{meltem}.

\medskip

This work is a  first step to extend those results  to random real trees
called   L\'evy  trees   introduced   by  Duquesne   and   Le  Gall   in
\cite{DLG02,DLG05} which  are scaling  limits of (sub)critical  GW trees
and can be seen as genealogical trees for (sub)critical continuous state
branching  processes  (CSBP); see  also  \cite{ad:crtmp,  DW07} for  the
extension of this  latter representation to the  super-critical case. We
shall only  consider Feller diffusions,  which correspond to  CSBPs with
quadratic branching mechanism and whose genealogy can be described using
the   Brownian    continuum   random    tree   introduced    by   Aldous
\cite{a:crt3}. Our results belong also  to the family of works dedicated
to the description  of limits of conditioned random real  trees, in this
direction, see \cite{Li00, L07, D09, AD09}.

\subsection{Feller diffusion with Poisson immigration}
We consider a quadratic CSBP $Z=(Z_t,  t\geq 0)$  associated with the
branching mechanism:
\[
\psi_\theta(\lambda)=\beta\lambda^2+2\beta\theta\lambda,
\]
with  $\beta>0$ and
$\theta\in  \R$.
The process $Z$ is a solution to the 
stochastic differential equation (SDE):

\[
  \rd Z_t= \sqrt{2 \beta Z_t} \, \rd B_t -2 \beta \theta Z_t \rd t, \quad\text{ for $t\geq 0$},
\]
where $(B_t, t\geq  0)$ is some standard Brownian  motion.
The  CSBP is  sub-critical (resp.  super-critical) if  $\theta>0$ (resp.
$\theta<0$). The  time scaling parameter  $\beta$ will be fixed,  but we
shall stress in the notations  the size scaling parameter $\theta$, and
denote by $\P^\theta_x$ the distribution of $Z$ starting at $Z_0=x$. 

Let
$a=(a_t, t\geq  0)$ be  a non-negative function.  We shall  consider the
local limit  of the  process $Z$ conditionally  on $\{Z_t=a_t\}$  as $t$
goes  to  infinity,  that  is  the  possible  limiting  distribution  of
$Z_{[0,s]}=(Z_r, r\in [0, s])$, with  $s$ fixed, conditionally on $\{Z_t=a_t\}$ as
$t$ goes to infinity. We recall that this question is related to the description
  of the  Martin boundary  of Markov  processes and  extremal time-space
  harmonic  functions, see~\cite{d:78}. We have for $t\geq s\geq 0$ and
  $H_s$ a bounded  $\sigma(Z_{[0,s]})$-measurable random variable:
\[
  \E^\theta_x\left[H_s \, |\, Z_t =a_t\right]
  = \E^\theta_x\left[H_s  \, K(s, Z_s; t, a_t)\right],
\]
where $K$ is the so-called Martin kernel. Then, all the extremal time-space harmonic functions $h$ appear as the limit of:
\begin{equation}
  \label{eq:harmo}
  h(s,x)=\lim_{t\rightarrow +\infty } K(s,x; t, a_t) \quad \text{for
    all  $s,x\in \R_+$}
\end{equation}
for some non-negative function  $a=(a_t, t\geq 0)$. Overbeck~\cite{Ov94}
gives  all  the  extremal  time-space harmonic  functions  $h$  for  the
critical Feller diffusion (that is  $\theta=0$), and gives also the SDE
solved   by   the  Doob   $h$-transform   of   the  process   $Z$,   see
Lemma~\ref{lem:extremal}  for   the  extremal  harmonic   functions  and
Corollary~\ref{cor:Z=Ov}  for the  SDE  below for  $\theta\in \R$  which
includes  the sub-critical  and super-critical  cases.  For  keeping the
introduction as simple as possible, we  shall stick to the critical case
$\theta=0$  considered  in~\cite{Ov94}, and  choose
$\beta=1$ (the  general case  can be deduced  using a  deterministic time
change or a  Girsanov transformation of $Z$). In this  case, the extremal
harmonic functions $h$ are, with $\Bz $ defined in~\eqref{eq:def-B0}:
\begin{itemize}
\item\textbf{Extinction   case}    $a_t=0$ for $t$ large:    $h^\emptyset(s,x)=1$;
\item\textbf{Low regime} $a_t>0$ and
$a_t=o(t^2)$:  $h^0(s,x)=x$;
\item\textbf{Moderate                                           regime}
  $a_t\sim \alpha t^2$ with $\alpha\in (0, +\infty )$:
  $h^\alpha(s,x)=   \expp{-\alpha  s}\,   x  \,   \Bz (\alpha  x)$.
\end{itemize} 

For $\alpha\in [0, +\infty )$, the Doob $h$-transform of the process $Z$
using the harmonic function $h^\alpha$, denoted by
$Z^\alpha=(Z^\alpha_t, t\geq 0)$, satisfies
the following SDE according to~\cite[Theorem~3]{Ov94}:
\[
      \rd  Z_t^{\alpha}= \sqrt{2 Z_t^{\alpha}}\, \rd  B_t
    +2 g(Z_t^\alpha)\, \rd  t,\quad t\geq0, 
\]
where, with $\ch^\alpha (s,y)=\expp{-\alpha  s} \,  y\Bz (\alpha y)$, the
function    $g$  is equal to:
\[
  g(y)=y  \, \partial_y  \log( \ch^{\alpha} (\cdot,  y))= 1+ \alpha  y
  \frac{\Bz'(\alpha y)}{\Bz (\alpha  y)}\cdot
\]
 Motivated by  the backbone
decomposition of the  corresponding genealogical tree given  in the next
section, we provide a new representation of the process $Z^\alpha$ using
a Poisson immigration given in Corollary~\ref{cor:Z=Ov} which is stated
for the general case $\theta\in \R$. 
\begin{prop}[Representation using a Poisson immigration, case $\theta=0$]
  \label{propI:poisson}
  Let $\alpha\geq 0$ and $(S^{\alpha}_t,  t\geq 0)$ be a Poisson process
  with  intensity $\alpha  \rd t$,  independent of  the Brownian  motion
  $(B_t,  t\geq 0)$.  The process  $Z^\alpha$ starting  at $Z^\alpha_0=0$ is
  distributed as the solution $Y^\alpha=(Y^\alpha_t, t\geq 0)$ of:
\begin{equation}
    \label{eq:Y-intro}
      \rd Y^\alpha_t
      = \sqrt{2 Y^\alpha_t}\, \rd B_t + 2\, (S^{\alpha}_t+1)\, \rd t
      \quad\text{with}\quad
Y^\alpha_0=0    .
\end{equation} 
\end{prop}
When the process  $Z^\alpha$ starts at $Z^\alpha_0=x>0$, the constant 1 in the drift term of~\eqref{eq:Y-intro} must be replaced by a random constant independent of $B$ and $S^\alpha$, see the beginning of the proof of Corollary~\ref{cor:Z=Ov}
in Section~\ref{sec:proof-cor-Ov}.
The proof of  this result, given in  Section~\ref{sec:proof:Zaq}, uses a
result from  Rogers and  Pitman \cite{rp:mf} for  a transformation  of a
Markov process  to still be a Markov process.

\subsection{Decomposition of the Brownian CRT with respect to $n$ leaves taken at random at a given height}

We  denote   by  $\N^\theta$   the  canonical   $\sigma$-finite  measure
associated with the CSBP  $Z$ under $\P^\theta_\cdot$.  Intuitively, under
$\N^\theta$, the  population starts  with an infinitesimal  individual at
time $t=0$.  Let $\ct$ denote under $\N^\theta$ the genealogical tree of
the  process $Z$,  it is  the so-called  Brownian continuum  random tree
(CRT) introduced by Aldous. In this context, the random tree $\ct$ can be
easily built from a Brownian  excursion, and the measure $\N^\theta$ can
then be  identified with the  excursion measure of the reflected Brownian motion.
We write  $\root$ for the  root of $\ct$.  In~\cite[Theorem 4.5]{DLG05},
Duquesne  and  Le   Gall  give  a  decomposition  of   the  critical  or
sub-critical Brownian tree $\ct$ by taking a leaf uniformly at random at
level $t\geq  0$ and decorating  the branch from  the root to  this leaf
with independent  Brownian CRTs. There  is no difficulty to  extend this
result to  the supercritical  case, see  Corollary~\ref{cor:Thmbis}.  We
then  extend  this  representation  by giving  a  decomposition  of  the
Brownian CRT when taking $n$ leaves uniformly at random at level $t\geq 0$
and decorating the discrete tree spanned  by the $n$ leaves and the root
with independent  Brownian CRTs, see Theorem~\ref{Thmbis}.   This result
completes  the  description  of  \cite{DW07} where  one  chooses  theses
vertices at random without condition on their level.

Stating and  proving this result  relies on  a lengthy study  of various
spaces of trees  and the measurability of various maps  defined on those
sets  of trees,  which are  detailed in  Section~\ref{sec:topo-tree}. We
shall present informally the mathematical  objects and state the theorem
in the critical  case $\theta=0$ with $\beta=1$ for  simplicity; we also
write $\N$ for  $\N^0$.  Let $\Lambda_t$ denote the local  time at level
$t$ associated  with the Brownian tree  $\ct$ (its total mass  is equal to
$Z_t$ the size  of the population at level $t$):  this measure allows to
sample  random  individual  ``uniformly''  at  level  $t$;  the  measure
$\Lambda_t$ is  supported by the  leaves of  $\ct$ at level  $t$.  Under
$\N[\rd  \ct] \,  \Lambda_t^{\otimes n}(\rd  \bv^*)$ we  can sample  the
Brownian CRT $\ct$  with $n$ leaves $\bv^*=(v_1,  \ldots, v_n)\in \ct^n$
at level $t$.  To those $n$  distinguish vertices, we shall add the root
$\root$  of   $\ct$  and  set   $\bv=(\root,  \bv^*)$,  and   shall  see
$(\ct, \bv)$  as an element  of the Polish  metric space of  the locally
compact   of   $n+1$-pointed   trees,    $\TLn$,   equipped   with   the
local Gromov-Hausdorff distance  (and where all $n+1$-pointed  trees which are
isomorphic are identified),  see Section~\ref{sec:topo-tree} for precise
details.   We describe  the  rooted tree  spanned by  the  root and  the
distinguished  $\bv^*$ vertices  using a  combinatorial construction  on
growing  discrete  planar  trees  with   fixed  height  $t$  defined  in
Section~\ref{sec:bT} where  starting from one  branch of height  $t$, we
graft uniformly successively branches with their leaf at height $t$. Let
us stress that we use the planar  structure of the trees in this section
only, and that the  grafting of the new branch is  uniformly done on the
right  or on  the  left. After  $n$  such steps,  we  obtain the  random
$n+1$-pointed tree  $(\bTun, \bv_n)$,  where the  distinguished vertices
$\bv_n$ are first the root, and  then the leaves ranked in their arrival
order  (and  not   in  the  planar  order).   Then,  very informally,  on this  discrete tree,  for all
$i\in  I$ a countable set of indices,   we  graft  at   $x_i\in  \bTun$  a  subtree   $T_i$,  where
$\cm(\rd x,  \, \rd \ct)= \sum_{i\in  I} \delta_{(x_i, T_i)} (\rd  x, \,
\rd    \ct)$   is    a    Poisson   point    measure   with    intensity
$2  \, d\length(\rd  x)\, \N[\rd\ct]$,  where $d\length$  is the  length
measure on  $\bTun$.  This grafting  procedure is rigorously  defined in
Section~\ref{sec:discrete-app}  based  on  the technical  material  from
Section~\ref{sec:topo-tree}.    So   we   are    now   able   to   state
Theorem~\ref{Thmbis}    for    $\theta=0$   and    $\beta=1$.     Recall
$\bv=(\root, \bv^*)\in \ct^{n+1}$, with $\root$ the root of $\ct$.

 \begin{theo}[Generalized $n$-leaves decomposition, case $\theta=0$]
   \label{thmI:k-Bismut}
   Let  $t>0$ and $n\in \N^*$.
For every non-negative  measurable
  function $F$ defined on $\TLn$, we
  have:
\[
\N^{\theta}\left[\int_{\ct^{n}}\!\!\!\Lambda_t^{\otimes n}(\rd
   \bv^*)\, F (\ct, \bv)\right]
=n!\, t^{n-1} 
\, \E\left[F\Bigl(\graft_n\bigl((\bTun,\bv_n),\cm\bigr)\Bigr)\right].
\]
\end{theo}

Let   us  mention   here  that   there  have   been  several   works  on
skeletal/backbone decompositions  for (spatial) branching  processes and
their    corresponding    genealogical    trees,   for    example    see
\cite{AD09,BKM11prolific,CRY19LLN,DW07,              fekete2019skeletal,
  FFK2020skeletal,KR12backbone}   and   the  references   therein.    In
particular, in \cite{ fekete2019skeletal},  coupled systems of SDEs were
established    to   represent    the    skeletal   decompositions    for
continuous-state  branching processes  (conditioned on  survival), where
the   skeletons   are   determined  by   continuous-time   Galton-Watson
processes. And we  refer to \cite{DW07} for the reconstruction  of a L\'evy
tree from  a backbone tree,  which could be  formed by leaves  taken at
random  in a  Poissonnian  way from  the L\'evy  tree  according to  the
so-called  mass measure;  see Remark  5.4 there  and \cite{DLG02}.   For
representations of branching processes  (with immigrations) via SDEs, we
also refer to \cite{DaLi12} and references therein.

\subsection{Local limit of conditioned Brownian CRT}

The Brownian CRT $\ct$ gives the genealogical structure of the CSBP $Z$.
We shall  give a description of  the genealogical structure  of the
CSBP  associated  with the  Doob  $h$-transform  and prove  that it  appears
naturally as  local limit of  the Brownian  CRT $\ct$ conditioned  to be
large. We stress that the local limits obtained here are  different from the
one  obtained by  conditioning  on  the non  extinction  at large  time,
see~\cite{AD09} in this direction. We denote by $\TL=\TLzer$ the set (of
equivalence  classes)   of  locally  compact  rooted   real  trees,  see
Section~\ref{sec:topo-tree} for more details.

Recall that  in the critical case  $\theta=0$ the Brownian CRT  $\ct$ is
compact.  In the introduction, we simply denote by $\bt_t$ the real tree
$\bt$ truncated at  level $t$.  We denote by  $\cg_t$ the $\sigma$-field
generated by  $\ct_t$ for $t\geq  0$; in  particular the process  $Z$ is
adapted to  the filtration $(\cg_t, t\geq  0)$.  Let $F$ be  any bounded
continuous function defined on $\TL$.

\begin{itemize}
\item\textbf{Extinction case:} $a_t=0$ for $t$ large. We have:
  \[
    \lim_{t\rightarrow \infty } \N\bigl[F(\ct_s)\,
    \ind_{\{Z_t=a_t\}} \bigr]=
    \N \bigl[F(\ct_s)\bigr].
  \]
The  result  is  obvious  for  the    critical  case
  as   the   tree  $\ct$ is   compact
  $\N$-a.e., that  is $\ct_t=\ct$ for  $t$ large enough.
We obtain the same result in the sub-critical case $\theta>0$. 
  In the
  super-critical  case $\theta$, using  the
  Girsanov  transformation   from  \cite{ad:crtmp}  to   define  the
  super-critical  L\'evy  tree,   see  also~\eqref{eq:abs-cont2}, we get
  that:
  \[
    \lim_{t\rightarrow \infty } \N^\theta\bigl[F(\ct_s)\,
    \ind_{\{Z_t=a_t\}} \bigr]=
    \N^{|\theta|} \bigl[F(\ct_s)\bigr].
  \]   
  Those
  results hold also  in general  for any  compact L\'evy  trees.
\item\textbf{Low regime:}  $a$ is  positive and $a_t=o(t^2)$.  We recall
  that the Kesten  tree $\ct^*$ is informally  obtained by grafting
  the trees $(T_i,i\in  I)$ respectively at levels $(h_i,i\in  I)$ on an
  infinite       spine,       where        the       point       measure
  $\sum_{i\in I} \delta_{h_i, T_i}(\rd h, \rd  \ct) $ is a Poisson point
  measure  with  intensity  measure   $2\ind_{\{h>0\}}\rd  h  \,  \N[\rd
  \ct]$. See Lemma~\ref{lem:Kesten} for a  more formal definition of the
  Kesten tree.  The Kesten tree  appears already  as the local  limit of
  general compact L\'evy trees when conditioning instead by $\{Z_t>0\}$,
  see \cite{AD09}.  The next result  is Theorem~\ref{theo:main-result-2}
  restricted to the critical case $\theta=0$  with
  $\ct^{0,\theta}=\ct^*$ (notice  therein the difference of the limit
  between the super-critical case and the sub-critical one).
\begin{theo}
  \label{theoI:kesten}
  We have in the low regime,  $a_t=o(t^2)$ and $a_t>0$,  that:
   \[
    \lim_{t\rightarrow \infty } \N\bigl[F(\ct_s)\bigm|\, Z_t=a_t\bigr]=
    \E \left[F\left(\ct^*_s\right)\right].
  \]
  \end{theo}

\item\textbf{Moderate   regime:}  $   a_t\sim   \alpha   t^2  $,   where
  $\alpha\in  (0,  +\infty  )$.   We  first  consider  a  backbone  tree
  $\fT^{\alpha, 0}$ representing in  some sense the genealogy associated
  with    a    Poisson    immigration     with    rate    $\alpha$,    see
  Section~\ref{sec:main-result}   for   a  more   precise   description.
  Secondly,           let           the          point           measure
  $\sum_{i\in I} \delta_{x_i,  T_i}( \rd x, \rd  \ct)$ be, conditionally
  given $\fT^{\alpha, 0}$,  a Poisson point measure  with intensity rate
  $2 \length(\rd x) \, \N^{q}[\rd \ct]$ with $\length(\rd x)$ the length
  measure on $\fT^{\alpha, 0}$.  Then, the random tree $\ct^{\alpha, 0}$
  is obtained by grafting, for $i\in  I$, the tree $T_i$ at vertex $x_i$
  on  the   backbone  tree  $\fT^{\alpha,  0}$.   (As,  for  $\alpha=0$,
  $\fT^{0,0}$  can be  seen as  an infinite  spine, the  Kesten tree  is
  indeed                          distributed                         as
  $\ct^{0,  0}$.)   The  next  result  is  Theorem~\ref{theo:main-result}
  restricted to the critical case $\theta=0$.
 
\begin{theo}
  \label{theoI:gen-loc}
  We  have  in  the  moderate  regime, $  a_t\sim  \alpha  t^2  $  with
  $\alpha\in (0, +\infty )$, that:
   \[
    \lim_{t\rightarrow \infty } \N\bigl[F(\ct_s)\bigm|\, Z_t=a_t\bigr]=
    \E \left[F\left(\ct^{\alpha, 0}_s\right)\right].
  \]
\end{theo}
 Let
  us  stress  that  the  backbone  tree does  not  enjoy  the  branching
  property,  as already  observed  by \cite{abd:llfgt,  ad:apegwt} in  a
  discrete setting. In a forthcoming paper, we shall recover the branching structure in the backbone by considering a weighted tree.  

\item\textbf{High                                               regime:}
  $   \lim_{t\rightarrow   \infty   }   t^{-2}  a_t=+   \infty   $   (or
  $ \lim_{t\rightarrow \infty } \expp{-2  \beta |\theta| t} a_t=+ \infty
  $ if  $\theta\neq 0$). The description  of the possible limit  in this
  regime is still  an open question. As in the  discrete setting studied
  in  \cite{abd:llfgt},  one  could  ask  if  there  is  a  condensation
  phenomenon at  the root.  However,  to study such local  limits, which
  would not be locally compact (at least at or near the root), one would
  require a non trivial extension  of the current topology developed for
  locally compact trees.

\end{itemize}

\subsection{Outline of the paper}
Section~\ref{sec:not+csbp} is  devoted to some notations  and elementary
facts for the quadratic CSBP,  the transition kernel under the canonical
measure $\N^\theta$, and the Martin  boundary for the process $Z$ (under
$\P^\theta_x$ and the excursion measure  $\N^ \theta$).  We then present
families of martingales for the process $Z$ and then the local limits of
the  process $Z$  conditionally  on  $Z_t=a_t$ for  $t$  large and  some
deterministic  function  $(a_t,  t>  0)$ (under  $\P^\theta_x$  and  the
excursion  measure  $\N^  \theta$) in  Section~\ref{sec:loc-lim-Z}.   We
prove  Proposition~\ref{propI:poisson}  on  the  representation  of  the
Dood-$h$ transform of the process $Z$ with $h$ harmonic extremal using a
Poisson   immigration   in   the   general  case   $\theta\in   \R$   in
Section~\ref{sec:h-trans}, see Corollary~\ref{cor:Z=Ov}.  We provide the
backbone  decomposition  in Section~\ref{sec:backbone-intro},  with  the
  decomposition with respect to $n$ leaves   from     Theorem~\ref{thmI:k-Bismut}     in
Section~\ref{sec:backbone}  and   the  local   limit  of   Brownian  CRT
conditioned     to     $Z_t=a_t$    from     Theorems~\ref{theoI:kesten}
and~\ref{theoI:gen-loc} in Section~\ref{sec:main-result}.   We have made
the  choice   to  use  some   intuitive  (but  abusive)   definition  in
Section~\ref{sec:backbone-intro},   in    particular   considering   the
$\graft_n$ map in order to state the result without burdening the reader
with  too  much  technicalities;  we  clarify  all  the  definitions  in
Section~\ref{sec:appli}      using       the      lengthy      technical
Section~\ref{sec:topo-tree}.        We      also       believe      that
Section~\ref{sec:topo-tree}  helps  clarifying  some  previous  work  on
random trees  and could be  useful also  for further works  on grafting,
splitting, and decorating for trees.

Let us  mention that the  introduction is written with  the time scale 
parameter $\beta=1$  and for the  critical case $\theta=0$.  The general
cases $\beta>0$ and $\theta\in \R$ could  be deduced in finite time from
the  particular case  by  scaling or  using  Girsanov transformation  on
CSBP. However, if those computations  are not that complicated, they are
lengthy and treacherous;  so we  decided to  treat the
general cases but for the introduction.

An index of all the (numerous) relevant notations is provided at the end
of the document.

\section{General quadratic CSBP}
\label{sec:not+csbp}
\subsection{Notations}
  We set $\R_+=[0,+\infty)$, $\R_+^*=(0,\, +\infty)$,
$\N=\{0, 1,2,\cdots\}$ and $\N^*=\{1,2,\cdots\}$.

For $x\in \R$, we set $x_+=\max(0, x)$ and $x_-=\max(0, -x)$, so that $
x=x_+ - x_-$. We write $\delta_x$ for the Dirac mass at $x$.

Let $(E, d)$ be a metric  space. We denote by $\cm_+(E)$ the space
of non-negative measures on $E$ endowed with the vague topology.
For $\mu\in \cm_+(E)$ and $A$ a Borel subset of $E$, we denote by
$\mu_{|A}(\rd x)$ the measure $\ind_A (x) \mu(\rd x)$.
We write $\mu(f)=(f,\mu )=\int f \, \rd \mu =\langle f, \mu \rangle$
for the integral of the measurable real-valued function $f$ with respect
to the measure $\mu$, whenever it is meaningful.

We say that a function from a measurable space to a measurable space is
bi-measurable if it is measurable and the image of any measurable set is
a measurable set (when the function is one-to-one this is equivalent to
the function and its inverse being measurable).

\subsection{Quadratic CSBP}
\label{sec:2csbp}
Most results in this section can be found in
\cite{DLG02,jfd08,ad:crtmp,DW07}. 
Let $\beta>0$ be fixed. Let $\theta\in \R$. We consider the quadratic branching
mechanism $\psi_\theta$  given for $\lambda\in
\R$ by:
\begin{equation}
  \label{eq:def-psi}
\psi_\theta(\lambda)=\beta \lambda^2 + 2\beta\theta \lambda.
\end{equation}
The corresponding CSBP $Z=(Z_t, t\geq 0)$ is the unique strong solution to the following
stochastic differential equation (SDE):
\begin{equation}
   \label{eq:eds-Zq}
\rd Z_t= \sqrt{2 \beta Z_t} \, \rd B_t -2 \beta \theta Z_t \rd t \quad\text{ for $t\geq 0$},
\end{equation}
where $B=(B_t, t\geq 0)$ is a standard Brownian motion and $Z_0=x\geq 0$.
For $t\geq 0$, let $\cf_t$ be the $\sigma$-field
generated  by $(Z_s, t\in [0,
t])$. We write $\P^\theta_x$ to stress the value of the parameter $\theta$,
and the initial value of the process $Z$,  $Z_0=x$. We denote
by $\N^\theta$ the canonical measure of the process $Z$,
normalized in such a way that for $\lambda\geq 0$:
\[
  \N^\theta\left[1- \expp{-\lambda \sigma}\right]=\psi_\theta^{-1}(\lambda),
\]
where $\sigma=\int_0^\infty  Z_t\, \rd t$ is the total size of the
population under  the canonical
measure  $\N^\theta$ and $\psi_\theta^{-1}(\lambda)$ is the only solution $t$ to
$\psi_\theta(t)=\lambda$ such that $t\geq 2 \theta_- $. In particular, the process $(Z_t,t\ge 0)$ under $\P_x^\theta$ is distributed as the process $\left(\sum_{i\in I}Z_t^{(i)},t\ge 0\right)$ where $\sum_{i\in I}\delta_{Z^{(i)}}$ is a Poisson point measure with intensity $x\N^\theta(dZ)$. We refer to \cite{DLG02} for $\theta\geq 0$ (critical
and sub-critical case) and to
\cite{jfd08,ad:crtmp,DW07} for $\theta<0$ (super-critical case) for a detailed
presentation of the CSBP process $Z$ and the corresponding continuum
Brownian random tree $\ct$.

With a  slight abuse, we say  that two random  variables or
functionals $Y'$ and $Y''$ have  the same distribution under $\N^\theta$
if the  pushed forward  measures of  $\N^\theta$ by  $Y'$ and  $Y''$ are
equal.

\medskip

In order to recall the Laplace transform of $Z_t$, we introduce the
following positive functions $c^\theta$ and $\tilde c^\theta$ defined for $t\in
(0, +\infty )$ by:
\begin{equation}
   \label{eq:def-cqt}
  c^\theta_t=\frac{2\theta}{\expp{2\beta\theta t} -1}
  \quad\text{and}\quad
  \tilde c^\theta_t = \frac{2\theta}{1-\expp{-2\beta\theta t}},
\end{equation}
with  the convention  $c^0_t=\tilde c^0_t=1/{\beta t}$.
The functions   $c^\theta$ and
$\tilde c^\theta$ are decreasing with:
\begin{equation}
   \label{eq:lim-cq}
\lim_{t \rightarrow 0+}
c^\theta_t=\lim_{t \rightarrow 0+} \tilde c^\theta_t=+\infty,
\quad \lim_{t \rightarrow +\infty }
c^\theta_t=2 \theta_-
\quad\text{and}\quad
\lim_{t \rightarrow +\infty }
\tilde c^\theta_t=2 \theta_+ .
\end{equation}
We also have for $t>0$:
\begin{equation}
   \label{eq:tc=c}
  \tilde c^\theta_t =  c^{-\theta}_t= c^\theta_t+ 2\theta.
\end{equation}

\begin{rem}[Scaling property of $Z$]
  \label{rem:q=0}
  In  this   remark,  we   write  $Z^{[\beta,\theta]}$  for   $Z$  under
  $\N^\theta$  or $\P_x^\theta$  in order  to stress  the dependence  in
  $\beta>0$  and $\theta\in  \R$.  Let  $Y=(Y_s, s\geq  0)$ be  a Feller
  diffusion   that is, $Y=Z^{[1, 0]}$. Under $\P_x$, it is given as the
  unique strong solution to the SDE, with initial condition $Y_0=x$:
\begin{equation}
   \label{eq:SDE-Feller}
\rd  Y_s
  = \sqrt{2   Y_s}\,
  \rd  B_s,\quad\text{for $s\geq 0$}.
\end{equation}
We denote by $(Q_t, t\geq 0)$
  the semi-group of the diffusion $Y$ and recall that
it is a Feller 
semi-group   with   the   so-called  following   branching   property:
\begin{equation}
  \label{eq:branching-prop}
  Q_t(x+x', \cdot)=Q_t(x , \cdot) *  Q_t(x', \cdot)
  \quad\text{for all}\quad
t\geq 0 \quad\text{and}\quad
 x, x'\in  \R_+. 
\end{equation}
We shall denote by $\N$ the canonical  measure of $Y$.

For  $\beta>0$ and
  $\theta\in \R$, the process  $Z^{[\beta,
    \theta]}$ under $\N^\theta$ (resp. $\P_x^\theta$) is distributed as
  the process:
\begin{equation}
  \label{eq:Z=Y-bq}
  \left( \expp{-2 \beta \theta t}
  Y_{1/c_t^\theta}, \, t\geq 0\right)
\end{equation}
under $\N$ (resp. $\P_x$). 
Notice that  the range of $1/c^\theta_t$ as $t$ runs in
$\R_+$ is $[0, 1/(2\theta_-))$.
Even though, using this scaling and time change, it is (almost) enough to state
the forthcoming results for the particular case $\beta=1$ and $\theta=0$, we shall
keep general values for the parameters in order to better understand
their role. 
\end{rem}


We define for $t>0$ and $\lambda> - \tilde c^\theta_t$:
\begin{equation}
   \label{eq:def-u}
u^\theta(\lambda, t)
=\frac{\lambda c^\theta_t}{\tilde c^\theta_t +\lambda}
= c^\theta_t -
\frac{c^\theta_t \, \tilde c^\theta_t}{\tilde c^\theta_t +\lambda},
\end{equation}
and set  $u^\theta(\lambda, 0)=\lambda$ for $t=0$. This gives that for
$t>0$ and $\lambda> - \tilde c^\theta_t$:
\[
u^\theta(\lambda, t) =
\begin{cases}
\displaystyle \frac{2\theta \lambda} { (2\theta +\lambda)
    \expp{2\beta\theta t} - \lambda},  & \text{if } \theta\ne 0,\\
    \lambda/(1+ \lambda \beta t) , & \text{if }\theta=0.
    \end{cases}
  \]
  For $r>0$ and $t\geq 0$, we have that:
  \[
    u^\theta (c^\theta_r, t)=c^\theta_{t+r}.
  \]
We recall from the above mentioned references (\cite{DLG02,jfd08,ad:crtmp,DW07}) for $\lambda\geq 0$ and
by  analytic continuation for $\lambda<0$, that  for  $t> 0$ and $x\geq 0$:
\begin{equation}
  \label{eq:NeZ}
  \N^\theta \left[1- \expp{-\lambda Z_t}\right]=u^\theta(\lambda, t)
  \quad\text{and}\quad
  \E^\theta_x \left[\expp{-\lambda Z_t}\right]=\expp{-xu^\theta(\lambda,
    t)}
  \quad\text{for all $\lambda > - \tilde c_t^\theta$.}
\end{equation}

\medskip

We denote by $\zeta=\inf\{t>0; \, Z_t=0\}$ the lifetime of the process
$Z$. We recall that for all $t>0$:
\[
  \N^\theta[\zeta>t]=\lim_{\lambda\rightarrow \infty } u^\theta(\lambda,
  t)=c^\theta(t).
\]
By considering
the series in $\lambda$ in \eqref{eq:def-u} and \eqref{eq:NeZ}, we deduce that for all $t> 0$ and $n\in \N^*$:
\begin{equation}
  \label{eq:N[Zn]}
  \N^\theta \left[\left(\tilde c_t^\theta Z_t\right)^n\right]= n! c_t^\theta.
\end{equation}

We now give a martingales related to the CSBP $Z$.
Since $u^\theta(\lambda, t)=u^{-\theta}(\lambda+ 2\theta,t) -2\theta$
for $\lambda \geq  c^\theta_t -2\theta$, 
we deduce that 
 the   process
  $\left(\expp{2\theta   Z_t},  \,  t\in  I\right)$  is  a
  martingale     with   respect   to  the   filtration
  $(\cf_t, t\geq 0)$ under  $\N^\theta$  with   $I=\R_+^*$  and  under
  $\P^\theta_x$   with  $I=\R_+$. Furthermore, according to \cite[Section
4]{ad:crtmp}, we have that for $\theta\in \R$, $t>0$ and $x\geq 0$:
\begin{equation}
   \label{eq:abs-cont}
   \N^{-\theta}[\rd Z] _{|\cf_t}=\expp{2\theta Z_t}  \N^{\theta}[\rd Z]
   _{|\cf_t}
   \quad\text{and}\quad
      \E_x^{-\theta}[\rd Z] _{|\cf_t}=\expp{2\theta (Z_t-x)} \,  \E_x^{\theta}[\rd Z]
   _{|\cf_t}.
\end{equation}  
Recall that $\tilde
c^\theta_t$ is decreasing in $t$, and thus $-\tilde c^\theta_{t+r}> -
\tilde c^\theta_r$.
The next lemma is an easy consequence of
 \eqref{eq:NeZ} and the following elementary equality:
\[
  u(- \tilde c_{t+r}^\theta,t)= - \tilde c_r^\theta \quad\text{for all
    $t\geq 0$ and $r>0$.}
\]

\begin{lem}
  \label{lem:}
  Let  $\theta\in  \R$,  $x\in  \R_+$,  $r>0$  and  the  quadratic  CSBP
  $(Z_t,t\geq   0)$   solution  of~\eqref{eq:eds-Zq}.    The   process
  $\left(\expp{\tilde  c^\theta_{t+r}  Z_t},  \,  t\in  I\right)$  is  a
  martingale     with   respect   to  the   filtration
  $(\cf_t, t\geq 0)$ under  $\N^\theta$  with   $I=\R_+^*$  and  under
  $\P^\theta_x$   with  $I=\R_+$.
\end{lem}

\subsection{Transition densities and Martin Kernel}

We first provide  the densities of the
entrance law $q_t^\theta(\rd x)$ and the  transition kernel $q_t^\theta(x,\rd
y)$ of the CSBP $Z$ under its excursion measure, where for $t,s>0$, $x>0$
and $y\geq 0$:
\[
  q_t^\theta(\rd x)
  =\rd \N^\theta[Z_t=x, \, \zeta>t]
  \quad\text{and}\quad
  q_{t,s}^\theta(x,\rd y)
=\rd \N^\theta[Z_{t+s}=y|\, Z_s=x].
\]
We shall consider the function $\Bz $ and $\Bnz$ on $ \R_+$ defined by:
\begin{equation}
   \label{eq:def-B0}
  \Bz (x)= \sum_{k\in \N} \frac{x^{k}}{k! (k+1)!}
  \quad\text{and}\quad
\Bnz(x)=x\Bz (x)= \sqrt{x}\,  I_1\big(2  \sqrt{x}\big),  
\end{equation}
where  $I_1(x)=\sum_{i\in \N} (x/2)^{2i+1}/i!(i+1)!$ is the Bessel
function. Notice that $\Bz (0)=1$.

\begin{lem}[Entrance law and transition densities of $Z$]
  \label{lem:density}
  Let $\theta\in \R$. Let $t,s>0$, $x>0$
  and $y\geq 0$. We have $q_{t,s}^\theta(0, \rd y) = \delta_0(\rd y)$
  and:
  \[
     q_t^\theta(\rd x)=q_t^\theta(x) \, \rd
     x \quad\text{and}\quad
     q_{t,s}^\theta(x, \rd y) = \expp{-x c^\theta_t} \delta_0(\rd y) +
     q_t^\theta(x, y)\, \rd y,
\]
where:
\begin{align}
      \label{eq:entrance-law}
  q_t^\theta(x)
  &= c_t^\theta \tilde c^\theta_t \expp{- \tilde c^\theta_t x },\\
   \label{eq:density}
  q_t^\theta(x, y)
  &= xc^\theta_t \tilde c ^\theta_t \,  \expp{ - (x+y) c^\theta_t -2\theta y}
    \Bz \left(xy c^\theta_t \tilde c ^\theta_t    \right).
\end{align}
\end{lem}

\begin{proof}
  We  omit  the parameter $\theta$ in the proof.
   On one hand, from the definition of $q_t(\rd x)$, we get that for $\lambda\geq 0$:
\[
\int_0^{+\infty } \expp{-\lambda x} q_t(\rd x)
= \N\left[\expp{-\lambda Z_t}\ind_{\{\zeta>t\}}\right]
= - \N\left[1- \expp{-\lambda Z_t}\right] + \N\left[\zeta>t\right]
= c(t) -u(\lambda, t).
\]
On the other hand, using \eqref{eq:def-u}, we get:
\[
  \int_0^\infty  c _t \tilde c_t \expp{-(\tilde c_t
    + \lambda)x}\rd x
  =c(t) -u(\lambda, t).
\]
Then use that finite positive measures on $\R_+$ are characterized by
their Laplace transform to obtain that $   q_t(\rd x)=q_t(x) \, \rd
x$ with $q_t$ given by  \eqref{eq:entrance-law}.

From the  definition of $q_t(x,\rd
y)$, we get that for $\lambda\geq 0$:
\[
\int_0^{+\infty } \expp{-\lambda y} q_t(x, \rd y)
= \N\left[\expp{-\lambda Z_{t+s}}\Bigm|\, Z_s=x\right]
=\expp{-x u(\lambda, t)}= \expp{- \frac{a}{b}  + \frac{a}{b+ \lambda}},
\]
where, thanks to~\eqref{eq:def-u},
$a= x c _t \tilde c_t$ and $
b=\tilde c_t$.
Notice that:
\[
\expp{ \frac{a}{b+ \lambda}}
= 1 + \sum_{k\in \N} \inv{(k+1)!} \left(\frac{a}{b+ \lambda}\right)^{k+1}
= 1+ a\sum_{k\in \N} \int_0^{+\infty }
\frac{(ay)^{k}}{k! (k+1)! } \expp{-by -\lambda y}  \rd y.
\]
Using \eqref{eq:def-cqt}, we deduce  that $q_t(x, \rd y) = \expp{-x c_t} \delta_0(\rd y) +
q_t(x, y)\, \rd y$, with $q_t(x, y)$ given by \eqref{eq:density}.
\end{proof}

Let us notice that $  q_{t,s}^\theta(x,\rd y)$ is also the transition
kernel of the CSBP $Z$ under  $\P^\theta_{x_0}$ for every $x_0\geq 0$.  
The Martin kernel is  defined for $t\geq s\geq 0$ and $x,y\in \R_+$ by:
\begin{equation}
   \label{eq:Martin-kernel}
  K(s,x;t,y)=\frac{ q_{t-s,s}^\theta(x,\rd y)}{ q_{t, 0}^\theta(1,\rd
    y)}\cdot
\end{equation}

\subsection{Martin boundary}
According to Overbeck~\cite{Ov94}, see also~\cite[Section~10]{d:78}, all
extremal (non-negative)  time-space harmonic functions   for the CSBP
$Z$ appear  as the limit of the Martin kernel   $K(s,x;t,a_t)$,
see~\eqref{eq:Martin-kernel}, 
as $t$ goes  to infinity and
$(a_t, t\geq 0)$ is a non-negative function.
To  study the possible
limits as $t$ goes to infinity of:
\begin{equation}
  \label{eq:lim-K}
\lim_{t\rightarrow \infty }
 K(s,x;t,a_t),
\end{equation}
we
shall consider the functions on $\R^2_+$:
\begin{equation}
  \label{eq:h0a-2}
   H^{0, \theta}(s,x)=x \,  \expp{2 \beta \theta s}
  \quad\text{and}\quad
  H^{\alpha,\theta} (s,x)
= \expp{-\alpha/c_s^{\theta}}
\, \frac{\Bnz \left(\alpha x \expp{2 \beta \theta s}\right)}{\alpha},
\end{equation}
for  $\alpha>0$  and with  $\Bnz$   defined  in~\eqref{eq:def-B0}.  Notice  that
$\lim_{\alpha\rightarrow  0}  H^{  \alpha, \theta}=H^{0,  \theta}$.

We consider the following intermediary result.

\begin{lem}
  \label{lem:cv-q/q}
  Let $s\geq 0$ and  $x\geq 0$. If  $(a_t, t\geq 0)$ is positive and
  $\lim_{t\rightarrow +\infty } a_t \, c^\theta _t  \tilde
  c^\theta_t=\alpha\in [0, +\infty )$, then we have:
 \begin{equation}
   \label{eq:lim-q/q}
    \lim_{t\rightarrow \infty }
    \frac{q^\theta_{t-s}(x,a_t)}{q^\theta_t (a_t)} = \expp{-2\theta_- x}
    \,
    H^{\alpha, |\theta|} (s,x).  
  \end{equation} 
\end{lem}

\begin{proof}
  We   omit  the   super-script   $\theta$  in   the   proof.   We   get
  from~\eqref{eq:density} and~\eqref{eq:tc=c}  that    for
  $t\geq s>0$ and $y>0$:
\[
  \frac{q^\theta_{t-s}(x,y)}{q^\theta_t (y)}=
   x\expp{-x c_{t-s}}
  \, \expp{ -y (c_{t-s}  - c_t)}\, \frac{c_{t-s} \tilde c_{t-s}}{c_t
    \tilde c_t} \, \Bz \left(xy c_{t-s} 
    \tilde c _{t-s}    \right) .
\]
Recall from~\eqref{eq:lim-cq} that $\lim_{t \rightarrow +\infty }
c_t=2 \theta_-$ so that:
\begin{equation}
   \label{eq:lim-h0}
  \lim_{t\rightarrow \infty } \expp{-xc_{t-s} }= \expp{-2\theta_- x}.
\end{equation}
 It is elementary to check that:
\[
  \lim_{t\rightarrow \infty } c_{t-s}\tilde{c}_{t-s}/c_{t}\tilde{c}_{t}
  =\expp{2\beta|\theta|s}
  \quad\text{and}\quad
  \lim_{t\rightarrow \infty } (c_{t-s}- c_{t})/c_t \tilde c_t
 =  \frac{\expp{2\beta|\theta|s} -1}{2|\theta|}=\inv{c^{|\theta|}_s},
\]
where the latter limit is simply  $\beta s$ if $\theta=0$. The result is
then immediate.
\end{proof}

The result below for $\theta=0$ appears  in~\cite[Section~5]{Ov94}, and the
proof for general $\theta$ is similar.
\begin{lem}[Martin boundary]
  \label{lem:at=0+K}
  Let $s\geq 0$ and $x\geq 0$. The limit~\eqref{eq:lim-K}
  exists only in the following cases:
 \begin{enumerate}[(i)]
   \item Extinction case. If  $a_t=0$ for $t$
     large enough, then the limit~\eqref{eq:lim-K} exists and is equal
     to:
\[ 
  h^{\emptyset,\theta} (s,x)=  h^ {\emptyset,\theta}(x)=\expp{-2
    \theta_-(x-1)}.  
\] 
\item \label{item:low-mod}
  Low and moderate regimes. If the sequence $(a_t, t\geq 0)$ is positive and
  $\lim_{t\rightarrow +\infty } a_t \, c^\theta _t  \tilde
  c^\theta_t=\alpha\in [0, +\infty )$, then the limit~\eqref{eq:lim-K} exists and is equal
     to:
\[
  h^{ \alpha,\theta} (s,x)=  h^{ \emptyset,\theta} (x)\,\,
  \frac{H^{\alpha,|\theta|}(s,x)}{H^{\alpha, |\theta|}(0,
    1)} \cdot
\]

\item High  regime. If 
  $\lim_{t\rightarrow +\infty } a_t \, c^\theta _t  \tilde
  c^\theta_t= +\infty$, then the limit~\eqref{eq:lim-K} exists and is equal
  to:
  \[
  h^\infty (s,x)=  \ind_{\{s=0, \,x=1\}} .
  \]
 \end{enumerate}
\end{lem}

\begin{rem}[Equivalent condition for the moderate regime]
  \label{rem:equiv-cond}
 The   moderate regime condition
  $\lim_{t\rightarrow   +\infty   }   a_t    \,   c^\theta   _t   \tilde
  c^\theta_t=\alpha\in (0, +\infty )$, which appears in
  Lemma~\ref{lem:at=0+K}~\ref{item:low-mod},   is in fact equivalent to:
  \begin{equation}
    \label{eq:equiv-cond}
       a_t\sim \begin{cases}
      \alpha \beta^2 t^2 & \text{if
        $\theta= 0$}, \\
        \alpha (2\theta)^{-2} \expp{2 \beta |\theta| t} & \text{if
        $\theta\neq 0$}.
    \end{cases}
  \end{equation}
\end{rem}

\begin{proof}
  We  omit  the super-script  $\theta$  in  the  proof.
The low and moderate regimes are a direct consequence of
Lemma~\ref{lem:cv-q/q}. In the extinction case, use that $K(s,x;
t,0)=\expp{-xc_{t-s} + c_t}$ and~\eqref{eq:lim-h0} to get the result.  
For the high  regime, using that for $y>0$: 
\[
  K(s,x;t,y)=\expp{-xc_{t-s} + c_t}\,\times\, 
  \begin{cases}
  \displaystyle 0
  & \text{if $x=0$,}  \\
 \displaystyle \frac{\Bnz(xy\, c _{t-s} \tilde c_{t-s})}
  {\Bnz( y \, c _{t} \tilde c_{t})}\,  \expp{- y(c_{t-s} -
    c_t) }
  & \text{if $x>0$,}\\
\end{cases}
\]
Equation~\eqref{eq:def-B0}, the asymptotics of
the Bessel function $I_1(z)\sim \expp{z}/\sqrt{2\pi  z}$ as $z $ goes to
infinity and~\eqref{eq:lim-h0}, we deduce that
$ \lim_{t\rightarrow \infty }
 K(s,x;t,a_t)= \ind_{\{s=0,\, x=1\}}$.
\end{proof}

Using similar arguments as in~\cite[Section~5]{Ov94} stated for
$\theta=0$, Girsanov transform \eqref{eq:abs-cont} to reduce the cases $\theta<0$ to
$\theta>0$ and then Remark~\ref{rem:q=0} to reduce those latter  cases  to the
case $\theta=0$, we get the following result.

 \begin{lem}[Extremal harmonic functions]
   \label{lem:extremal}
   Let $\beta>0$ be fixed. Let  $\theta\in \R$.  The extremal time-space
   harmonic  functions  of  $Z^{[\beta,   \theta]}$  are  the  functions
   $h^{\alpha,\theta}$         for          $\alpha\in         \ca$         where
   $\ca= \{\emptyset\}\bigcup [0, +\infty )$.
 \end{lem}

 \begin{proof}
   Notice  that $h^\infty  $  is not  an  harmonic function.   According
   to~\cite[Section~10]{d:78},  the  functions $h^{\alpha,\theta}$  with
   $\alpha\in \ca$  are the  only possible extremal  harmonic functions.
   Thanks  to~\eqref{eq:abs-cont}  it is  enough  to  consider the  case
   $\theta\geq 0$.  Thanks to  Remark~\ref{rem:q=0}, for $\theta\geq 0$,
   we have that:
 \begin{align*}
     \E^\theta\left[ F(Z^{[\beta, \theta]}_{[0,t]})\,  h^{\alpha,\theta} (t,
       Z^{[\beta, \theta]}_t) \right]
&     =  \E\left[ F(\expp{-2 \beta \theta s} Y_{1/c^\theta_s} , s\in
       [0,t])\,  h^{\alpha, \theta} (t, \expp{-2 \beta \theta t} Y_{1/c^\theta_t} 
) \right] \\
&     = C_\alpha\, \E\left[ F(\expp{-2 \beta \theta s} Y_{1/c^\theta_s} , s\in
       [0,t])\,  \expp{- \alpha/c^\theta_t} \, \Bnz(\alpha  Y_{1/c^\theta_t} 
) \right],
 \end{align*}
 where for $\alpha=0$, $C_\alpha \Bnz(\alpha x)$ is simply replaced by
 $C_0\, x$, and $C_\alpha$ is a finite positive constant. 
Then      use     that      $1$,      $(Y_s,      s\geq     0)$      and
$(\expp{-\alpha     s}     \Bnz(\alpha      Y_s),     s\geq     0)$     for
$\alpha\in  (0, +\infty  )$  are extremal  martingales, see~\cite{Ov94},  to conclude when
$\theta\geq 0$.
 \end{proof}

\section{Local limits for the process $Z$}
\label{sec:loc-lim-Z}
\subsection{Some martingales}
We present in this section two martingales which will naturally appear
in the local limits for the Brownian continuum random tree (CRT). Let       $\alpha\in\R$. Define:
\begin{align}
\label{eqn:defhalpha}
  \ch^{\alpha}(s, y)=  \expp{-\alpha s} \, y \Bz (\alpha y)
    ,\quad s\geq 0,\,  y\geq 0,
\end{align}
where $\Bz $ is defined in~\eqref{eq:def-B0}.  Recall $\theta\in \R$.
Let 
$M^{\alpha, \theta}=(M^{\alpha, \theta}_t, t>     0)$ be   the
    process defined  by:
\begin{equation}
  \label{Mmartingale}
M^{\alpha, \theta}_t
 = \ch^{\alpha}(1/c_t^\theta,\, \expp{2\beta \theta t}\, Z_t).
\end{equation}

For $\theta=0$ and $\alpha\in [0, +\infty )$, this formula corresponds to \cite[Eq.~(19)]{Ov94} up to a normalizing 
constant.

Using Theorem 3 of \cite{Ov94} and Remark~\ref{rem:q=0}, we
 get the following result.

\begin{prop}\label{Prop:Mmartingale}
  Let   $\alpha,    \theta\in   \R$,    $x\in   \R_+$.     The   process
  $\left(M^{\alpha, \theta}_t,  \, t\in I\right)$ is  a non-negative    martingale under
  $\N^\theta$  with  $I=(0,  +\infty  )$ and  under  $\P^\theta_x$  with
  $I=\R_+$.
\end{prop}

\begin{proof}
  We first consider the case  $\theta=0$. For $\alpha\in [0, +\infty )$,
  under  $\P_x$   this  result  is  in   \cite[Section~5]{Ov94}.   Using
  $|M_t^{\alpha,   \theta}|\leq   \expp{ 2|\alpha| /c_t^\theta}\, M_t^{|\alpha|,   \theta}$   to   get   the
  integrability  for $\theta=0$ and  that  $\Bz $  is analytic  to  get  the  martingale
  property, we deduce that the result holds also for $\alpha\leq 0$
  under $\P_x$.
For $\theta\neq 0$, use   Remark~\ref{rem:q=0} to get the result under
$\P_x$ for all $\alpha,\theta\in \R$. 

\medskip

Moreover,  for  all $t>0$, we have, using \eqref{eq:N[Zn]} and $\tilde
c^\theta_t/c^\theta_t=\expp{2\beta \theta t}$, and thus  that for
$\alpha, \theta\in \R$:
\[
\N^\theta \left[ M^{\alpha, \theta}_t\right]=1.
\]
Then use the Markov property under the excursion measure to conclude the
result also holds under $\N^\theta$. 
\end{proof}


We introduce an other family of related martingales. For $\alpha, \theta\in \R$, we set $\tilde M^{\alpha,  \theta}=(\tilde
M^{\alpha,  \theta}_t, t>0)$ with:
\begin{equation}
  \label{eq:M=tM}
  \tilde M^{\alpha,  \theta}_t= \expp{2\theta Z_t}\, M^{\alpha,  -
    \theta}_t
  =\ch^{\alpha}(1/\tilde{c}_t^\theta,\, \expp{-2\beta \theta t}\, Z_t)\,   \expp{2\theta Z_t} ,
\end{equation}
using  \eqref{Mmartingale} and $c_t^{-\theta}=\tilde c_t^\theta$ for the
second equality. 
We then deduce from Proposition \ref{Prop:Mmartingale} the following
corollary.


\begin{cor}
  \label{cor:RM-mart}
  Let $\theta, \alpha\in \R$. The process $\tilde M^{\alpha,
    \theta}$ is a martingale under
  $\N^\theta$, and for  $t>0$ and  any non-negative
  $\cf_t$-measurable random variable $H_t$, we have:
\begin{equation}
   \label{eq:NM=NtM}
  \N^\theta[H_t \, \tilde M^{\alpha, \theta}_t]  =  \N^{-\theta}
  \left[H_t \, M^{\alpha, -\theta}_t\right].
\end{equation}
\end{cor}

\begin{rem}[The case $\theta=0$ and $\alpha=0$]
   \label{rem:M00}
Let $t>0$. For  $\theta=0$, we have:
\[
  \tilde M^{\alpha, 0}_t=
M^{\alpha, 0}_t=  \ch^{\alpha}(\beta t,\, Z_t)\cdot
\]
For  $\alpha=0$, we have:
\begin{equation}
   \label{eq:M-a=0}
  M^{0, \theta}_t= Z_t \expp{2\beta \theta t}
  \quad\text{and}\quad
  \tilde     M^{0, \theta}_t
    = Z_t \expp{2 \theta (Z_t -\beta t)}.
\end{equation}
Then for $\alpha=\theta=0$, we have:
\[
  \tilde M^{0, 0}_t=  M^{0, 0}_t= Z_t.
\]
\end{rem}

\subsection{Local limit}
  We  first consider  the
Poisson regime, whose name is inherited from the representation given in
Proposition \ref{prop:Zaq} based on a Poisson immigration.
Let  $a=(a_t, t>0)$  be  a  positive function.
\begin{prop}[Poisson regime]
  \label{prop:cv-At-P}
  Let $\theta\in \R$, $s>0$ and  $H_s$
  be a bounded non-negative $\cf_s$-measurable random variable.
  Let $\alpha\in (0, +\infty )$. Assume the function $a$ is such that
  as $t\rightarrow\infty$ large:
\[
    a_t\sim \begin{cases}
      \alpha \beta^2 t^2 & \text{if
        $\theta= 0$}, \\
        \alpha (2\theta)^{-2} \expp{2 \beta |\theta| t} & \text{if
        $\theta\neq 0$}.
    \end{cases}
\]
    Then we have:
\begin{equation}
   \label{eq:limA}
    \lim_{t\rightarrow \infty } \N^\theta[H_s |\,
    Z_{t}=a_{t}]  =\N^{|\theta|} \left[H_s \, M^{\alpha, |\theta|}_s\right]=
    \begin{cases}
   \N^\theta[H_s \, M^{\alpha, \theta}_s]   & \text{if
     $\theta\geq 0$}, \\
   \N^\theta[H_s \, \tilde M^{\alpha, \theta}_s]   & \text{if
     $\theta\leq 0$}.
 \end{cases}
\end{equation}
\end{prop}

 \begin{proof}
   Let
   $s>0$ and $H_s$ be fixed. For $t>0$,  thanks
 to \eqref{eq:entrance-law} and \eqref{eq:density}, we have:
 \[
\N^\theta[H_s |\, Z_{t+s}=a_{t+s}]
 =\frac{\N\left[H_sq_t(Z_s, a_{t+s})\right]}{q_{t+s}(a_{t+s})}\cdot
\]
Then use Lemma~\ref{lem:cv-q/q} to get the existence of the limit of $q_t(Z_s,
a_{t+s})/q_{t+s}(a_{t+s})$ and  Proposition~\ref{Prop:Mmartingale} to
get the convergence in $L^1$. To conclude, notice that
Lemma~\ref{lem:cv-q/q}  gives:
\[
  \lim_{t\rightarrow +\infty } \frac{ q_t(Z_s,
    a_{t+s})}{q_{t+s}(a_{t+s})}
  =
    \begin{cases}
    M^{\alpha, \theta}_s   & \text{if
     $\theta\geq 0$}, \\
    \tilde M^{\alpha, \theta}_s   & \text{if
     $\theta\leq 0$}.
 \end{cases}
\]

\end{proof}

The same proof can be used for  the Kesten regime.
\begin{prop}[Kesten regime]
  \label{prop:cv-At-K}
  Let $\theta\in \R$, $s>0$ and  $H_s$
  be a bounded non-negative $\cf_s$-measurable random variable.
  Assume the function $a$ is positive  ($a_t>0$) and such that
  as $t\rightarrow\infty$:
\[
    a_t=  \begin{cases}
      o(t^2) & \text{if
        $\theta= 0$}, \\
         o( \expp{2 \beta |\theta| t}) & \text{if
        $\theta\neq 0$}.
    \end{cases}
\]
   Then we have:
\begin{equation}
   \label{eq:limA-K}
    \lim_{t\rightarrow \infty } \N^\theta[H_s |\, Z_{t}=a_{t}]
    =
    \N^{|\theta|} \left[H_s \, Z_s \expp{2\beta |\theta| s} \right]
    =
    \begin{cases}
   \N^\theta[H_s \, M_s^{0, \theta}  ]    & \text{if
     $\theta\geq 0$},\\
   \N^\theta[H_s \,  \tilde M_s^{0, \theta}]    & \text{if
     $\theta\leq 0$}.
 \end{cases}
\end{equation}
\end{prop}

For completeness, we add the well known extinction case, that is the function $a_t=0$ for large $t$, which is a
direct consequence of \eqref{eq:abs-cont}. Since the event $\{Z_t=0\}$
has infinite measure under $\N^\theta$, we  consider the restriction
instead of the conditioning. 
\begin{prop}[Extinction  regime]
  \label{prop:cv-At-E}
  Let $\theta\in \R$, $s>0$ and  $H_s$
  be a bounded non-negative $\cf_s$-measurable random variable.
   Then we have:
\begin{equation}
   \label{eq:limA-E}
    \lim_{t\rightarrow \infty } \N^\theta \left[H_s \, \ind_{\{ Z_{t}=0\}}\right]
    = \N^{|\theta|}[H_s ]
    =
    \begin{cases}
   \N^\theta[H_s]     & \text{if
     $\theta\geq 0$},\\
   \N^{-\theta}[H_s]    & \text{if
     $\theta\leq 0$}.
 \end{cases}
\end{equation}
\end{prop}

\section{$h$-transform}
\label{sec:h-trans}
We give  a representation of the  distribution of the process  $Z$ under
the  $h$-transform given  by  the martingale  $M^{\alpha, \theta}$ using
a Poisson immigration; and we identify it with the solution of the SDE
from \cite[Theorem~3]{Ov94}. Even if Proposition~\ref{prop:Zaq} and
Corollary~\ref{cor:Z=Ov} below are a direct consequence of
Proposition~\ref{prop:Taq=TMaq} and Theorem~\ref{theo:main-result}
(see Remarks~\ref{rem:tree-Z} and~\ref{rem:over}), we provide an
independent proof which is interesting by itself.   The
proof  will  be done  for  $\beta=1$  and  $\theta=0$,  and then  use  a
time-change,  see Remark~\ref{rem:q=0},  to   get  $\theta\in  \R$.
\medskip

\subsection{SDE representation}

Let $\beta>0$  and $\theta\in \R$.  Let  $B=(B_t, t\geq 0)$ be  a standard
Brownian motion. Let $\alpha>0$ and  $S^{\alpha,\theta}(\rd t)$ be a Poisson point
measure           on           $\R_+$           with           intensity
$\alpha \beta \expp{2\beta\theta  t} \rd t$, independent  of the Brownian
motion $B$. We set $S^{\alpha,\theta}_t=S^{\alpha,\theta}([0, t])$ for $t\in \R_+$.
We define  the process $Z^{\alpha}=(Z_t^{\alpha},
t\geq 0)$ under  $\P^\theta$ as the unique strong solution  (conditionally on $S$) of the following SDE:
\begin{equation}
   \label{eq:Z+S}
  \rd Z_t^{\alpha}
  = \sqrt{2\beta Z_t^{\alpha}}\, \rd B_t
- 2\beta \theta Z_t^{\alpha} \rd t+ 2\beta \, (S^{\alpha,\theta}_t+1)\, \rd t
\quad
\text{for $t\geq 0$,}\quad\text{and $
Z_0^{\alpha}=0$}.
\end{equation}

\begin{prop}[An SDE with Poisson drift]
  \label{prop:Zaq}
 Let $\alpha>0$,  $\theta\in  \R$ and $t_0>0$. The     process
   $(Z_t,        t\in         [0,        t_0])$        under
$\N^\theta\left[  \bullet  \, M^{\alpha,\theta}_{t_0  }\right]$
(resp. under
$\N^\theta\left[ \bullet \, \tilde M^{\alpha,\theta}_{t_0 }\right]$)
 is distributed  as the
process     $(Z^\alpha_t, t\in  [0, t_0])$ under  $\P^\theta$ (resp.
$\P^{-\theta}$).
\end{prop}

The proof of this proposition is detailed in Section~\ref{sec:proof:Zaq}.
The process  $(Z^\alpha_t, t\geq 0)$ appears also in 
\cite[Theorem~3]{Ov94} when $\beta=1$ and $\theta=0$ (the function $h$
therein is given by $y^{-1}\, \ch^\alpha(s,y)$ up to a multiplicative
constant).
As  $  \partial_y \log ( \ch^\alpha (t, y) )$ does
 not depend on $t$, we simply write  $  \partial_y \log ( \ch^\alpha
 (\cdot, y) )$. 

\begin{cor}[The SDE with Poisson drift is a diffusion]
  \label{cor:Z=Ov}
   Let $\alpha>0$ and  $\theta\geq 0$.  The process  $(Z^\alpha_t, t\geq
   0 )$  satisfies the stochastic differential equation:
 \begin{equation}
      \label{eq:Ov-c>0}
      \rd  Z_t^{\alpha}= \sqrt{2\beta Z_t^{\alpha}}\, \rd  B_t
      -2\beta\theta Z_t^{\alpha} \, \rd t+{2\beta}\expp{2\beta\theta
        t} Z_t^{\alpha} \, \partial_y \log( \ch^{\alpha}
      (\cdot, \expp{2\beta\theta t}Z_t^{\alpha}))\, \rd  t,\quad t\geq0.  
    \end{equation}
\end{cor}
The proof of this corollary is detailed in Section~\ref{sec:proof-cor-Ov}.

\subsection{Proof of Proposition~\ref{prop:Zaq}}
\label{sec:proof:Zaq}
 Following Remark
\ref{rem:q=0}, we first use a scaling argument to remove the parameters
$\beta$ and $\theta$.
\medskip

Let  $\alpha>0$. Let  $S^\alpha=(S^\alpha_t, t\geq 0)$  be a  Poisson  process with  parameter
$\alpha$    independent   of    the    Brownian    motion   $B$.     Let
$Y^\alpha=(Y^\alpha_t, t\geq  0)$ be the  unique strong solution  (conditionally on $S$) of the
following SDE:
\begin{equation}
   \label{eq:eds-Ya}
  \rd Y^\alpha_t
  = \sqrt{2 Y^\alpha_t}\, \rd B_t + 2\, (S^{\alpha}_t+1)\, \rd t
\quad
\text{for $t\geq 0$,}\quad\text{and $
Y^\alpha_0=0$}.
\end{equation}
  Let $\beta, \alpha>0$ and
  $\theta\in \R$, and write $Z^{[\beta, \theta, \alpha]}$ for the
  process $Z^\alpha$ under $\N^\theta$ or $\P^\theta$ to stress the
  dependence in $\beta$ and $\theta$. Define the process
  $(Y'^\alpha, S'^\alpha)=\left((Y'^\alpha_s,S'^\alpha_s),
  s\in [0, 1/(2\theta_-))\right)$  by:
\begin{equation}
   \label{eq:time-changeYa=0}
  Y'^\alpha_s= \expp{2 \beta\theta t}
  Z_{t }^{[\beta,\theta,\alpha]}
  \quad\text{and}\quad
  S'^\alpha_s=S^{\alpha, \theta}_t,
  \quad\text{with}\quad
s=\inv{c^\theta_t} \cdot
\end{equation}

Then, it is  elementary that this deterministic time change yields the following result.
\begin{lem}
  \label{lem:YS}
  Let $\beta, \alpha>0$ and $\theta\in \R$.
The process  $\left(Y'^\alpha, S'^\alpha\right)$ under
$\P^\theta$ (whose law depends on $(\beta, \theta)$ and $\alpha$) is
distributed as $\Bigl((Y^\alpha_s, S^\alpha_s), s\in \bigl[0,
  1/(2\theta_-)\bigr)\Bigr)$.
\end{lem}

Let $(P_t, t\geq 0)$ be the transition semi-group on $\R_+ \times \N$ of the Markov process
$(Y^\alpha,S^\alpha)$.

\begin{lem}
  \label{lem:Pt=feller}
  The semi-group $(P_t, t\geq 0)$ is Feller, that is for all $t\geq 0$
  and all bounded continuous function $f$ defined on $\R_+\times \N$, the
  function  $P_t(f)$ is also bounded and continuous.
\end{lem}

\begin{proof}
  Let  $\left((Y_t^{\alpha,(x,s)},S_t^{\alpha,(x,s)}),   t\ge  0\right)$
  denote  the  solution  of  the  SDE  \eqref{eq:eds-Ya}  starting  from
  $(x,s)$. Let $\left(X^x_t, t\ge 0\right)$ be a Feller process starting
  from   $x$   (it  is   distributed   as   a   solution  to   the   SDE
  \eqref{eq:SDE-Feller}),          independent           of          the
  $(Y_t^{\alpha,(x,s)},S_t^{\alpha,(x,s)})_{t\ge   0}$.    By the branching  property, see~\eqref{eq:branching-prop}, we have the equality
  in distribution for the processes:
\[
  \left((Y_t^{\alpha,(x,s)},S_t^{\alpha,(x,s)}),t\ge 0\right)
  \overset{(d)}{=}\left((Y_t^{\alpha,(0,s)}+X^x_t,S_t^{\alpha,(0,s)}),t\ge
    0\right).
\]
Recall  $Q_t$   denote  the  semi-group   of  the  process   $X^x$,  see
Remark~\ref{rem:q=0}. Then  for every $t\ge 0$,  $x,y\in \R_+$, $s\in\N$
and every bounded continuous function $f$ defined on $\R_+\times \N$, we
have:
\begin{align*}
P_tf(x,s)-P_tf(y,s)
& =\E\left[f\left(Y_t^{\alpha,(x,s)},S_t^{\alpha,(x,s)}\right)
-f\left(Y_t^{\alpha,(y,s)},S_t^{\alpha,(y,s)}\right)\right]\\
& =\E\left[f\left(Y_t^{\alpha,(0,s)}+X^x_t,S_t^{\alpha,(0,s)}\right)
-f\left(Y_t^{\alpha,(0,s)}+X^y_t,S_t^{\alpha,(0,s)}\right)\right]\\
& =\E\left[Q_tf_{\left(Y_t^{\alpha,(0,s)},S_t^{\alpha,(0,s)}\right)}(x)
-Q_tf_{\left(Y_t^{\alpha,(0,s)},S_t^{\alpha,(0,s)}\right)}(y)\right]
\end{align*}
where $f_{(y,s)}$ is the continuous map $x\mapsto
f\left(y+x,s\right)$. By the Feller property of the semi-group $Q_t$ and
the dominated convergence theorem, we deduce that $\lim_{y\to
  x}P_tf(x,s)-P_tf(y,s)=0$.
This gives  the Feller property of the kernel $P_t$.
\end{proof}

We now give the density of $(Y^\alpha_t, S^\alpha_t)$. Recall that
$Y^\alpha_0=S^\alpha_0=0$. Let $\rN$ be the counting
measure on $\N$.
\begin{lem}
  \label{lem:densityYS}
  Let $t>0$. The random variable $(Y^\alpha_t, S^\alpha_t)$ has a density $f$ on
  $\R_+\times \N$ with respect
  to $\rd y \otimes \rN(\rd k)$ given by:
\begin{equation}
   \label{eq:def-density}
    f(y,k)=\inv{t^2}\frac{\alpha^k\, y^{k+1}}{k! (k+1)!} \expp{-(\alpha t+t^{-1} y)},\quad y\geq0,\, k\in\N.
\end{equation}
\end{lem}

\begin{proof}
 Conditionally on $S$, by \eqref{eq:eds-Ya}, we can see $Y^\alpha$ as a quadratic
  CSBP (with $\beta=1$) with immigration whose rate is $2 (S^\alpha_t+1)\rd
  t$. This implies that, conditionally on $S^\alpha$, the process $Y^\alpha$  is
  distributed as $\left(\sum_{i\in I}\ind_{\{h_i\leq
     t\}}\, Z^{(i)}_{t-h_i}, t\geq 0\right)$, where  $\sum_{i\in
   I}\delta_{(h_i, Z^{(i)})}(\rd t, \rd Z)$  is  a Poisson point
   measure on $\R_+ \times \cc[0, \R_+)$ with intensity $2 (S^\alpha_t+1)\rd
   t\,  \N[\rd Z]$ and  $\N$ is the excursion measure of a CSBP with
   branching mechanism $\psi(\lambda)=\lambda^2$.

\medskip
We deduce that for $\lambda, \mu\geq 0$:
\[
\E\left[\expp{-\lambda Y^\alpha_t - \mu S^\alpha_t}\right]
=\E\left[\expp{-\mu S^\alpha_t - \int_0^t2 (S^\alpha_r+1)
    \N\left[1-\expp{-\lambda Z_{t-r}}\right]\, \rd r}\right]
=\E\left[\expp{-\mu S^\alpha_t - 2\int_0^t (S^\alpha_r+1)\frac{\lambda}{1+ (t-r)
      \lambda}\, \rd r
    }\right],
\]
where we used  \eqref{eq:NeZ} for the last equality  (with $\beta=1$ and
$\theta=0$).  Denote  by $(\xi_i, i\in\N^*)$ the  increasing sequence of
the  jumping   times  of  the   Poisson  process  $S^\alpha$,   and  set
$\xi_0=0$. Then, we have on $\{S^\alpha_t=k\}$:
\begin{align*}
  \int_0^t (S^\alpha_r+1)\frac{\lambda}{1+ (t-r) \lambda}\, \rd r
  &=\sum_{i=0}^ k (i+1) \int^{\xi_{i+1}\wedge t}_{\xi_i}
    \frac{\lambda}{1+ (t-r) \lambda}\, \rd r\\
  &=-\sum_{i=0}^ k (i+1) \log(1+ (t-r)\lambda)
\bigg{|}^{\xi_{i+1}\wedge t}_{\xi_i}\\
 &=\sum_{i=0}^ k  \log(1+ (t-\xi_i)\lambda).
\end{align*}
Conditionally on $\{S^\alpha_t=k\}$, the random set $\{\xi_1, \ldots, \xi_k\}$
is distributed as $\{tU_1, \ldots, tU_k\}$ (notice the order is
unimportant and is not preserved), where $U_1, \ldots, U_k$
are independent random variables uniformly distributed on $[0, 1]$.
We deduce that:
\begin{align*}
   \E\left[\expp{-\lambda Y^\alpha_t - \mu S^\alpha_t}\right]
  &=\sum_{k\in \N} \frac{(\alpha t)^k\expp{-\alpha t- \mu k}}{k!}\,
    \E\left[\prod_{i=1}^k \bigl(1+t(1-U_i) \lambda\bigr)^{-2}\right] (1+t\lambda)^{-2}\\
  &= \sum_{k\in \N} \frac{(\alpha t)^k\expp{-\alpha t- \mu k}}{k!}\,
    (1+t\lambda)^{-k-2}\\
  &=\sum_{k\in \N} \int_{\R_+} \rd y \,\,   f(y,k) \expp{-\lambda y - \mu k},
\end{align*}
where for the last equality, we used the definition of $f$ given in
\eqref{eq:def-density}. This finishes the proof.
\end{proof}
Let $q'_t$ be the distribution of $Y^\alpha_t$ for $t\in \R_+$. We have
$q'_0=\delta_0$ the Dirac mass at 0,  and for
$t>0$, we deduce from Lemma \ref{lem:densityYS} that  $q'_t(\rd y)$ has a density, also denoted by $q'_t$, on $\R_+$ with
respect to the Lebesgue measure given by:
\[
q'_t(y)= t^{-2}  \expp{- y/t }\ch^{\alpha}(t,\, y),\quad t>0,\, y\geq0, 
\]
where $\ch^\alpha$ is defined in~\eqref{eqn:defhalpha}.
We now give some properties of the conditional law of $S_t$ given
$Y_t$. Recall $\Bz $ defined in~\eqref{eq:def-B0}.
\begin{lem}
  \label{lem:S|Y}
Let $y\in \R_+$.   The law of $S^\alpha_t$ conditionally on $\{Y^\alpha_t=y\}$ does
not depend on $t$. More precisely, we get for all $t\geq 0$, $k\in \N$
and $y\geq 0$:
\begin{equation}
   \label{eq:S|Y}
\P(S^\alpha_t=k|\, Y^\alpha_t=y)
= \inv{\Bz (\alpha y)} \, \frac{(\alpha y)^{k}}{k! (k+1)!}\cdot
\end{equation}
\end{lem}

\begin{proof}
Using Lemma \ref{lem:densityYS}, we directly have \eqref{eq:S|Y} for $t>0$.
Notice that for $y=0$, we have $\Bz (0)=1$ and
\[
\inv{\Bz (\alpha y)}  \, \frac{(\alpha y)^{k}}{k! (k+1)!}=\ind_{\{k=0\}}.
\]
As $(Y^\alpha_0,
S^\alpha_0)=(0,0)$, we deduce that  \eqref{eq:S|Y} also holds for $t=0$.
\end{proof}

We can now prove the Markov property of the process $Y=(Y_t, t\geq 0)$.
\begin{lem}
  \label{lem:Y-Markov}
  The process $Y^\alpha$ is Markov, and its transition
semi-group $(Q_t, t\in \R_+)$ is the unique Feller semi-group such that
$q'_t=q'_0 Q_t$ for $t\in \R_+$, with $q'_t$ the distribution of $Y^\alpha_t$.
\end{lem}
\begin{proof}
  We say  a probability kernel $K$  is continuous if for  all continuous
  and bounded function  $f$, $Kf$ is also continuous  (and bounded).  We
  shall check  hypothesis from \cite[Lemma 1]{rp:mf}.  With the notation
  therein ($X=(Y^\alpha,S^\alpha)$ and $\phi(y,s)=y$), the semi-group $(P_t, t\geq 0)$
  is  Feller,  see  Lemma \ref{lem:Pt=feller}.  The  probability  kernel
  $\Lambda(y; \, \rd z, \rd k)=  \P(S^\alpha_t=k|\, Y^\alpha_t=y) \, \delta_y(\rd z) \,
  \rN(\rd   k)$   is   clearly  continuous and
  does not depend on $t$.   The   probability   kernel
  $\Phi(y,k;\, \rd z)=\delta_y(\rd z)$  is also clearly continuous. Lemma
  \ref{lem:S|Y}  gives exactly condition  (i) in  \cite[Lemma 1]{rp:mf}.  We now
  check   condition  (ii)   in  \cite[Lemma   1]{rp:mf},  that   is  the
  one-dimensional marginal distributions of $Y^\alpha$, $(q'_t, t\in \R_+)$, are
  determining, that is  if $h$ and $g$ are  bounded continuous functions
  defined on  $\R_+$, then  $\E[h(Y^\alpha_t)]=\E[g(Y^\alpha_t)]$ for all  $t\in \R_+$
  implies $h=g$. To prove this, notice that:
\[
  t^2 \expp{\alpha t} \, \E\bigl[h(Y^\alpha_t)\bigr]= \int_{\R_+} \expp{-t^{-1} y } H(y) \rd y,
\]
where $H(y)  = h(y) y \, \Bz (\alpha y) $.  As the
Laplace  transform characterizes  the  bounded  continuous function,  we
deduce that if  $\E[h(Y^\alpha_t)]=\E[g(Y^\alpha_t)]$ for all $t\in  \R_+$, then $H=G$
(with $G(y) =g (y) y \, \Bz (\alpha y)$) and thus
$h=g$ on $(0, +\infty )$ and by continuity on $\R_+$.

As the assumption of \cite[Lemma 1]{rp:mf} are satisfied, we deduce that
$Y^\alpha$  is   a  Markov   process,  and   that  its   transition  semi-group
$(Q_t,  t\in   \R_+)$  is  the   unique  Feller  semi-group   such  that
$q'_t=q'_0 Q_t$ for $t\in \R_+$, with $q'_t$ the distribution of $Y^\alpha_t$.
\end{proof}

We now  compare the distribution  of $Y^\alpha$ and the  distribution of
the Feller  diffusion $Y$  defined in Remark  \ref{rem:q=0}, which  is a
CSBP    with    parameter    $\beta=1$   and    $\theta=0$.    Following
\eqref{Mmartingale}, we set for $t>0$:
\[
  M^\alpha_t  =
  \ch^{\alpha}( t ,\, Z_t)
  =  \expp{-\alpha  t} \, Z_t\,   \Bz (\alpha Z_t).
\]
Let  $\N$ denote the canonical measure of $Y$.

\begin{lem}
  \label{lem:Y=Z}
  Let $\alpha>0$. Let $t_0>0$.
  The process $(Y^\alpha_t, t\in [0, t_0])$ has  the same distribution
  as the process $ (Y_t, t\in [0, t_0])$
  under
  $\N\left[   \bullet   \,   M^{\alpha}_{t_0    }\right]$.
\end{lem}

\begin{proof}
  We first check the two processes have the same one-dimensional
  marginals. Clearly $Y^\alpha_0=Y_0=0$.
Let $t>0$. According to Lemma \ref{lem:density}, the entrance law of $Y_t$ under $\N$
has density $y\mapsto t^{-2} \expp{-y/t}$.
We deduce that  for $\lambda\geq 0$:
\[
  \N\left[\expp{-\lambda Y_t} M^{\alpha}_{t}\right]
  =\int_{\R_+}\expp{-\lambda y} \, \ch^\alpha(t,y) \, 
    t^{-2}\expp{-y/t} \,
    \rd y
    = \int_{\R_+}\expp{-\lambda  y}\, q'_t(y)    \,\rd y 
= \E\left[\expp{-\lambda Y^\alpha_t}\right].
\]
Since  the Laplace transform characterizes the probability
distribution on $\R_+$, we obtain that $Y^\alpha_t$ has  the same distribution
  as $ Y_t$
  under   $\N\left[   \bullet   \,   M^{\alpha}_{t   }\right]$.
\medskip

Using    Doob's    $h$-transform,    we    get    that    the    process
$(Y_t,            t\in            [0,            t_0])$            under
$\N\left[ \bullet \, M^{\alpha}_{t_0}\right]$ is Markov.  Using  that
$M^{\alpha}$ is a martingale under $\N$ (see Proposition \ref{Prop:Mmartingale} and
use that $Y$ is distributed as $Z$ when $\beta=1$, $\theta=0$), that
$M^\alpha_t$ is a function of $Y_t$, and that $Y$  is   Feller   under   $\N$,  we   get   that
$(Y_t,            t\in            [0,            t_0])$            under
$\N\left[  \bullet \,  M^{\alpha}_{t_0}\right]$ is  also
Feller.    We   deduce   from   the   uniqueness   property   of   Lemma
\ref{lem:Y-Markov}  and   the  identification  of   the  one-dimensional
marginals from the first step of the proof, that
 $(Y^\alpha_t, t\in [0, t_0])$ has  the same distribution
  as $ (Y_t, t\in [0, t_0])$
  under
  $\N\left[   \bullet   \,   M^{\alpha}_{t_0    }\right]$.
\end{proof}

We can now give the proof of Proposition~\ref{prop:Zaq}.
Let $\beta,  \alpha>0$, $\theta\in  \R$ and $t_0>0$.  Using  the time
changes  given  by   Remark~\ref{rem:q=0}
and 
\eqref{eq:time-changeYa=0},     we    deduce     that    the     process
$(Z^\alpha_t, t\in  [0, t_0])$ under  $\P^\theta$ is distributed  as the
process        $(Z_t,        t\in         [0,        t_0])$        under
$\N^\theta\left[  \bullet  \, M^{\alpha,\theta}_{t_0  }\right]$.   Then,
using Corollary \ref{cor:RM-mart},  we also deduce that  the
process $(Z^\alpha_t,  t\in [0, t_0])$ under  $\P^{-\theta}$ is distributed
as     the     process     $(Z_t,      t\in     [0,     t_0])$     under
$\N^\theta\left[ \bullet \, \tilde M^{\alpha,\theta}_{t_0 }\right]$.
This finishes the proof of Proposition \ref{prop:Zaq}.

\subsection{Proof of Corollary~\ref{cor:Z=Ov}}
\label{sec:proof-cor-Ov}
\begin{proof}
Since $Y^\alpha$ satisfies \eqref{eq:eds-Ya}  for $Y^{\alpha}_0=0$, then
by Lemma \ref{lem:S|Y} and the property of Poisson point process, the
process  $\left(Y_t^{\alpha}\right)_{t\geq 0}$ starting from
$Y_0^{\alpha}=y$
satisfies:
\[
\rd Y_t^{\alpha}= \sqrt{2Y_t^{\alpha}}\, \rd B_t
 +2(S_t^{\alpha}+\xi^{y})\, \rd t,\quad t\geq 0,\quad 
\]
where $S^{\alpha}$ and $\xi^y$ are independent and (since
$S^\alpha_0=0$)  for $k\in \N$: 
\[
  \P(\xi^y=k+1)= \inv{\Bz (\alpha y ) } \, \frac{(\alpha y)^{k}}{k! (k+1)!}
\cdot
\]
By conditioning on $(S^{\alpha}, \xi^y)$ and applying It\^o's
formula, one has
for any $y\geq0$: 
\begin{align*}
    \lim_{t\rightarrow0+}\frac{\E[\expp{-\lambda
  Y_t^{\alpha}}|Y_0^{\alpha}=y]-\expp{-\lambda y}}{t} 
& 
= -2\lambda \E[\xi^y]\expp{-\lambda y}+\lambda^2 \expp{-\lambda y}y
\cr
&
= -2\lambda \expp{-\lambda y} \inv{\Bz (\alpha y)}\,
      \sum_{k=0}^{\infty}\frac{(\alpha y)^{k}}{(k!)^2}+
      \lambda^2\expp{-\lambda y}y. 
\end{align*}
Since $Y^{\alpha }$ is a diffusion  process by Lemma~\ref{lem:Y-Markov},
the above computation implies that for $t\geq 0$:
\begin{align*}
  \rd  Y_t^{\alpha}
  &= \sqrt{2Y_t^{\alpha}}\, \rd B_t
   +\frac{2}{\Bz (\alpha Y_t^{\alpha})} \,
   \sum_{k=0}^{\infty}\frac{(\alpha Y_t^{\alpha})^{k}}{(k!)^2}\, \rd t\\
  &= \sqrt{2Y_t^{\alpha}}\, \rd B_t
   +2 Y^\alpha_t  \, \partial_y \log ( \ch^\alpha (\cdot, Y^\alpha_t) ) \, \rd t,
 \end{align*}
 where  $  \partial_y \log ( \ch^\alpha (t, y) )$ does
 not depend on $t$ as:
 \[
   y\,  \partial_y \log ( \ch^\alpha (\cdot, y) )= \frac{2}{\Bz (\alpha y)} \,
   \sum_{k=0}^{\infty}\frac{(\alpha y)^{k}}{(k!)^2}\cdot
\]
We deduce  $Y^{\alpha  }$ is  a  solution  to  the equation  established  in
\cite[Theorem~3]{Ov94}.   Recall from \eqref{eq:time-changeYa=0} that
 $Z_t^{\alpha}=\expp{-2\beta\theta             t}Y_{s}$             with
 $s=1/c^\theta_t=(\expp{2\beta\theta    t}-1)/2\theta$.     With    this
 deterministic time-change, we deduce that the process $Z^{\alpha}$ also
 satisfies~\eqref{eq:Ov-c>0}.
\end{proof}


\section{Backbone decomposition}
\label{sec:backbone-intro}
We  introduce basic facts on the space  of real trees in
Section~\ref{sec:tree-notation}.   We  recall  some  properties  of  the
Brownian    CRT    in     Section~\ref{sec:Levy-T}.     We    give    in
Section~\ref{sec:dis-graft}  a recursive  construction of  some discrete
random    trees    using    a     grafting    procedure    defined    in
Section~\ref{sec:planar}.  Let   us  stress  that  the   measurable  and
topological properties of the grafting  procedure, as well as its formal
definition,  are discussed  in detail  in Section~\ref{sec:graft1}.   In
Section~\ref{sec:backbone}, we provide a decomposition of a (sub)critical
Brownian CRT according to $n$ leaves  at a given
distance from the root and  uniformly chosen at random, this is  a generalization of the case $n=1$ from
\cite[Theorem~4.5]{DLG05}.     We    prove    our   main    results    in
Section~\ref{sec:main-result} on  the local convergence of  the Brownian
CRT conditioned to have  a large population at time $t$,  as $t$ goes to
infinity.



\subsection{Notations for trees}
\label{sec:tree-notation}
\subsubsection{Real trees}
\label{sec:real-trees}
We use the framework of real trees to encode the genealogy of a continuous state branching process. We refer to \cite{ev08} for a detailed introduction to real trees.

A real tree (or simply a tree in the rest of the text) is a metric space $(T,d)$ that satisfies the two following properties for every $u,v\in T$:
\begin{enumerate}[(i)]
\item There is a unique isometric map $f_{u,v}$ from $[0,d(u,v)]$
  into $T$ such that:
\[
f_{u,v}(0)=u\qquad\text{and}\qquad f_{u,v}\bigl(d(u,v)\bigr)=v.
\]
\item If $\varphi$  is a continuous injective map from  $[0,1]$ into $T$
  such  that  $\varphi(0)=u$  and  $\varphi(1)=v$,  then  the  range  of
  $\varphi$ is also the range of $f_{u,v}$.
\end{enumerate}
The range of the map $f_{u,v}$ is denoted by $\lb u,v\rb$. It is the unique
continuous path that links $u$ to $v$ in the tree. We will write $\lb u,v\lb$ (resp. $\rb u,v\rb$, $\rb u,v\lb$) for $\lb u,v\rb\setminus\{v\}$ (resp. $\lb u,v\rb\setminus\{u\}$, $\lb u,v\rb\setminus\{u,v\}$).

A rooted tree is a tree $(T,d)$ with a distinguished vertex denoted by
$\root$ and called the root. We always consider rooted trees in this
work. For an element $x$ of a rooted tree $(T, d, \root)$, we denote by $H(x)=d(\root, x)$
its height, and we set $H(T)=\sup_{x\in T}H(x)$ the height of the tree $T$.

An element $x$ of   $T\setminus\{\root\}$  is a leaf if $T\setminus\{x\}$ has only
one connected component; by convention the root is a leaf if and only if
$T$ is reduced to the root. We denote by  $\leaf(T)$ the
(non-empty) set of leaves of $T$. The skeleton of the
tree is the set $\mathrm{Sk}(T)=T\setminus\leaf(T)$. The set of
branching points (or vertices) $\mathrm{Br}(T)$ is the set of $x\in T$ such that
$T\setminus \{x\}$ has at least 3 connected components if $x\neq \root$
or at least 2 components if $x=\root$.

For a vertex $x\in T$, we define the subtree $T_x$ ``above'' $x$ as:
\[
T_x=\bigl\{y\in T\, \colon\,  x\in \lb\root,y\rb\bigr\}.
\]
The  real tree $T_x$ is  endowed with the distance induced by
$T$ and  will be rooted at $x$.
If $u,v\in T$,
we denote by $u\wedge v$ the most recent common ancestor of $u$ and $v$,
\emph{i.e.} the unique vertex of $T$  such that:
\[
\lb\root,u\rb\cap\lb\root,v\rb=\lb\root,u\wedge v\rb.
\]
If $(T,d,\root)$ is a rooted real tree and $a$ is a positive real number, we define the scaled tree $aT$ as:
\begin{equation}\label{eq:scale}
aT=(T,ad,\root)
\end{equation} 
where all the distances in the tree $T$ are multiplied by the factor $a$.

The trace of the Borel $\sigma$-field
of $T$ on $\mathrm{Sk}(T)$ is generated by the sets $\lb s,s'
\rb,\ s,s' \in \mathrm{Sk}(T)$ (see \cite{EPW06}). Hence, there exists a
$\sigma$-finite Borel measure $\length^{T}$ on $T$, such that:
\[
\length^{T}\bigl(\mathrm{Lf}(T)\bigr) = 0
\quad\text{and}\quad
 \length^{T}\bigl(\lb s,s'
\rb\bigr)=d(s,s').
\]
This measure $\length ^T$ is called the length measure on $T$.
When there is no ambiguity, we
simply write $\length$ for $\length^T$.

\subsubsection{Gromov-Hausdorff distance and sets of trees}\label{sec:sets_of trees}

We endow the set of (isometry equivalence classes) of rooted real tree with the classical Gromov-Hausdorff distance whose definition (with the notion of correspondances) is described below. We refer to \cite{Khe23} for general results on Gromov-Hausdorff metrics.

Let $(T,d,\root)$ and $(T',d',\root')$ be two rooted compact real trees. A correspondence $\cR$ between $T$ and $T'$ is a subset of $T\times T'$ such that:
\begin{enumerate}[(i)]
\item for all $x\in T$, there exists $x'\in T'$ such that $(x,x')\in \cR$,
\item for all $x'\in T'$, there exists $x\in T$ such that $(x,x')\in \cR$,
\item $(\root,\root')\in\cR$.
\end{enumerate}
The distortion of such a correspondence $\cR$ is defined as:
\[
\mathrm{dist}(\cR)=\sup\left\{\bigl|d(x,y)-d'(x',y')\bigr|;\ (x,x'),(y,y')\in\cR\right\}.
\]
For two compact rooted trees $(T,d,\root)$ and $(T',d',\root')$ we set:
\[
d_{GH}(T,T')=\inf\frac{1}{2}\mathrm{dist}(\cR),
\]
where  the  infimum  is  taken   over  all  the  correspodences  between
$(T,d,\root)$ and $(T',d',\root')$. The function $\dgh$ is the so-called
Gromov-Hausdorff pseudo-distance,  see \cite{m:tmrag}.   Furthermore, we
have  that $\dgh(T,  T')=0$ if  and only  if there  exists an  isometric
bijection  from $(T, d)$ to $(T', d')$ which preserves the
root.  The  relation $\dgh(T,  T')=0$  defines  an equivalence  relation
between compact rooted  trees.  The set $\TK$ of  equivalence classes of
compact rooted trees endowed with $\dgh$  is then a metric Polish space,
see \cite[Proposition~9]{m:tmrag}. We shall consider below the trivial  tree
$\Tz\in \TK$ reduced to its root. 

\medskip
We can generalize this definition to compact $n$-pointed rooted trees
where a $n$-pointed rooted tree is a triplet $(T,d,\bv)$ where $(T,d)$
is a rooted real tree and $\bv=(v_0,v_1,\ldots,v_n)$ with that  $v_0=\root$ is the
root of $T$ and $v_1,\ldots,v_n$ are $n$ distinguished (possibly equal)
vertices. A correspondence between two $n$-pointed rooted trees
$(T,d,\bv)$ and $(T',d',\bv')$ is a correspondance between $(T,d,\root)$
and $(T',d',\root')$ which satisfies moreover $(v_i,v'_i)\in\cR$ for all
$i\in\{1,\ldots,n\}$, where $\bv'=(v'_0,v'_1,\ldots,v'_n)$ with $v'_0=\root'$, the root of $T'$. 
The distance $\dghn$ on the space $\TKn$ of equivalence classes of
compact $n$-pointed rooted trees is then defined in the same way as
$\dgh$ on $\TK$, and the metric space $(\TKn,\dghn)$ is Polish; and notice
that $(\TK, \dgh)=(\TKz, \dghz)$. 

\medskip

For a  rooted $n$-pointed tree $(T, d, \bv)$ and $t\geq t_T=\max_{i\in \{0,
  \ldots, n\}} d(\root, v_i)$, we define the rooted $n$-pointed tree  $T$ truncated at level $t$ as $(r_t(T,\bv), d, \bv)$ with:
\begin{equation}
   \label{eq:def-rt0}
  r_t(T, \bv)=\bigl\{x\in T\, \colon\, H(x)\leq t\bigr\}, 
\end{equation}
and the  distance on $r_t(T,  \bv)$ is given  by the restriction  of the
distance   $d$.     We   shall   simply   write    $r_t(T,   \bv)$   for
$(r_t(T, \bv),  d, \bv)$.  A  rooted $n$-pointed  tree $(T, d,  \bv)$ is
locally  compact if  $r_t(T,  \bv)$ is  a compact  rooted  tree for  all
$t\geq   t_T$.    The   locally   compact   trees   $(T,d,   \bv)$   and
$ (T',d', \bv')$ are equivalent if and only if there exists an isometric
one-to-one  map  from  $(T,  d)$  to  $(T',  d')$  which  preserves  the
distinguished vertices.   This defined  indeed an  equivalence relation.
The set $\TLn$ of equivalence classes of locally compact rooted trees is
then endowed with  a distance $\dlghn$ in the spirit  of \cite{adh}, see
Section~\ref{sec:root-Tn}      below    and           more          precisely
Proposition~\ref{prop:TLn-Polish}, so  that it is a  metric Polish space
and $\TKn$ is an open dense subset of $\TLn$. For $n=0$, we simply write
$\TL$ and $\dlgh$ for $\TLn$ and $\dlghn$.
We shall consider below the infinite spine tree $\Tu=(\R_+, |\cdot |,
0)\in \TL$, where $|\cdot |$ is the usual Euclidean distance. 

\subsubsection{Grafting operation}
\label{sec:graft-int}
We recall the grafting  operation of \cite{AD09}.  Let $(T,d,(\root,x))$
be  a locally  compact  rooted 1-pointed  tree  and $(T',d',\root')$  be
locally compact rooted trees. We  define the tree $T\circledast_x T'$ as
the tree obtained  by grafting $T'$ on  the tree $T$ at  vertex $x$.  We
set:
\begin{gather}
  \label{eq:graft-2dis-tree}
  T\circledast_x T'=T\sqcup \bigl(T'\setminus\{\root'\}\bigr),\\
  \label{eq:graft-2dis-dist}
  \forall y,y'\in T\circledast_x T',\ d^\circledast(y,y')=\begin{cases}
d(y,y') & \text{if } y,y'\in T,\\
d'(y,y') & \text{if }y,y'\in T',\\
d(y,x)+d'(\root',y') & \text{if } y\in T,\ y'\in T',
\end{cases}
\end{gather}
where $\sqcup$ denotes  the disjoint union of two  sets. By construction
$(T\circledast_x  T', d^\circledast,  \root)$ is  a locally
compact  rooted  tree.   It is  easy  to  see  that  the
equivalence class  of $T\circledast_x T'$ does not  depend of
the choice of the representatives in  the equivalence classes of $T$ and
$T'$ and hence the grafting operation is well-defined on $\TL$; it is
even continuous, see Lemma~\ref{lem:graft-intro}. 
We also refer to Section~\ref{sec:graft1} for a more general grafting
procedure and  its topological properties.

\medskip

Let $(T,\bv)\in\TLn$ be either the infinite spine tree $\Tu$ (and $n=0$)
or a  discrete tree, that  is, a compact rooted  real tree with  all its
leaves  being distinguished  (see~\eqref{eq:def-TDn} for  a more  formal
definition)        and        then        $n\in        \N^*$.        Let
$\cm=\sum_{i\in  I}\delta_{(x_i,T_i)}$   be  a  point  measure   on  the
$T\times     \TL$.       We     define     intuitively      the     tree
$\graft_n((T,\bv),\cm)\in \TLn$ as the tree:
\begin{equation}
   \label{eq:rev-graft-n}
   \graft_n((T,\bv),\cm)=\bigl(T\circledast _{x_i,i\in I}(T_i,i\in I),\bv\bigr)
\end{equation}
obtained by  grafting each   locally compact rooted  tree $T_i$  on $T$  at point
$x_i$ (and  keeping the $n$ distinguished  elements $\bv$ of $T$).   It is not
clear that the  resulting tree belongs to $\TLn$  (some assumptions must
be added  to $\cm$)  nor that  this infinite  grafting procedure  can be
proceeded in a measurable way (so that  we get indeed a random tree when
the   tree  $T$   and  the   point  measure   $\cm$  are   random).  
We  give  a formal  definition  of  this
procedure in
Section~\ref{sec:appli-graft} and check
that it is well defined (after some  lengthy topological  preliminaries)
with good measurable property in $\TLn$,  where  $\cm$
is a particular Poisson point measure  considered in the context of
the backbone decomposition from Section~\ref{sec:backbone}.
Even if the presentation~\eqref{eq:rev-graft-n} is abusive,   we
  stick to  this informal definition  for simplicity.

\subsection{Brownian CRTs  and Kesten trees}
\label{sec:Levy-T}

Brownian CRTs  are random trees in $\TL$ that code for the genealogy of continuous-state branching processes.

Before recalling the  definition of such trees, we  give some additional
notation.   For a  locally compact  rooted  tree $\bt$,  we define  the
population at level $a$ as the sub-set:
\[
\mathcal{Z}_\bt(a)=\{u\in\bt,\ H(u)=a\}.
\]
We denote by $(\bt^{(i),*},i\in I)$ the connected components of the open
set  $\bt\setminus  r_a(\bt)$. For  every  $i\in  I$, let  $\root_i$  be
 the   MRCA   of  $\bt^{(i),*}$,  which is equivalently characterized by
$\lb \root, \root_i\rb=\cap_{u\in \bt^{(i),*}}  \lb \root, u\rb$; notice
that    $\root     _i\in    {\mathcal{Z}_\bt(a)}$.    We     then    set
$\bt^{(i)}=\bt^{(i),*}\cup\{\root _i\}$ so that $\bt^{(i)}$ is a
locally compact rooted tree with root $\root  _i$, and we consider the point measure
on $\mathcal{Z}_\bt (a)\times \TL$:
\[
\cn_a^\bt=\sum_{i\in I} \delta_{(\root _i,\bt^{(i)})}.
\]

We then recall  the definition of the excursion  measure $\N^\theta$ for
$\beta>0$ and 
$\theta\ge 0$  associated with  a Brownian CRT from  \cite{DLG05}. The
underlying parameter $\beta$ is fixed, and will be omitted from the notation.
There
exists a measure $\N^\theta$ on $\TK$ (and hence on $\TL$) such that:
\begin{itemize}
\item[(i)] \textbf{Existence of a local time.}
For every $a\ge 0$ and for $\N^\theta[\rd \ct]$-a.e. $\ct$, there exists a
finite measure $\Lambda_a$ on $\ct$ such that
\begin{itemize}
\item[(a)] $\Lambda_0=0$ and, for every $a>0$, $\Lambda_a$ is
  supported on $\mathcal{Z}_\ct(a)$.
\item[(b)] For every $a>0$,  $\N^\theta[\rd \ct]$-a.e., we have
  $\{\Lambda_a\ne 0\}=\{H(\ct)>a\}$.
\item[(c)] For every $a>0$, $\N^\theta[\rd \ct]$-a.e., we have  for every  continuous function $\varphi$ on $\ct$:
\begin{align*}
  \langle \Lambda_a,\varphi\rangle
  & =\lim_{\varepsilon\to0+}\frac{1}{c_\varepsilon^\theta}
    \int\cn_a^\ct(\rd u, \rd\ct')\varphi(u)\ind_{\{H(\ct')\ge\varepsilon\}}\\
  & =\lim_{\varepsilon\to0+}\frac{1}{c_\varepsilon^\theta}
    \int\cn_{a-\varepsilon}^\ct(\rd u, \rd\ct')\varphi(u)\ind_{\{H(\ct')\ge\varepsilon\}}.
\end{align*}
\end{itemize}
\item[(ii)] \textbf{Branching property.}
For every $a>0$, the conditional distribution of the point measure
$\cn_a^\ct(\rd u, \rd\ct')$, under the probability measure
$\N^\theta[\rd\ct\,|\, H(\ct)>a]$ and given $r_a(\ct)$, is that of a
Poisson point measure on $\mathcal{Z}_\ct(a)\times \TL$ with intensity
$\Lambda_a(\rd u)\N^\theta[\rd\ct']$.
\item[(iii)] \textbf{Regularity of the local time process.}
We can choose a modification of the process $(\Lambda_a,a\ge 0)$ in such a way that the mapping $a\longmapsto \Lambda_a$ is $\N^\theta[\rd \ct]$-a.e. continuous for the weak topology of finite measures on $\ct$.
\item[(iv)] \textbf{Link with CSBP.}
Under $\N^\theta[\rd\ct]$, the process $(\langle
\Lambda_a,1\rangle,a\ge 0)$ is distributed as a CSBP  under its canonical measure with branching mechanism:
\[
\psi(\lambda)=\beta \lambda^2+2\beta\theta\lambda,\quad \lambda\geq0.
\]
\end{itemize}
\medskip

We now extend the definition of the  measure $\N^\theta$ (only on $\TL$) for
$\theta<0$  by a  Girsanov transformation,  following \cite{ADH14}.  For
$t\geq  0$, set  $\cg_t=\sigma(r_t(\ct))$ and  $Z_t=\Lambda_t(\ct)$, the
latter  notation is  consistent with  Section~\ref{sec:2csbp}. The  CSBP
process  $Z=(Z_t, t\geq  0)$ is  Markov with  respect to  the filtration
$(\cg_t, t\geq 0)$.  For $\theta<0$ and $t>0$, we set:
\begin{equation}
   \label{eq:abs-cont2}
   \N^{-\theta}[\rd \ct] _{|\cg_t}=\expp{2\theta Z_t}  \N^{\theta}[\rd \ct]
  _{|\cg_t}.
\end{equation}
Then properties  (i) to  (iv) still hold  for every  $\theta\in\R$. This
Girsanov transformation  is consistent with the  Girsanov transformation
of CSBP  given by~\eqref{eq:abs-cont}. Let  us stress that  the measure
$\N^\theta$ on $\TL$ depends also on the parameter $\beta>0$.

\medskip
The      so-called      Kesten      tree       with      parameters
$(\beta,  \theta)\in \R_+^*\times  \R_+$ can be defined as the genealogical tree associated with the continuous-state branching process with the same parameters, conditioned on non-extinction (see for instance \cite{L07}). This latter process can also be defined by adding to the initial process a particular immigration. We use this second approach to extend the definition of the Kesten tree for $\theta<0$.

Using    our    framework,    the   Kesten    tree    with    parameters
$(\beta, \theta)\in \R_+^*\times  \R$ is built as a  countable family of
trees defined  by a Poisson point  measure grafted on the  infinite spine
tree $\Tu$:
\begin{equation}
  \label{eq:def-Kesten}
  \ct^*=\graft_0(\Tu, \cm),
\end{equation}
   with $\cm(\rd h, \rd T)$ a Poisson  point measure on $\R_+\times \TL$ with
intensity  $2\beta  \ind_{\{h>0\}}\rd  h \,  \N^{\theta}[\rd  T]$.
We refer to Section~\ref{sec:appli-graft} for a more formal presentation
which in particular implies that the Kesten tree is a $\TL$-valued
random variable, see Lemma~\ref{lem:Kesten}. 



\subsection{The set of (planar) discrete trees}
\label{sec:planar}
A  discrete tree  is  a compact  rooted  tree with  a  finite number  of
leaves.  We denote  by $\TDn$  the subset  of $n$-pointed  discrete tree
whose  leaves are  distinguished  (see~\eqref{eq:def-TDn}  for a  formal
definition): for $(\bt, \bv)\in \TDn$  with $\bv=(v_0=\root, \ldots, v_n)$, we
have  that  $\leaf(\bt)\subset  \{v_0,  \ldots,  v_n\}$.   According  to
Lemma~\ref{lem:discrete=closed},  the  set  $\TDn$ is  closed.   We  can
consider  a discrete  tree with  a planar  structure by  enumerating its
leaves, or more precisely its distinguished vertices, ``from the left to
the right''.  This will allow  us to  define on oriented  grafting; this
will  be  used  in  the  next  section.   Intuitively  a  discrete  tree
$(\bt,  \bv)\in \TDn$  is a  planar tree  if for  all $x\in  \bt$, there
exists   $0\leq   i_\mathrm{g}\leq   i_\mathrm{d}\leq   n$   such   that
$v_i\in          \bt_x$         if          and         only          if
$i\in   \{i_\mathrm{g},    \ldots,   i_\mathrm{d}\}$;   we    check   in
Section~\ref{sec:planar-def} that  the set  of (equivalence  classes of)
$n$-pointed planar tree $\TPn \subset \TDn$ is also closed.

\begin{figure}[ht]
\includegraphics[scale=0.5]{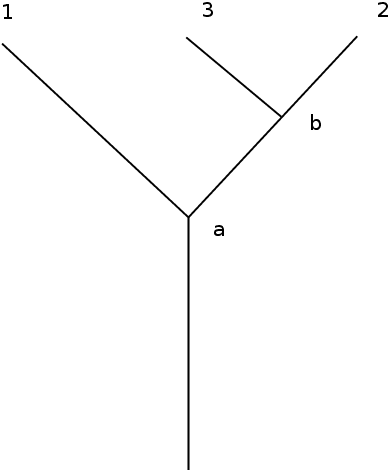}
\caption{A tree $(\bt,\bv)\in \T^{(3)}_{\mathrm{plan}}$ with $\bv=(\root,1,3,2)$}
\end{figure}

\medskip

We  now define  an oriented  grafting  of a  discrete tree  $\bt'$ on  a
discrete tree  $\bt$ at  point $x\in  \bt$; we shall  use later  on this
construction for  planar trees;  this is similar  to the  first grafting
defined  in  Section~\ref{sec:graft-int} but  for  the  ordering of  the
distinguished  vertices.  Formally,  if  $(\bt,\bv)$  be an  $n$-pointed
discrete  tree with  $\bv=(v_0=\root,v_1,\ldots,v_n)$, $(\bt',\bv')$  an
$m$-pointed discrete tree with $\bv'=(v'_0=\root',v'_1,\ldots,v'_m)$ and
$x\in \bt$, we define for $\varepsilon\in\{{\rm g},{\rm d}\}$:
\begin{equation}
   \label{eq:def-gd-graft}
(\bt,\bv)\circledast _x^\varepsilon(\bt',\bv')=(\bt\circledast
_x\bt',\bv\circledast^\varepsilon\bv')\in {\T^{(n+m)}_{\mathrm{dis}}} 
\end{equation}
with $\bt\circledast
_x\bt'$
defined in~\eqref{eq:graft-2dis-tree} and: 
\begin{align}
   \label{eq:def-vg-graft}
  \bv\circledast^{\rm g}\bv'
  & =(v_0,\ldots,v_{i_{\rm g} -1},v'_1,\ldots,v'_m,v_{i_{\rm g}} ,\ldots,v_n),\\
   \label{eq:def-vd-graft}
  \bv\circledast^{\rm d}\bv'
  & =(v_0,\ldots,v_{i_{\rm d} },v'_1,\ldots,v'_m,v_{i_{\rm d} +1},\ldots,v_n),
\end{align}
where:
\begin{equation}
   \label{eq:def-igd-graft}
  i_{\rm g}= \min\{i\in \{0, \ldots, n\}\, \colon \, v_i\in \bt_x\}
  \quad\text{and}\quad
  i_{\rm d}=\max\{i\in \{0, \ldots, n\}\, \colon \, v_i\in \bt_x\},
\end{equation}
and the convention that
if $i_{\rm g}=0$ (that is, $x=\root$), then 
$\bv\circledast^{\rm g}\bv'=(v_0,v'_1,\ldots,v'_m,v_{1} ,\ldots,v_n)$,
and 
if $ i_{\rm d}=n$, then
$\bv\circledast^{\rm d}\bv'=(v_0,\ldots,v_n,v'_1,\ldots,v'_m)$. 
Let us stress that $  i_{\rm g}$ and $ i_{\rm d}$ are well defined as all the leaf are
distinguished.
Notice  also that if $(\bt, \bv)$ and $(\bt', \bv')$ are planar, so is
$(\bt,\bv)\circledast _x^\varepsilon(\bt',\bv')$.

Furthermore, for  $i\in\{1,\ldots,n\}$ and $h\le
H(v_i)$, we shall consider the grafting of $\bt'$ at $x_{i,h}\in \bt$
the point of $\lb \root,v_i\rb$ at height $h$: 
\begin{equation}
  \label{eq:graft}
(\bt,\bv)\circledast _{i,h}^\varepsilon(\bt',\bv')=(\bt,\bv)\circledast
_{x_{i,h}}^\varepsilon(\bt',\bv'). 
\end{equation}
Notice this latter grafting is well defined on the equivalent classes of
discrete trees, and it is measurable thanks to
Lemma~\ref{lem:graft2-gd-cont}.

\subsection{A discrete random tree constructed by successive grafts}
\label{sec:dis-graft}

\subsubsection{A random tree}
\label{sec:bT}
In this section, for $a\geq 0$,
we denote by
$\bigl([0,a],(0, a)\bigr)\in  \TDu$ the (equivalent class of the) tree  $[0,a]$ endowed  with the
usual distance on $\R$, rooted at $\root=0$ and pointed at $a$;
and when there is no possible confusion we simply denote it by $[0,
a]$.

Let $t>0$ and  let $\nu$ be a probability measure  on $[0,t]$.  Let
$\xi=(\xi_k,  \,  k\in  \N^*)$  be  a  sequence  of  independent  random
variables with distribution $\nu$ and  let $\bigl((K_k,\varepsilon_k), \,k\in \N^*\bigr)$ be
a sequence of  independent random variables independent  of the sequence
$\xi$, with $K_k$ uniformly  distributed on $\{1,\ldots,k\}$ independent
of $\varepsilon_k$  uniformly distributed  on $\{{\rm g},  {\rm d}\}$.   For every
integer  $n\ge  2$,  we set  $(\xi_1^{(n)},\ldots,\xi_{n-1}^{(n)})$  the
increasing order statistic of $(\xi_1,\ldots,\xi_{n-1})$. Then we define
the            family             of            pointed planar           trees
$\bigl((\bT_1^{(n)},  \bv_1^{(n)}),\ldots,(\bT_n^{(n)}, \bv_n^{(n)})\bigr)$,  with
$(\bT_k^{(n)}, \bv_k^{(n)}) \in \TPk$, recursively by:
\begin{itemize}
\item  $\bT_1^{(n)}=[0, t]$, that is, $(\bT_1^{(n)}, \bv_1^{(n)})=
  \bigl([0,t],(0, t)\bigr)\in  \TPu$.
\item  For   every  $k\in  \{1,  \ldots,   n-1\}$,  conditionally  given
  the  random variable
  $(\bT_k^{(n)},  \bv_k^{(n)})$ in $\TPk$,  we  define the  $\TLkk$-valued  random
  variable
  $( \bT_{k+1}^{(n)}, \bv_{k+1}^{(n)})$ by grafting a branch of
  length  $t-   \xi_{k+1}^{(n)}$ uniformly on the left or on the
  right of a uniformly chosen   vertex  among
  the $k$ vertices of
  $\bT_k^{(n)}$ at level $\xi_{k+1}^{(n)}$, and the new leaf (which is,
  as all the other leaves,  at
  level $t$) is added to the vector recording the distinguished 
  vertices. Formally, using the grafting
  procedure~\eqref{eq:graft}
  we set:
  \begin{equation}
   \label{eq:Tk+1}
   (\bT_{k+1}^{(n)}, \bv_{k+1}^{(n)})=(\bT_k^{(n)},  \bv_{k}^{(n)})
   \circledast_{K_k, \, \xi_{k+1}^{(n)}}^{\varepsilon_{k+1}}
   \left[0,t-\xi_{k+1}^{(n)}\right]\cdot
  \end{equation}
\end{itemize}

\begin{figure}[ht]
\includegraphics[scale=0.5]{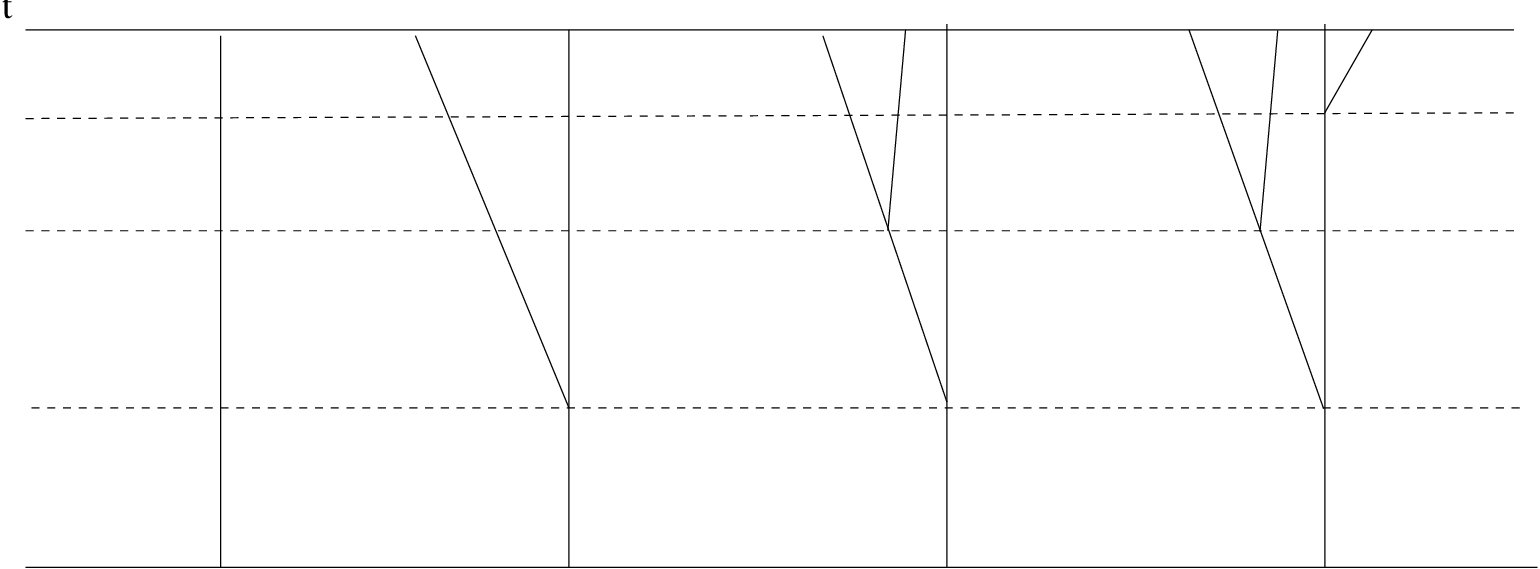}
\caption{The trees $\bT_1^{(4)}$, $\bT_2^{(4)}$, $\bT_3^{(4)}$ and $\bT_4^{(4)}$ obtained from the sequences $(K_1=1,K_2=1,K_3=2)$ and $(\varepsilon_1={\rm g},\varepsilon_2={\rm d},\varepsilon_3={\rm d})$. The dashed lines represent the levels $\xi_1^{(4)},\xi_2^{(4)},\xi_3^{(4)}$.}
\end{figure}

By construction,  we get that  $(\bT_k^{(n)}, \bv_k^{(n)})$ belongs
to  $\TPk$  
for  all
$k\in  \{1,   \ldots,  n\}$.    To  simplify   the  notations,   we  set
$\bT_n=(\bT_n^{(n)},   \bv_n^{(n)})$.    
\medskip

Recall  that  for a  rooted  tree  $T$,  $\length^T$ denotes  its  length
measure;  and we  simply write  $\length$  when there  is no  ambiguity.
The next lemma relates  the distributions of
$\bT_n$ and of $\bT_{n+1}$; its proof is given in the next section.
Recall that  $\TPn$ is a subset  of $\TDn$. 

\begin{lem}
\label{lem:Tn*=Tn+1}
Let  $t\geq 0$. Assume that the probability distribution $\nu$ has a positive
density $\fd$ with respect to the Lebesgue measure on $[0, t]$. For  $n\in
\N^*$, $G$ a  measurable non-negative  function defined
 on  $\TDo$, and $\varepsilon$ a random
 variable uniformly distributed on $\{{\rm g}, {\rm d}\}$ and independent of $\bT_n$, we have:
\begin{equation}
  \label{eq:Tn*=Tn+1-0}
  \E\left[\int_{\bT_n}\length(\rd x)\, \fd \bigl(H(x)\bigr)\,
    G\Bigl(\bT_n\circledast_x^\varepsilon\bigl[0,t-H(x)\bigr]\Bigr)\right]
  =\frac{n+1}{2}\, \E\left[G\big(\bT_{n+1}\big)\right].
\end{equation}
\end{lem}

\begin{rem}
  \label{rem:comment}
  We  comment  on the  left-hand  side  of \eqref{eq:Tn*=Tn+1-0}.  First
  notice the grafting  on $\bT_n$ is oriented, which justify  to build the
  discrete tree  $\bT_n$ as  a planar  tree. Second,  we check  that the
  integral
  $\ci=\int_{\bT_n}\length(\rd              x)\,             \fd(H(x))\,
  G\big(\bT_n\circledast_x^\varepsilon[0,t-H(x)]\big)$ is a non-negative
  random  variable.   Recall  that  $\bT_n=(\bT_n^{(n)},  \bv_n^{(n)})$.
  Then, we can write $\ci$ as follows:
\[
 \ci= \sum_{k=1}^n \int_{\xi_{k-1}^{(n)}}^t \rd h\, \fd(h) \,
  G\bigl(\bT_n\circledast_{k,h}^\varepsilon [0,t -h]\bigr),
\]
with      the      convention       that      $\xi_{0}^{(n)}=0$.
Therefore, using the continuity of the  grafting
function, see Lemma~\ref{lem:graft2-gd-cont},  we  obtain   that  $\ci$  is  a
non-negative real-valued  random variable,  and thus its  expectation is
well defined.
\end{rem}

\subsubsection{Proof of Lemma~\ref{lem:Tn*=Tn+1}}
The proof is  based on two technical lemmas.  We  first consider the case
$t=1$ and $\nu$  the uniform distribution on $[0,1]$.
Let us denote by  $\bTun$ for $\bT_n$ when $\nu$ is the uniform
distribution on $[0,1]$.

\begin{lem}
\label{lem:Tn*=Tn+1_unif}
For  $n\in
\N^*$, $G$  a measurable non-negative  functional defined
 on  $\TDo$, and $\varepsilon$ a $\{{\rm g}, {\rm d}\}$-valued uniform random
 variable independent of $\bTun$, we have:
\begin{equation}
  \label{eq:Tn*=Tn+1}
  \E\left[\int_{\bTun}\length(\rd x)\,
    G\Bigl(\bTun\circledast_x^\varepsilon\bigl[0,t-H(x)\bigr]\Bigr)\right]
  =\frac{n+1}{2}\, \E\left[G\big(\bTunn\big)\right].
\end{equation}
\end{lem}

\begin{rem}
From \eqref{eq:Tn*=Tn+1}, we see $\frac{n+1}{2}$ is just the mean length of $\bTun$.
\end{rem}

\begin{proof}
  To simplify notation,  we write $\bT_n$ for $\bTun$.  We  give a proof
  by  induction.  For  $n=1$,  this  is  a  direct  consequence  of  the
  construction     of      $\bT_2=\bT_2^{(2)}$     from      $\bT_1^{(2)}=\bT_1=[0,1]$     given
  by~\eqref{eq:Tk+1}. \medskip

Let $n\in \N^*$ and assume that \eqref{eq:Tn*=Tn+1} holds for $n$ replaced
by any  $k\in \{1, \ldots,  n-1\}$.
We will  use for  the proof  a special  representation of  planar binary
trees. Let $T$  be a compact planar binary tree  rooted at $\root$, with
 all leaves at height 1; in particular the tree $T$ has a finite number
 of leaves. If $T$ has at least two leaves, since it is compact with a
 finite number of leaves,  there exists a
lowest branching vertex, say $x$. 
We set $h=H(x)$ and   $\tilde T^{\rm g}$ (resp. $\tilde  T^{\rm d}$) the
left (resp. right) subtree above $x$. In our settings, we have:
\[
h=  H(x) \quad\text{and}\quad
  T=\bigl(\lb\root,x\rb\circledast _x\tilde T^{\rm
    g}\bigr)\circledast_x^{\rm d}\tilde T^{\rm d}
  =\bigl(\lb\root,x\rb\circledast _x\tilde T^{\rm
    d}\bigr)\circledast_x^{\rm g}\tilde T^{\rm g},
\]
where, with a  slight abuse of language (see Lemma~\ref{lem:erase-cont}
for formal justification), one has  removed the vertex $x$
from  the distinguished  vertices after  the graftings.   For convenience,  we
consider      the       scaled      left      and       right      trees
$T^{\rm     g     }=     (1-h)^{-1}    \tilde     T^{\rm     g}$     and
$T^{\rm d }=  (1-h)^{-1} \tilde T^{\rm d}$  (recall \eqref{eq:scale} for
the definition of a scaled tree), so  that $T^{\rm g }$ and $T^{\rm d }$
are rooted bounded  binary planar trees with all their  leaves at height
1.  We call  $(h,T^{\rm  g },  T^{\rm  d })$  the  decomposition of  $T$
according to its lowest branching vertex.

\medskip

Let  $(\xi^{(n+1)}_1,  \bT_{n+1}^{\rm{g}}, \bT_{n+1}^{\rm{d}})$  be  the
decomposition  of $\bT_{n+1}$  according to  its lowest  branching vertex
(which is indeed at height  $\xi^{(n+1)}_1$ by construction).  Denote by
$I_{n+1}$  the  number  of  leaves of  $  \bT_{n+1}^{\rm{g}}$.  Using  a
P\'{o}lya urn starting with two balls of color $\rm{g}$ and $\rm{d}$, we
get that,  by construction, $I_{n+1}$  is the  number of balls  of color
$\rm{g}$  in the  urn  after $n$  draws. Thus  $I_{n+1}$  is uniform  on
$\{1, \ldots,  n\}$ and independent  of $\xi^{(n+1)}_1$. Notice  that if
$U$ is  a uniform random  variable on  $[0,1]$, for every  $h\in (0,1)$,
conditionally given $\{U\ge h\}$,  the random variable $(1-h)^{-1}(U-h)$
is   still  uniformly   distributed   on  $[0,1]$.   This  gives   that,
conditionally  on  $\{\xi^{(n+1)}_1=h\}$  and $\{I_{n+1}=i\}$,  the  two
trees $\bT_{n+1}^{\rm{g}}$ and  $\bT_{n+1}^{\rm{d}}$ are independent and
distributed respectively as $\bT_i$ and $\bT_{n+1-i}$.  \medskip

We consider  a  measurable non-negative
functional $G$ defined on the space of rooted  compact  binary planar
trees with a
finite number of leaves, all of them at height 1 of the form:
\begin{equation}
   \label{eq:F=ggg}
G(T)=g_1(h)\, g_2 ( T^{\rm
  g }) \, g_3(T^{\rm d }),
\end{equation}
where the $g_i$'s are measurable non-negative functionals and  $(h,T^{\rm
  g }, T^{\rm d })$ is the decomposition of $T$ according to its lowest
branching vertex. Setting  $f_j(i)= \E\left[g_j(\bT_i)\right]$ for $j\in
\{2, 3\}$, we have since $\xi_1^{(n+1)}$ is distributed according to a
$\beta(1,n)$ distribution (as the maximum of $n$ independent uniform
random variables):
\begin{align}
\nonumber
  \E\left[G(\bT_{n+1})\right]
&  = \left(\int_0^1 g_1(h) \, n (1-h)^{n-1}\, \rd h \right)
  \, \inv{n}\sum_{i=1}^n f_2(i) f_3(n+1-i)\\
\label{eq:FTn+1}
&  = \left(\int_0^1 g_1(h) \, (1-h)^{n-1}\, \rd h \right)
  \, \sum_{i=1}^n f_2(i) f_3(n+1-i).
\end{align}

\medskip

On the other hand,  let  $(\xi_1^{(n)}, \bT_{n}^{\rm{g}}, \bT_{n}^{\rm{d}})$ be the
decomposition of $\bT_{n}$ according to its lowest
branching vertex. Let $x\in \bT_n$ and set $h=H(x)$.
\begin{itemize}
\item      If      $h<\xi_1^{(n)}$,       the      decomposition      of
  $\bT_n\circledast_x^{\rm{g}}[0,1-h]$ according to its lowest branching
vertex is given  by $(h, [0,1] , (1-h)^{-1} \bT'_n)$  where $\bT'_n$ is
  as  the  tree $\bT_n$  but  for  its  lowest  branch whose  length  is
  $\xi_1^{(n)}-h$ instead  of $\xi_1^{(n)}$.  Notice that the  shapes of
  the tree  $\bT'_n$ and $ \bT_n$  are the same. Then  using again
  the property of  conditioned uniform random variables,  we deduce that
  conditionally on $\{\xi_1^{(n)}\ge h\}$, the tree $ (1-h)^{-1} \bT'_n$
  is distributed as $\bT_n$. Thus, we get:
\begin{equation}
   \label{eq:0Xg}
   \E\left[\int _{\bT_n} \ind_{\{H(x)< \xi_1^{(n)}\}}\,
     \length(\rd x) \, G\bigl(\bT_n\circledast_x^{\rm{g}}[0,1-h]\bigr) \right]
   = \E\left[\int_0^{\xi_1^{(n)}} g_1(h) \, \rd h \right] \, f_2(1)     f_3(n).
\end{equation}

By symmetry, we have:
\begin{equation}
   \label{eq:0Xd}
   \E\left[\int _{\bT_n} \ind_{\{H(x)<\xi_1^{(n)}\}}\,  \length(\rd x)
   \,   G\bigl(\bT_n\circledast_x^{\rm{d}}[0,1-h]\bigr) \right]
   = \E\left[\int_0^{\xi_1^{(n)}} g_1(h) \, \rd h \right] \, f_2(n)
     f_3(1).
\end{equation}

\item For $x\in \tilde \bT_{n}^{\rm{g}}$,
   the
decomposition of $\bT_n\circledast _x^\varepsilon [0,1-h]$ according to its lowest
branching vertex is given by $(\xi_1^{(n)},(1-h)^{-1}\Tau', \bT_{n}^{\rm{d}})$,  where
$\Tau'=\tilde \bT_{n}^{\rm{g}} \circledast_x^\varepsilon[0,
1-h]$. Notice that the length measure on the tree
$\bT_{n}^{\rm{g}}$ is obtained by scaling by  $(1-\xi_1^{(n)})^{-1}$ the
length measure on $\bT_n$ restricted to $\tilde\bT_{n}^{\rm{g}}$.
We deduce that:
\begin{multline*}
   \E\left[\int _{\bT_n} \ind_{\{x\in \tilde \bT_{n}^{\rm{g}}\}}\,
     \length^{\bT_n}(\rd x)\, G\bigl(\bT_n\circledast _x^{\varepsilon}[0,1-h]\bigr)
     \right]\\
   \begin{aligned}
     &   =  \E\left[g_1(\xi_1^{(n)}) g_3(\bT_{n}^{\rm{d}})  \int
       _{\bT_n }
       \ind_{\{x\in \tilde \bT_{n}^{\rm{g}}\}}\,  \length^{\bT_n}(\rd x)\,
    g_2\Bigl((1-\xi_1^{(n)})^{-1}\bigl(\tilde \bT_{n}^{\rm{g}}
   \circledast_x^{\varepsilon}[0,1-H(x)]\bigr)\Bigr)\right] \\
 & =  \E\left[  g_1(\xi_1^{(n)}) g_3(\bT_{n}^{\rm{d}})  \int
   _{\bT_n^{\rm{g}} } (1-\xi_1^{(n)})
   \length^{\bT_n^{\rm g}}(\rd y) \, g_2\Bigl( \bT_{n}^{\rm{g}} \circledast_y^{\varepsilon}\bigl[0, 1
   -H(y)\bigr]\Bigr)  \right] \\
& =  \E\left[ (1-\xi_1^{(n)}) g_1(\xi_1^{(n)})\right]
    \inv{n-1}\sum_{i=1}^{n-1} f_3(n-i)
    \E\left[  \int _{\bT_i}  \length^{\bT_i}(\rd y) \,
        g_2\Bigl(\bT_{i} \circledast_y^{\varepsilon}\bigl[0,1-H(y)\bigr]\Bigr)\right].
\end{aligned}
\end{multline*}
where we used the distribution of $(\bT_n^{\rm g},\bT_n^{\rm d})$ conditionally on $\xi_1^{(n)}$ and $I_{n}$ for the last equality.
Using that, by induction, \eqref{eq:Tn*=Tn+1} holds for $n=i$, we get:
\begin{multline}
   \label{eq:Xg}
  \E\left[\int _{\bT_n} \ind_{\{x\in \tilde \bT_{n}^{\rm{g}}\}}\,
     \length(\rd x)\, G\bigl(\bT_n\circledast _x^{\varepsilon}[0,1-h]\bigr)
     \right]\\
 \begin{aligned}
& =  \E\left[ (1-\xi_1^{(n)}) g_1(\xi_1^{(n)})\right]
\inv{n-1}\sum_{i=1}^{n-1} \frac{i+1}{2}\, f_2(i+1)\,  f_3(n-i)\\
& =  \E\left[ (1-\xi_1^{(n)}) g_1(\xi_1^{(n)})\right]
 \inv{n-1}\sum_{i=2}^{n} \frac{i}{2}\, f_2(i)\,  f_3(n-i+1).
 \end{aligned}
\end{multline}

\item By symmetry, for $x\in \tilde
\bT_{n}^{\rm{d}}$, we get:
\begin{multline}
   \label{eq:Xd}
  \E\left[\int _{\bT_n} \ind_{\{x\in \tilde \bT_{n}^{\rm{d}}\}}\,
     \length(\rd x)\, G(\bT_n\circledast _x^{\varepsilon}[0,1-h])
     \right]\\
 \begin{aligned}
& =  \E\left[ (1-\xi_1^{(n)}) g_1(\xi_1^{(n)})\right]
 \inv{n-1}\sum_{i=2}^{n} \frac{i}{2}\, f_3(i)\,  f_2(n-i+1)\\
&  =  \E\left[ (1-\xi_1^{(n)}) g_1(\xi_1^{(n)})\right]
 \inv{n-1}\sum_{i=1}^{n-1} \frac{n-i+1}{2}\, f_2(i)\,  f_3(n-i+1).
 \end{aligned}
\end{multline}
 \end{itemize}
 Summing     \eqref{eq:0Xg}      times     $\P(\varepsilon=\rm{g})=1/2$,
 \eqref{eq:0Xd}  times  $\P(\varepsilon=\rm{d})=1/2$, \eqref{eq:Xg}  and
 \eqref{eq:Xd},   and   using   that  $\xi_1^{(n)}$   has   distribution
 $\beta(1, n-1)$ so that:
 \[
   \E\left[\int_0^{\xi_1^{(n)}}
     g_1(h) \, \rd h \right] =
\inv{n-1} \E\left[ (1-\xi_1^{(n)}) g_1(\xi_1^{(n)})\right]
= \int_0^1  g_1(h) (1-h)^{n-1}\, \rd h,
\]
we deduce that:
 \begin{multline*}
 \E\left[\int_{\bT_n}\length(\rd x)\,
   G\Bigl(\bT_n\circledast_x^{\varepsilon }\bigl[0,1-H(x)\bigr]\Bigr)\right]\\
     =\left(\int_0^1  g_1(h) (1-h)^{n-1}\, \rd h
     \right)
      \sum_{i=1}^n \frac{n+1}{2} \, f_2(i) f_3(n+1-i).
\end{multline*}
Thanks to \eqref{eq:FTn+1}, we deduce that \eqref{eq:Tn*=Tn+1} holds for
 $G$ given by
\eqref{eq:F=ggg}. Then use a monotone class argument to conclude
that \eqref{eq:Tn*=Tn+1} holds for
 any measurable non-negative $G$. This concludes the proof by induction.
\end{proof}

 We now  consider  $t\geq 0$ and assume that the probability distribution $\nu$ has a positive
density $\fd$ with respect to the Lebesgue measure on $[0, t]$.
Let $F$  denote the  cumulative distribution function  of $\nu$.  By the
assumptions on $\fd$, $F$ is a bijection  from $[0,t]$ onto $[0,1]$ and its
inverse  $F^{-1}$  is  continuous.   For  a  compact  rooted  real  tree
$(T,d,\root)$, we define:
\begin{align*}
\forall x\in T, & \ H^\fd(x)=F^{-1}(H(x)),\\
\forall x,y\in T, & \ d^\fd(x,y)=H^\fd(x)+H^\fd(y)-2H^\fd(x\wedge y).
\end{align*}
  The scaling map $R^\fd: (T,d,\root)\longmapsto (T,d^\fd,\root)$ is then
  well-defined from $\{T\in \TK:\ H(T)\le 1\}$ to $\TK$. We shall now prove it is continuous.

\begin{lem}
   \label{lem:def-Rf}
  The map $R^\fd$ from $\{T\in \TK,\ H(T)\le 1\}$ to $\TK$   is  continuous.
\end{lem}
\begin{proof}
Let $\varepsilon>0$. As $F^{-1}$ is uniformly continuous with our assumptions, there exists $\delta>0$ such that, for every $x,y\in[0,1]$:
\[
|x-y|<\delta \Longrightarrow \bigl|F^{-1}(x)-F^{-1}(y)\bigr|\le \frac{\varepsilon}{2}\cdot
\]
Let $T,T'\in\TK$ with $H(T)\le 1$ and $H(T')\le 1$  such that $\dgh(T,T')<\delta/8$. Then,  there exists a
correspondence $\cR$  between (elements in the  equivalence classes) $T$
and $T'$ such that $\dist(\cR)\le 2\dgh(T,T')+\delta/4<\delta/2$.

For every $(x,x'),(y,y')\in\cR$, we have:
\begin{align*}
\bigl|d^\fd(x,y)-d'{}^\fd(x',y')\bigr| & =\bigl| H^\fd(x)+H^\fd(y)-2H^\fd(x\wedge
                               y)-H^\fd(x')\\
  & \hspace{4cm} -H^\fd(y')+2H^\fd(x'\wedge y')\bigr|\\
& \le \Bigl|F^{-1}\bigl(H(x)\bigr)-F^{-1}\bigl(H(x')\bigr)\Bigr|+\Bigl|F^{-1}\bigl(H(y)\bigr)-F^{-1}\bigl(H(y')\bigl)\Bigr|\\
&\hspace{4cm}+2\Bigl|F^{-1}\bigl(H(x\wedge y)\bigr)-F^{-1}\bigl(H(x'\wedge y')\bigr)\Bigr|.
\end{align*}
As $(x,x')\in\cR$,  we have $\bigl|H(x)-H(x')\bigr|\le \dist(\cR)<\delta$  and consequently,
$\Bigl|F^{-1}\bigl(H(x)\bigr)-F^{-1}\bigl(H(x')\bigr)\Bigr|<\varepsilon/2$.                 Similarly,
we have $\Bigl|F^{-1}\bigl(H(y)\bigr)-F^{-1}\bigl(H(y')\bigr)\Bigr|<\varepsilon/2$.  We also have:
\begin{align*}
  \bigl|H(x\wedge y)-H(x\wedge y)\bigr|
  & =\frac{1}{2}\bigl|H(x)+H(y)-d(x,y)-H(x')-H(y')+d'(x',y')\bigr|\\
& \le \frac{1}{2}\bigl|H(x)-H(x')\bigr|+\frac{1}{2}\bigl|H(y)-H(y')\bigr|+\frac{1}{2}\bigl|d(x,y)-d'(x',y')\bigr|\\
  & \le \frac{3}{2} \dist(\cR)\\
  &<\delta.
\end{align*}
This gives $\Bigl|F^{-1}\bigl(H(x\wedge y)\bigr)-F^{-1}\bigl(H(x'\wedge y')\bigr)\Bigr|<\varepsilon/2$.

To conclude,  we have $\dist\!\!\!{}^\fd(\cR)<2\varepsilon$  which implies
that $d^\fd_{\mathrm{GH}}(T,T')<\varepsilon$. This gives the continuity of
the map $R^\fd$.
\end{proof}

We now prove Lemma~\ref{lem:Tn*=Tn+1}.  Recall that $\bT_n$ denotes the trees
constructed with the probability measure $\nu(\rd x)=\fd(x)\, \rd x$ and
$\bTun$ the trees  constructed with the uniform  distribution on $[0,1]$
as studied in the first step.  By construction, for all $n\in \N^*$, the
random variables $R^\fd(\bTun)$ and  $\bT_n$ have the same distribution.
Notice also that, for every  $T\in\TK$ and every non-negative measurable
function $g$ on $\R_+\times \TL$, we have:
\[
\int _T \length^T(\rd y)\, g\bigl(H(y), T\bigr)=\int_{R^\fd(T)}
\length^{R^\fd(T)}(\rd x)\,
\fd\bigl(H^\fd(x)\bigr)\, g\bigl(H^\fd(x), R^\fd(T)\bigr).
\]

Let $G$ be a measurable non-negative  functional defined on the space of
rooted compact binary  planar trees with a finite number  of leaves, all
of them at height $t$. We first have:
\begin{multline*}
\E\left[\int_{\bT_n}\length^{\bT_n}(\rd x)\, \fd\bigl(H(x)\bigr) \,
  G\Bigl(\bT_n\circledast_x^\varepsilon \bigl[0,t-H(x)\bigr]\Bigr)\right]\\
\begin{aligned}
&= \E
\left[\int_{R^{\fd}(\bTun)}\length^{R^\fd(\bTun)}(\rd x)\, \fd\bigl(H^\fd(x)\bigr)\,
  G\Bigl(R^\fd(\bTun)\circledast_x^\varepsilon\bigl[0,t-H^\fd(x)\bigr]\Bigr)\right]\\
&=\E \left[\int_{\bTun}\length^{\bTun}(\rd y)\,
\,  G\circ R^\fd\Bigl(\bTun\circledast_y^\varepsilon\bigl[0,1-H(y)\bigr]\Bigr)\right].
\end{aligned}
\end{multline*}
Applying Lemma~\ref{lem:Tn*=Tn+1}, and then that $R^\fd(\bTunn)$ and
$\bT_{n+1}$ have the same distribution, we get the result.

\subsubsection{An infinite tree with no leaves}
\label{sec:tree_ske}

Let $\fint$ be a positive locally integrable function on
$[0,+\infty)$. Let $S$ be a Poisson point measure on $\R_+$ with
intensity $\fint(t)\, \rd t$. We denote by $(\xi_i, i\ge 1)$ the
increasing sequence of the atoms of $S$ and by $N$ the process $\bigl(N_t=S\bigl([0,t]\bigr),\ t\ge 0\bigr)$.

Let $(\varepsilon_n,n\ge  1)$ be independent random  variables uniformly
distributed  on  $\{{\rm  g},{\rm  d}\}$   and  let  $(K_n,n\ge  1)$  be
independent random variables uniformly distributed on $\{1,2,\ldots,n\}$
respectively, all  these variables being independent  and independent of
$S$.

We  define  a tree-valued  process  $(\fT_t,t\ge  0)$ where,  for  every
$t\ge 0$, the random tree $\fT_t$ has height $t$ and $N_t+1$ leaves, all
of them  at  height  $t$.  Before  going into this  construction, we
first define a growing procedure  on rooted $n$-pointed trees
for $n\in \N^*$:
\begin{equation}
   \label{eq:growth}
   \mathrm{Growth}_n((T,  \bv),h)\in \TDn
   \quad\text{with}\quad
   (T, \bv)\in \TDn
   \quad\text{and}\quad
   h\in \R_+, 
 \end{equation}
 as the tree  obtained by grafting on all the  distinguished vertices of
 $T$, but the root (that is,  on $\bv^*=(v_1, \ldots, v_n)$) a branch of
 length  $h$, distinguishing  the new  leaves with  the order  naturally
 induced by $\bv^*$  and removing the vertices $\bv^*$ from  the list of
 distinguished  vertices.   This function is formally defined in
Section~\ref{sec:growth}, see also Lemma~\ref{lem:growth-meas} for its
measurability. 
\medskip

We can now construct the process $(\fT_t,t\ge 0)$ inductively.
For  $0\le t\le \xi_1$, we set
$\fT_{t}=\bigl([0,t],(0,t)\bigr)$ and $N_t=0$.

Let  $n\in \N^*$  and assume  that  $(\fT_{\xi_n},\bv_n)$ is  a tree  of
height  $\xi_n$ with  $n$  leaves, all  of them  at  height $\xi_n$  and
distinguished  (\emph{i.e.}   the  vector  $\bv_n$   is  composed  of  the   root  of
$\fT_{\xi_n}$  and  all its  leaves).  Then,  we  define the  process  on
$(\xi_n,\xi_{n+1}]$ by setting, for every $t\in (\xi_n,\xi_{n+1}]$:
\[
  \fT_{t}=\mathrm{Growth}_n(\fT_{\xi_n},t-\xi_n)
  \circledast_{K_n,\xi_n}^{\varepsilon_n}[0,t-\xi_n]
  \quad\text{and}\quad
  N_t=n.
\]

Standard properties of Poisson processes give the following result.

\begin{lem}\label{lem:T_ske distribution}
For every $n\ge 1$ and every $t>0$, conditionally given $N_t=n-1$, the
tree $\fT_t$ is distributed as the tree $\bT_n$ of Section \ref{sec:bT}
associated with the density $\fd$ on $[0, t]$ given by:
\begin{equation}
   \label{eq:def-fd-fi}
  \fd(u)=\frac{\fint(u)}{F(t)}\ind_{[0,t]}(u)
  \quad\text{with}\quad F(t)=\int_0^t \fd(u)\, \rd u.
\end{equation}
\end{lem}

We now view the tree $\fT_t$ as a real-tree of $\TL$ (we forget about
the distinguished leaves which is a continuous operation thanks to
Lemma~\ref{lem:erase-cont}). It  is
easy to  see that the process  $(\fT_t,t\ge 0)$ 
satisfies the Cauchy property in      $\TL$      as
$r_s(\fT_t)=r_s(\fT_{t'})$ for every  $s\le t\le t'$. Thus
this sequence converges a.s. in $\TL$, and we write:
\begin{equation}\label{eq:def_Tske}
\fTs=\lim_{t\to+\infty}\fT_t.
\end{equation}
The  tree $\fTs$  is  a $\TL$-valued random variable which has  no
leaves. 
The tree $\fTs$  will serve  as a backbone  for the description  of the
genealogical tree of the conditioned CSBP.

\medskip

We present now an ancillary result which is interesting by itself; it is
a consequence  of Lemma~\ref{lem:Tn*=Tn+1} on two  tree-valued processes
that have the same one-dimensional marginal.

We first consider the process $(\fT_t,t\ge 0)$ associated with the
intensity $\fint\equiv 1$, that is,  $\fint(t)=1$ for all $t\geq 0$.
Then we  construct a sequence  $\bt=(\bt_n, n\ge 1)$ of  increasing real
trees, with  $\bt_n\in \TKn$ for every  $n\ge 1$, all of  them of height
1.  Let  $(\varepsilon_k,  k\ge  1)$  be  independent  random  variables
uniformly distributed on  $\{{\rm g},{\rm d}\}$. We  define the sequence
$\bt$ by induction by setting first $\bt_1=\bigl([0,1],(0,1)\bigr)$. Let $n\geq 1$
and assume  that $(\bt_n,\bv_n)$ is a  tree of $\TKn$ with  height 1 and
with $n$ leaves all of them at height $1$.  Conditionally given $\bt_n$,
let $V_{n+1}$ be a random  element on $\bt_n$ uniformly chosen according
to the length measure; that is $V_{n+1}$ is distributed according to the
measure  $c_n\,  \length$,  with  $\length$  the  length
measure on $\bt_n$ and the normalization $c_n=1/\length(\bt_n)$.
Notice that $V_{n+1}$ is a.s. not a leaf nor the root of $\bt_n$. Then we set:
\[
\bt_{n+1}=\bt_n\circledast_{V_{n+1}}^{\varepsilon_{n+1}}\bigl[0,1-H(V_{n+1})\bigr].
\]
In particular, for every measurable nonnegative function $G$, we have:
\begin{equation}\label{eq:def_bt}
\E\bigl[G(\bt_{n+1})|\, \bt_1, \ldots, \bt_n, \varepsilon_{n+1}\bigr]=\int_{\bt_n}\frac{\length(\rd
  x)}{\length(\bt_n)}\, G\Bigl(\bt_n\circledast
_x^{\varepsilon_{n+1}}\bigl[0,1-H(x)\bigr]\Bigr).
\end{equation}

Recall     the     measurable     function     $\tilde     N_t$     from
Definition~\eqref{eq:def-tildeN} which records the number of vertices at
level       $t$       of        a       tree       without       leaves:
$     \tilde     N_t(T)=\Card\Bigl(\bigl\{     x\in     T\,     \colon\,
H(x)=t\bigr\}\Bigr)$. Let us consider the continuous (see Lemma~\ref{lem:erase-cont}) canonical projection $\Pi_n^\circ : \TKn\longrightarrow \TK$  defined by $\Pi_n^\circ(\bt,\bv)=\bt$.
\begin{prop}
  Let $n\geq 1$ and  $\fint \equiv 1$. For all
  measurable non-negative functional $G$ defined on $\TDn$,
 we have, with $\length$ the length measure on $\bt_n$:
\begin{equation}
   \label{eq:bt=bT}
     \E\left[G(\fT_1)\Bigm| N_1=n-1\right]
    =\frac{2^{n-1}}{n!} \E\left[G(\bt_n) \,
        \prod_{k=1}^{n-1} \length(\bt_{k})\right],
    \end{equation}
    and for all  measurable non-negative functional $G$ defined on $\TK$
    (or on $\TL$):
\begin{equation}
   \label{eq:bt=bT-bis}
    \E\left[G(\fTs_1)\Bigm| \tilde N_1(\fTs)=n\right]
    =   \E\left[G\circ \Pi^\circ_n(\fT_1)\Bigm| N_1=n-1\right].
  \end{equation}
\end{prop}

\begin{proof}
  By      construction,      we      have     that      the      process
  $\left(\bigl(\fTs_t,  \tilde   N_t(\fTs)\bigr),  t\geq   0\right)$  is
  distributed                as               the                process
  $\Bigl(\bigl(\Pi^{\circ}_{N_t+1}(\fT_t),       N_t+1\bigr),      t\geq
  0\Bigr)$. This gives~\eqref{eq:bt=bT-bis}.

\medskip

We   now  prove~\eqref{eq:bt=bT}   by   induction.    Thanks  to   Lemma
\ref{lem:T_ske  distribution}, conditionally  given $N_1=n-1$,  the tree
$\fT_1$   is    distributed   as   $\bTun$.    For    $n=1$,   we   have
$\bTun=\bt_1=([0,1],(0,1))$ hence Equation  \eqref{eq:bt=bT} holds.  Let
us  suppose that~\eqref{eq:bt=bT}  holds  for some  $n\ge 1$.   Applying
Lemma \ref{lem:Tn*=Tn+1}, one gets:
\[
  \E\bigl[G(\bT^{\mathrm{unif}}_{n+1})\bigr]=\frac{2}{n+1}\E\left[\length(\bTun)
    \int_{\bTun}\frac{\length^\bTun(\rd x)}{\length^\bTun(\bTun)
  }G\Bigl(\bTun\circledast^\varepsilon _x\bigl[0,1-H(x)\bigr]\Bigr)\right].
\]
Now we apply the induction assumption for the right-hand side of the previous equation to get:
\begin{align*}
  \E\bigl[G(\bT^{\mathrm{unif}}_{n+1})\bigr]
  & =\frac{2}{n+1}\frac{2^{n-1}}{n!}\left[\length^{\bt_n}(\bt_n)\int_{\bt_n}
    \frac{\length^{\bt_n}(\rd x)}{\length^{\bt_n}(\bt_n)
    }G\Bigl(\bt_n\circledast ^{\varepsilon_{n+1}}_x\bigl[0,1-H(x)\bigr]\Bigr) \, \prod_{k=1}^{n-1} \length^{\bt_n}(\bt_{k})\right]\\
  & =\frac{2^n}{(n+1)!}\E\left[G(\bt_{n+1})\, \prod_{k=1}^{n} \length^{\bt_{n+1}}(\bt_{k})\right]
\end{align*}
by definition of the tree $\bt_{n+1}$ and  by \eqref{eq:def_bt}. This
gives that~\eqref{eq:bt=bT} holds with $n$ replaced by $n+1$. This
concludes the proof by induction.
\end{proof}

\subsection{The $n$-leaves decomposition of the Brownian CRT}
\label{sec:backbone}
The decomposition of a (sub)critical Brownian CRT ${\ct}$ according to a
spine $\lb\emptyset, x\rb$,  where $x\in\ct$ is a leaf  picked at random
at level $t>0$, that is according  to the local time $\Lambda_t(\rd x)$,
is given  in Theorem  4.5 in  \cite{DLG05}.  In our  setting, it  can be
rephrased  in the  next theorem.   Notice that,  for $t>0$,  the (planar
discrete) $1$-pointed tree $[0, t]\in \TLu$ denotes the segment $[0, t]$
endowed  with  the  Euclidean  distance,  with  the  root  $0$  and  the
distinguished vertex $t$.  Recall that the grafting operation $\graft_n$
on a $n$-pointed discrete tree  of trees formalized by atoms of a
Poisson   point   measure   $\cm$   has   been   intuitively   presented
in~\eqref{eq:rev-graft-n} (or~\eqref{eq:def-graftn-intuition}) and
formally defined       in         Section~\ref{sec:discrete-app},
see~\eqref{eq:def-graftn-TK} therein,  using the theoretical background
of Section~\ref{sec:graft}, so that $\graft_1([0,t],\cm)$ in~\eqref{eq:Bismut0}
 below is a well defined $\TLu$-valued random variable. 



\begin{theo}[\cite{DLG05}]
  \label{Thmbis}
  Let $\beta>0$,
  $\theta\geq  0$ and  $t>0$. Let $\cm$ be under $\E$ a Poisson measure
  with intensity $2\beta\,\ind_{[0,t]}(s)\rd s\N^\theta[\rd\ct]$. For every  non-negative  measurable
  functional $F$ on $\TLu$ (or $\TKu$), we have, with $\root$ the root of $\ct$:
\begin{equation}
  \label{eq:Bismut0}
\N^{\theta}\left[\int_{\ct}\Lambda_t(\rd v) \, F\bigl(\ct,(\root,v)\bigr)
\right]
=\expp{-2\beta\theta t}
  \E\left[F\bigl(\graft_1([0,t],\cm)\bigr)\right].
\end{equation}
\end{theo}

We extend this result to the super-critical case $\theta<0$.

\begin{cor}[One-leaf decomposition]
  \label{cor:Thmbis}
  Let $\beta>0$,
  $\theta\in \R$ and  $t>0$. Let $\cm$ be under $\E$ a Poisson measure
  with intensity $2\beta\, \ind_{[0,t]}(s)\rd s\N^\theta[\rd\ct]$. For every  non-negative  measurable
  functional $F$ on $\TLu$, Equation~\eqref{eq:Bismut0} holds.
 \end{cor}

\begin{proof}
Let $(T,\bv)\in \TLu$ with $\bv=(\root,v)$. We denote by $(T_i^\circ,i\in I)$ the connected components of the set $T\setminus\lb\root,v\rb$. For every $x\in T$, there exists a unique $x_i\in T$ such that $\cap_{x\in T_i^\circ}\lb\root,x\rb=\lb\root,x_i\rb$ and we set $T_i=T_i^\circ\cup\{x_i\}$ viewed as a real tree rooted at $x_i$. Then we define the point measure $\cmt(T,\bv)$ on $\R_+\times \TL$ by:
\[
\cmt(T,\bv)=\sum_{i\in I}\delta_{H(x_i),T_i}.
\]
This application is well defined according to
Corollary~\ref{cor:Span-Ma-1}. Even if we shall not use it as such, let us
mention that a.s.\ $\cmt(\graft_1([0,t],\cm))=\cm$; this  can be easily
deduced from Proposition~\ref{prop:meas(T)-mesurable}. 

 We first prove \eqref{eq:Bismut0}
  for functionals $F$ of the form:
\begin{equation}
   \label{eq:F=expo}
  F(T,\bv)=\expp{-\langle\Phi,\cmt(T,\bv)\rangle},
\end{equation}
  where  $(T, \bv)\in  \TLu$, $\Phi$  is  a continuous  non-negative
  function  with   bounded  support  defined on $\R_+\times   \TLz$  (with
  $\TLz=\TL\setminus\{\Tz\}$ where  $\Tz\in \TL$ is the  tree reduced to
  its root, see Section \ref{sec:def-M}).
  \medskip

For simplicity, we write $(\ct,v)$ for the $1$-pointed tree
$(\ct, (\root, v))$. Let $\theta>0$.  Using \eqref{eq:abs-cont2}, we have for every $s>t$ that:
\[
   \N^{-\theta}\left[\int_{\ct}\Lambda_t(\rd v) \expp{\langle \Phi,\,
       \cmt  (r_s(\ct,v))  \rangle}
  \right]
  =  \N^{\theta}\left[\int_{\ct}\Lambda_t(\rd v) \expp{2\theta Z_s
    -\langle \Phi,\, \cmt(r_s(\ct,v)) \rangle}
  \right].
\]
We apply then \eqref{eq:Bismut0} to get:
\begin{multline*}
     \N^{-\theta}\left[\int_{\ct}\Lambda_t(\rd v) \expp{-\langle \Phi,\,
         \cmt(r_s(\ct, v))  \rangle}
  \right]\\
  \begin{aligned}
&=\expp{-2\beta \theta t}  \E\left[\expp{2\theta Z_s}
    F(r_s(\graft_1([0,t],\cm))\right]\\
&=\exp\left\{-2\beta\theta t-2\beta\int_0^t \rd a\,
  \N^\theta\left[1-\expp{-\Phi(a, r_{s-a}(\ct))+ 2\theta Z_{s-a}}\right]\right\}\\
&=\exp\left\{-2\beta\theta t-2\beta\int_0^t \rd a\,
 \left(\N^{-\theta}\left[1-\expp{-\Phi(a, r_{s-a}(\ct))}\right] +
 \N^{\theta}\left[1-\expp{2\theta Z_{s-a}}\right] \right) \right\}\\
&= \exp\left\{2\beta\theta t-2\beta\int_0^t \rd a\,
 \N^{-\theta}\left[1-\expp{-\Phi(a, r_{s-a}(\ct))}\right]  \right\},
\end{aligned}
\end{multline*}
where we  used  standard property of Poisson point measures  for   the  second equality,
\eqref{eq:abs-cont2}     again    for     the     third one,    and     that
$\N^{\theta}\left[1-\expp{2\theta     Z_a}\right]      =     u(-2\theta,
a)=-2\theta$, see \eqref{eq:tc=c} and \eqref{eq:def-u}, for the last one. As
$\Phi$ has bounded support, we get taking $s$ large enough:
\[
   \N^{-\theta}\left[\int_{\ct}\Lambda_t(\rd v) \expp{-\langle \Phi,\, \cmt(\ct, v)  \rangle}
  \right]
= \exp\left\{2\beta\theta t-2\beta\int_0^t \rd a\,
  \N^{-\theta}\left[1-\expp{-\Phi(a, \ct)}\right]  \right\}.
\]
Then the result follows from the definition of $\graft_1([0,t],\cm)$, that
is~\eqref{eq:Bismut0}
holds for $F$ given by~\eqref{eq:F=expo}.
\medskip

As $(\ct, v)$ is a measurable function of $\cmt(\ct, v)$, see
Section~\ref{sec:detail-proof-cor},  we  then conclude  by the  monotone class  theorem that
Equation~\eqref{eq:Bismut0}  holds  for  any  non-negative  measurable
function $F$ defined on $\TLn$.
\end{proof}

Let  $\beta>0$, $\theta\in\R$ and $t>0$. Recall $\tilde c^\theta_t= (2\theta)/(1-\expp{-2\beta\theta t})$
defined in~\eqref{eq:def-cqt}. We consider the probability
measure on $[0, t]$:
\begin{equation}
   \label{eq:def-fd}
   \nu(ds)=\frac{2\beta\theta\expp{2\beta\theta s}}{\expp{2\beta\theta
       t}-1}\, \ind_{[0,t]}(s)\, \rd s=  \beta \, \tilde c_t(\theta)
   \,\expp{-2\beta\theta(t-s)}\, \ind_{[0,t]}(s)\, \rd s .
\end{equation}
 Let $(\bT_{n},\bv_n)$ be, under $\P^{\theta,t}$, the planar tree,
 element  of
 $\TDn$, defined 
 in  Section  \ref{sec:bT}  associated  with  the  measure  $\nu$ and
 $t>0$ (recall that all the distinguished vertices from $\bv_n$ but the
 root are at distance $t$ from the root).
 The
 following  theorem  is  a  generalization of  Theorem  \ref{Thmbis}  when
 picking $n$ leaves uniformly at random at level $t$.

 \begin{theo}[Generalized $n$-leaves decomposition]
   \label{thm:k-Bismut}
   Let $\beta>0$,
  $\theta\in \R$,  $t>0$ and $n\in \N^*$.
For every non-negative  measurable
  function $F$ defined on $\TLn$, we
  have:
\begin{equation}
    \label{eqn:k-Bismut}
\N^{\theta}\left[\int_{\ct^{n}}\!\!\!\Lambda_t^{\otimes n}(\rd
   \bv^*)\, F (\ct, \bv)\right]
=n!\, \left(\tilde c_t^{\theta}\right)^{1-n}
\expp{-2\beta\theta t}\, \E^{\theta,t}\left[F\Bigl(\graft_n\bigl((\bT_{n},\bv_n),\cm\bigr)\Bigr)\right],
\end{equation}
where $\bv=(\root,  \bv^*)\in \ct^{n+1}$, with $\root$ the root of
$\ct$, and, under $\E^{\theta,t}$, conditionally given $(\bT_n,\bv_n)$,
$\cm(\rd x,\rd\ct)$ is a Poisson point measure on $\bT_m\times \TL$ with
intensity $2 \beta \, d\length^{\bT_n}(\rd x)\N^\theta[\rd\ct]$.
\end{theo}
We stress again that measurability of  the grafting map $\graft_n$ on a
discrete tree is formally stated in Section~\ref{sec:discrete-app}, so
that $\graft_n\bigl((\bT_{n},\bv_n),\cm\bigr)\Bigr)$ is indeed a
$\TLn$-valued random variable (see in particular
\eqref{eq:def-graftn-intuition} and~\eqref{eq:def-graftn-TK} therein).  
The proof of this theorem is postponed to
Section~\ref{sec:proof-k-Bismut} as it heavily relies on the topological
setting developed in Section~\ref{sec:topo-tree}.

\subsection{Local limit of conditioned Brownian CRT}
\label{sec:main-result}
Let $\beta>0$,  $\theta, \alpha\in\R_+$  and let  $S^{\alpha, \theta}$  be a
Poisson   point  measure   on   $[0,\infty)$   with  intensity   measure
$\fint(t)\, \rd t$, where:
\begin{equation}
  \label{eq:def-fint} 
  \fint(t)=\alpha\beta\expp{2\beta\theta      t},\quad t\geq0.
 \end{equation}
We first consider the case $\alpha>0$.     Denote      by
$(\xi_i,    i\in\N^*)$   the   increasing sequence of  jumping   times
of   the inhomogeneous Poisson   process
$(N^{\alpha, \theta}_t=S^{\alpha, \theta}([0, t]),  t\geq 0)$. We
consider the  $\TDn$-valued random variable $\fT_{\xi_n}$ of Section
\ref{sec:tree_ske} for $n\geq 1$ associated with $\fint$.
In particular, recall that, for every $n\ge 1$, $\fT_{\xi_n}$ is a
discrete tree with $n$ distinguished leaves,  where all of them are at height
$\xi_n$. Recall the construction of the infinite backbone $\fTs$ in
Section \ref{sec:tree_ske} from  the sequence of
trees~$\fT_{\xi_n}$. Notice its distribution depends on $\alpha$ and
$\theta$ (and also $\beta$ which is fixed). 
We  informally  define $\ct^{\alpha,\theta}$  as  the  tree obtained  by
grafting on $\fTs$ (whose distribution depends on $\alpha$ and $\theta$)
a tree  $\ct_i$ at  point $x_i$ where,  conditionally given  $\fTs$, the
family $((x_i,\ct_i),i\in I)$ is the atoms of a Poisson point measure on
$\fTs\times               \TL$              with               intensity
$2\beta\, \length^{\fTs}(\rd  x)\N^\theta(\rd\ct)$. 

%

For $\alpha=0$,
the infinite backbone rooted tree $\fTs$ has only one branch and is identified
with $(\R_+,0)$, and the tree  $\ct^{0, \theta}$ is then identified with
the Kesten tree with parameter $(\beta, \theta)$ defined in
Section~\ref{sec:Levy-T} and formally in Section~\ref{sec:def-formal-kesten}. 

\medskip

Since we are considering equivalence class  of trees, it is ambiguous to
present $\fTs$ as a subtree of $\ct^{\alpha,\theta}$. This motivate the
introduction  of marked  trees in  Section~\ref{sec:TnSn}; and  to avoid
confusion,  we shall  denote  $ \fT^{\alpha,  \theta}$  the subtree  of
$\ct^{\alpha,\theta}$; it is in
the same equivalent class as $\fTs$ in $\TL$.  We  refer to
Section~\ref{sec:ct-aq}  for a  formal and  more rigorous  definition of
the trees  $(\ct^{\alpha,\theta}, \fT^{\alpha,  \theta}) $.  
  We then define the random process
$(\ct_t^{\alpha,\theta}, t\geq 0)$ by setting:
\[
  \ct_t^{\alpha,\theta}=r_t(\ct^{\alpha,\theta}).
\]
Recall that the $\TL$-valued
function  $\ct$ is under $\N^\theta$ a L\'evy tree;  and we write
$\ct_t=r_t(\ct)$. We now give the main result of this section.

\begin{prop}[Representation of an $h$-transform of the CRT]
  \label{prop:Taq=TMaq}
  Let $\beta \in \R_+^*$, $\theta,  \alpha\in\R_+$ and $t>0$.  For every
  non-negative measurable functional $F$ on $\TL$ (or $\TK$), we have:
 \[
   \E\left[F\left(\ct^{\alpha,
           \theta}_t\right)\right]
   =\N^{\theta}\left[F\left(\ct_t\right)\, M_t^{\alpha,
       \theta}\right].
 \]
\end{prop}

\begin{rem}[On the $h$ transform of $Z$]
  \label{rem:tree-Z}
  By    considering   the    size   population    at   level    $t$   of
  $\ct^{\alpha, \theta}$, the above
  proposition 
  gives  a  representation  of  the   process  $(Z_t,  t\geq  0)$  under
  $\N^\theta[\cdot \,  M^{\alpha, \theta}]$ as  a quadratic CSBP  with a
  Poisson  immigration given  by  $\fTs$ and  the grafting  intensity
  $2\beta\,   \length^{\fTs}(\rd  x)\N^\theta(\rd\ct)$.
As $(\tilde N_t(\fTs), t\geq 0)$ is distributed as $(N^{\alpha,
  \theta}_t+1,  t\geq 0)$, that is, as $(S^{\alpha, \theta}_t+1, t\geq
0)$, this  provides
  another proof of Proposition~\ref{prop:Zaq}.
\end{rem}

\begin{proof}[Proof of Proposition~\ref{prop:Taq=TMaq}]
  We first  consider the case $\alpha>0$.   Let us fix $t>0$,  and write
  $N_t$  for $N^{\alpha,  \theta}_t$.  Recall  $\tilde N_t  (T)$ is  the
  number    of    vertices    of    $T$    at    level    $t$.     Since
  $\fTs=\fT^{\alpha,     \theta}$    in     $\TL$,    we     get    that
  $\tilde N_t  (\fT^{\alpha, \theta})=\tilde N_t (\fTs)$  is distributed
  as            $N_t+1$.             The            fact            that
  $(\ct^{\alpha,  \theta}_t, \tilde  N_t (\fT^{\alpha,  \theta}))$ is  a
  well   defined   random  variable   is   detailed   at  the   end   of
  Section~\ref{sec:ct-aq}. We shall also consider the truncated backbone
  $\fT^{\alpha, \theta}_t=r_t(\fT^{\alpha, \theta})$ for $t\geq 0$, and
  see $\fT^{\alpha, \theta}_t$ as a subtree of $\ct^{\alpha, \theta}_t$.
  
\medskip

Let $(\bT_n,n\ge 0)$ be the sequence of trees defined in Section \ref{sec:bT} associated with the function:
\begin{equation}
   \label{eq:def-fdLevy}
\fd(s)=\beta \, \tilde c_t(\theta)
   \,\expp{-2\beta\theta(t-s)}\, \ind_{[0,t]}(s).
 \end{equation}
Recall the continuous canonical projection $\Pi_n^\circ :
\TKn\longrightarrow \TK$  defined by $\Pi_n^\circ(\bt,\bv)=\bt$. Set $\graft_k^\circ=\Pi_k^\circ \circ \graft_k$. Then we have,
\begin{align*}
    \E\left[F\left(\ct^{\alpha,
  \theta}_t\right)\right]
  &= \sum_{n\in \N}   \E\left[F\left(\ct^{\alpha,
  \theta}_t\right)\Bigm|\,\tilde  N_t (\fT^{\alpha, \theta})=n+1\right]
    \, \P(N_t^{\alpha, \theta}=n)\\ 
  &= \sum_{n\in \N}  \E\left[F\Bigl(r_t\bigl(\graft_{n+1}^\circ(\fT_t^{\alpha,\theta}, \cm_t)\bigr)\Bigr) \Bigm|  N_t^{\alpha,
      \theta}=n\right]
    \frac{(\alpha /c_t^\theta)^{n} \expp {-\alpha/c_t^\theta}}{n!}\\
    & =\sum_{n\in \N}
      \E\left[F\Bigl(r_t\bigl(\graft_{n+1}^\circ(\bT_{n+1},
      \tilde\cm_t)\bigr)\Bigr)\right]
    \frac{(\alpha /c_t^\theta)^{n} \expp {-\alpha/c_t^\theta}}{n!},
\end{align*}
where we used that $\tilde N_t (\fT^{\alpha, \theta})$ is distributed as
$N_t^{\alpha, \theta}+1$  for the first equality,  that conditionally on
$\tilde   N_t   (\fT^{\alpha, \theta})=n+1$,   the   random   tree
$\ct_t^{\alpha,         \theta}$        is         distributed        as
$r_t\left(\graft_{n+1}^\circ(\fT_{t}^{\alpha,\theta},\cm_t)\right)$
conditionally     on      $N_t^{\alpha,     \theta}=n$  where,
conditionally given $\fT_t$, $\cm_t$ (resp. $\tilde\cm_t$) is a Poisson point measure on
$\fT_t^{\alpha,\theta}\times \TK$  (resp. $\bT_{n+1}\times \TK$)
with intensity
$2 \beta\, \length^{\fT_t^{\alpha,\theta}}(\rd x)\N^\theta(\rd\theta)$
(resp. 
$2 \beta\, \length^{\bT_{n+1}}(\rd x)\N^\theta(\rd\theta)$),   and     that
$N_t^{\alpha,  \theta}$  is  distributed   as  a  Poisson  process  with
intensity $\alpha$  at time $1/c^\theta_t$ (see  Lemma~\ref{lem:YS}) for
the     second    one,     and that    $\fT_{t}^{\alpha,\theta}$ conditionally    on
$N_t^{\alpha,  \theta}=n$  is  distributed  as  $\bT_{n+1}$  with  $\fint$
and $\fd$
in~\eqref{eq:def-fd-fi}     given      by~\eqref{eq:def-fint} and~\eqref{eq:def-fdLevy}     (see
Lemma~\ref{lem:T_ske distribution}) for the  last one. 
Using Theorem  \ref{thm:k-Bismut} and
that $\nu(\rd s)$ in~\eqref{eq:def-fd} is  exactly $\fd(s)\, \rd s$ with
$\fd$ given by~\eqref{eq:def-fdLevy}, we have:
\begin{align*}
  \E\left[F\left(r_t\left(\graft_{n+1}^\circ(\bT_{n+1},
        \tilde\cm_t)\right)\right) \right]
& = \frac{\left(\tilde c_t^{\theta}\right)^{n}
    \expp{2\beta\theta t}}{(n+1)!}\,
    \N^{\theta}\left[\int_{\ct^{n+1}}\!\!\!\Lambda_t^{\otimes (n+1)}(\rd
      \bv^*)\, F \bigl(r_t(\ct)\bigr) \right]\\
& = \frac{\left(\tilde c_t^{\theta}\right)^{n}
    \expp{2\beta\theta t}}{(n+1)!}\,
    \N^{\theta}\left[\int_{\ct^{n+1}}\!\!\!\Lambda_t^{\otimes (n+1)}(\rd
      \bv^*)\, F (\ct_t) \right]\\
& =      \frac{\left(\tilde c_t^{\theta}\right)^{n}
    \expp{2\beta\theta t}}{(n+1)!}\,
    \N^{\theta}\left[ Z_t^{n+1}F(\ct_t)\right],
\end{align*}
as $Z_t=\Lambda_t(\ind)$ is the total local time of $\ct$ at level $t$.
Thus,  using  the  definition  of  $M_t^{\alpha,  \theta}$  in
\eqref{Mmartingale}, we obtain:
\[
    \E\left[F\left(\ct^{\alpha,
          \theta}_t\right)\right]
  = \sum_{n\in \N}  \frac{\left(\tilde c_t^{\theta}\right)^{n}
    \expp{2\beta\theta t}}{(n+1)!}\, \N^\theta\left[Z_t^{n+1}
  F(\ct_t)\right]
    \frac{(\alpha /c_t^\theta)^{n} \expp {-\alpha/c_t^\theta}}{n!}
=\N^{\theta}\left[F\left(\ct_t\right)\, M_t^{\alpha,
       \theta}\right].
 \]

 The simpler case $\alpha=0$, which is left  to the reader,
 can also be handled in a similar way.
\end{proof}

As a conclusion, we deduce the following result for $\alpha>0$.
\begin{theo}[Local limit of CRT in the  Poisson  regime]
  \label{theo:main-result}
  Let $\alpha,\beta>0$,
  $\theta\in \R$. Assume that the function $a$ is such that
 as $t\rightarrow\infty$ :
\[
    a_t\sim \begin{cases}
      \alpha \beta^2 t^2, & \text{if
        $\theta= 0$};\\
        \alpha (2\theta)^{-2} \expp{2 \beta |\theta| t}, & \text{if
        $\theta\neq 0$}.
    \end{cases}
  \]
  For every  non-negative  measurable
  function $F$ on $\TK$ and  $s>0$, we
  have:
 \[
\lim_{t\rightarrow \infty }   \N^{\theta}\left[F\left(\ct_s\right)\,|\,
 Z_t=a_t\right]
   = \E\left[F\left(\ct^{\alpha,
           |\theta|}_s\right)\right].
 \]
\end{theo}
\begin{proof}
Clearly,  Proposition~\ref{prop:cv-At-P}  still  holds if $H_s$ is
$\cg_s=\sigma(r_s(\ct))$ measurable, that is $H_s=F(\ct_s)$ with $F$
non-negative defined on $\TL$, and $Z_t$ is the total local time of $\ct$
at level $t$, see Section~\ref{sec:Levy-T}. We deduce that:
\[
   \lim_{t\rightarrow \infty } \N^\theta[F(\ct_s) |\,
    Z_{t}=a_{t}]
    =\N^{|\theta|} \left[F(\ct_s) \, M^{\alpha, |\theta|}_s\right]
    =\E \left[F(\ct_s^{\alpha, |\theta|}) \right],
\]
where we used Proposition \ref{prop:Taq=TMaq}
for the last equality.
\end{proof}
Similarly, we also get the following result for $\alpha=0$. Recall that
$\ct^{0, \theta}$ is a Kesten tree with parameter $(\beta, \theta)$.
\begin{theo}[Local limit of CRT in the Kesten regime]
  \label{theo:main-result-2}
  Let $\beta>0$,
  $\theta\in \R$. Assume that the function $a$ is positive such that
  as $t\rightarrow\infty$:
\[
    a_t=  \begin{cases}
      o(t^2), & \text{if
        $\theta= 0$};\\
         o( \expp{2 \beta |\theta| t}) & \text{if
        $\theta\neq 0$}.
    \end{cases}
\]
  For every  non-negative  measurable
  function $F$ on $\TK$ and  $s>0$, we
  have:
 \[
\lim_{t\rightarrow \infty }   \N^{\theta}\left[F\left(\ct_s\right)\,|\,
 Z_t=a_t\right]
   = \E\left[F\left(\ct^{0,|\theta|}_s\right)\right].
 \]
\end{theo}

\begin{rem}
  \label{rem:over}
  Using~\cite{Ov94},
  Corollary~\ref{cor:Z=Ov} on the SDE
  for the size-population process $Z$, is a direct consequence of 
  Theorems~\ref{theo:main-result} (for $\alpha\in (0, +\infty )$)
  and~\ref{theo:main-result-2} (for $\alpha=0$) and
  Remark~\ref{rem:tree-Z}. 
\end{rem}

\section{Set of trees, topology  and measurability}
\label{sec:topo-tree}

In  a nutshell,  the main  objective of  this section  is to  define the
grafting  and  splitting  functions,  as  well  as  the  decorating  and
de-decorating  functions in  a  measurable  way on  the  set of  locally
compact  rooted  real   trees,  so  that  we  can   properly  define  in
Section~\ref{sec:appli} the  random variables  used the in  the previous
sections.  An  index of  all the (numerous)  relevant notations  of this
section is provided at the end of the document.

We keep  the basic  definitions and
notations for  rooted real trees from  Section~\ref{sec:tree-notation}.   In
Section~\ref{sec:cont-span} we consider the regularity of the spanning
of subtrees. In Section~\ref{sec:root-Tn}, we study  the
Polish spaces of  equivalent classes of compact  (resp.  locally compact)
rooted   trees   with   distinguished    vertices   endowed   with   the
Gromov-Hausdorff  distance.   We   define  various  grafting  measurable
operations (denoted by  $\circledast_*^*$) of a tree on  an another tree
in  Section~\ref{sec:graft1}.  Motivated  by the  fact that  some random
trees  are  obtained  as  decorated  backbone  trees,  we  introduce  in
Section~\ref{sec:TnSn} the space of marked  trees, that is of trees with
a distinguished subtree (or backbone  tree).  We also establish in this
section  the  measurability  of  various  truncation  maps.   The  short
Section~\ref{sec:TLs} is  devoted to special  case of the  backbone tree
being reduced  to an  infinite spine  (this is the  case for  the Kesten
tree).    In  Section~\ref{sec:def-vA-wA},   we  consider   specifically
discrete  trees which  are spanned  by $n$  distinguished vertices,  and
describe them as  a set of branches indexed by  all the possible subsets
of the  $n$ distinguished  vertices.  This description  is then  used in
Section~\ref{sec:split} to split (with  a function $\Split_n$) a locally
compact tree with  $n$ distinguished vertices as  subtrees supported by
the different branches of the discrete tree spanned by the distinguished
vertices.   Then, we  provide in  a  sense the  inverse construction  in
Section~\ref{sec:graft} where  (with a function $\graft_n$)  we decorate
the    branches   of    a    discrete   trees    with   subtrees.     In
Section~\ref{sec:def-M}, we describe a  measurable procedure to decorate
a  branch with  a family  of subtrees  given by  the atoms  of a  point
measure on  the set  of trees  (the function  $\tree$) and  a measurable
procedure to describe the decoration of a distinguished branch of a tree
(the  function $\cmt$)  through a  point measure  on the  set of  trees.
\medskip

\medskip

We  shall  use   many  times  Lusin's  theorem   from  \cite{purves}  or
\cite[Exercise~6.10.54~p.60]{bogachev2} which  states that, if $f$  is a
measurable function defined on a Borel subset $A$ of a Polish space to a
Polish space, then $f(B)$ is a Borel set for all Borel subsets $B\subset A$
if and only  if the set of  all values $y$, such  that $f^{-1}(\{y\})$ is
uncountable, is at most countable.

\subsection{Continuity of the map Span}
\label{sec:cont-span}
Recall the  definition of the  set $\TKn$ of $n$-pointed  compact rooted
tree in  Subsection \ref{sec:sets_of  trees}, endowed with  the distance
$\dghn$.   Recall  also  the  definition  of  the  tree  spanned  by  $n$
vertices.   For   a  rooted   $n$-pointed   tree   $(T,  d,   \bv)$,   with
$\bv=(\root,v_1 \ldots, v_n)$, we  denote the corresponding spanned tree
$\Span^\circ(T,\bv)$ as:
\begin{equation}
   \label{eq:def-span}
  \Span^\circ(T,\bv)= \bigcup _{k=1}^n \lb \root, v_k\rb.
\end{equation}
The tree  $(\Span^\circ (T, \bv), d,  \root)$ will be simply  denoted by
$\Span^\circ(T, \bv)$,  whereas we  will denote  by $\Span(T,  \bv)$ the
rooted $n$-pointed tree $(\Span^\circ (T, \bv), d, \bv)$.  For $y\in T$,
we    also   define    $p_\bv(y)$,    the   projection    of   $y$    on
$\Span^\circ (T,\bv)$, as  the only point of  $\Span^\circ (T,\bv)$ such
that:
 \begin{equation}\label{eq:def-pv}
 \lb\root,y\rb\cap \Span^\circ (T,\bv)=\lb\root,p_\bv(y)\rb.
 \end{equation}

 Let us state a technical result which will be used  several times in what follows.

 \begin{lem}
   \label{lem:dist-proj}
   Let $n\in\N$.
 Let $(T,d, \bv)$ and $(T', d', \bv')$ be two compact rooted $n$-pointed
 trees  and  let $\cR$  be  a  correspondence  between them.  For  every
 $(x,x')\in\cR$ with $x'\in \Span^\circ(T',\bv')$, we have:
\[
d(x,p_\bv(x))\le \frac{3}{2} \dist (\cR).
\]
 \end{lem}

 \begin{proof}

   Let $(x,x')\in \cR$ with $x'\in \Span^\circ(T',\bv')$.
   First  remark that  there exist  $k,\ell\in\{0,\ldots,n\}$ such  that
   $p_\bv(x)\in\lb         v_k,v_\ell\rb$          and         $x'\in\lb
   v'_k,v'_\ell\rb$.
   Indeed, let us set:
 \[
 A=\bigl\{v_k\colon\, p_\bv(x)\in \lb \root,
 v_k\rb\bigr\}\quad\text{and}\quad A'=\bigl\{v'_k\colon\, x'\in \lb \root',  v'_k\rb\bigr\}.
 \]
 Notice that $A\ne \emptyset$ and $A'\ne \emptyset$.
 If there  exists $k\ge 1$ such  that $v_k\in A$ and  $v'_k\in A'$, then
 one   can  take   $\ell=0$ so that $v_\ell=\root$ and $v'_\ell=\root'$.  Otherwise,   take  $k$ and $\ell$ with $k\neq  \ell$ such that $v_k\in   A$  and
 $v'_\ell\in A'$ . In this case, we get  $v_\ell\not\in A$.  Clearly we
 have  $p_\bv(x)\in\lb   v_k,v_\ell\rb$  and  by  a   similar  argument,
 $x'\in\lb v'_k,v'_\ell\rb$.  Therefore, we have:
\begin{multline*}
2d\bigl(x,p_\bv(x)\bigr)=d(x,v_k)+d(x,v_\ell)-d(v_k,v_\ell)
\le d'(x',v_k')+d'(x',v'_\ell)-d'(v'_k,v'_\ell)+ 3\, \dist (\cR).
\end{multline*}
Then,  use that $d'(x',v_k')+d'(x',v'_\ell)-d'(v'_k,v'_\ell)=0$, as  $x'\in\lb
   v'_k,v'_\ell\rb$, to conclude.
\end{proof}

If  $(T, \bv)$  and $(T',  \bv')$ belong  to the  same equivalence
class in $\TKn$, then  so do $\Span(T,\bv)$ and  $\Span(T',\bv')$ in $\TKn$. Therefore,   the
function     $(T,    \bv)     \mapsto     \Span     (T,    \bv)$ is  well defined from $\TKn$ to
$\TKn$. A first  consequence of
Lemma~\ref{lem:dist-proj} is that this  function is Lipschitz
continuous; this result will be completed in Lemma~\ref{lem:cont-span}.

\begin{lem}[Continuity of the map $\Span$]
  \label{lem:Lip-Span}
  Let  $n\in\N$.   The  map  $(T,  \bv) \mapsto  \Span  (T,  \bv)$ is $4$-Lipschitz continuous from
  $\TKn$ to $\TKn$.
\end{lem}

\begin{proof}
  Let $(T,\bv),(T',\bv')$  be two  compact rooted $n$-pointed  trees and
  let $\cR$ be a correspondence between them.  Let us set with obvious notations:
\begin{multline}
   \label{eq:def-Tilde-R}
  \tilde \cR
=\Bigl\{\bigl(x,p'_{\bv'}(x')\bigr)\colon\, (x,x')\in\cR,\ x\in \Span^\circ(T,\bv)\Bigr\}
 \\   \cup  \Bigl\{\bigl(p_\bv(x),x'\bigr)\colon\, (x,x')\in \cR,\ x'\in \Span^\circ(T',\bv')\Bigr\}.
\end{multline}
Clearly,  $\tilde\cR$ is  a  correspondence  between $\Span(T,\bv)$  and
$\Span(T',\bv')$.  We now  compute its distortion. We  consider the case
$x\in \Span(T,\bv)$, $y'\in \Span(T',\bv')$ and $(x,x'), (y,y')\in \cR$,
so that $(x,p'_{\bv'}(x'))$ and $(p_\bv(y),y')$ belong to $\tilde\cR$. We have:
\begin{align*}
\Bigl| d\bigl(x,p_\bv(y)\bigr)-d'\bigl(p'_{\bv'}(x'),y'\bigr)\Bigr| & = \Bigl| d(x,y)-d\bigl(y,p_\bv(y)\bigr)-d'(x',y')+d'\bigl(x',p_{\bv'}(x')\bigr)\Bigr|\\
& \le \bigl|d(x,y)-d'(x',y')\bigr| +d\bigl(y,p_\bv(y)\bigr)+d'\bigl(x',p_{\bv'}(x')\bigr)\\
& \le 4\, \dist(\cR),
\end{align*}
where we used Lemma \ref{lem:dist-proj} for the last inequality. The
other cases can be treated similarly. This implies that $\dist(\tilde  \cR)\leq
4\, \dist (\cR)$ and thus, by definition of $\dghn$:
\[
\dghn\bigl(\Span(T,\bv),\Span(T',\bv')\bigr)\le 4\, \dghn \bigl((T,\bv),(T',\bv')\bigr).
\]
\end{proof}

\subsection{Set of (equivalence classes of) rooted $n$-pointed locally  compact trees}
\label{sec:root-Tn}
Recall the definition of the height $H(x)=d(\root, x)$ of a vertex $x$ in a rooted tree $(T,
d, \rho)$.
For a rooted $n$-pointed  tree $(T,d, \bv)$ and $t\geq 0$, we define the
rooted $n$-pointed tree  $T$ truncated at level $t$ as $(r_t(T,\bv), d, \bv)$ with:
\begin{equation}
   \label{eq:def-rt}
  r_t(T, \bv)=\bigl\{x\in T\, \colon\, H(x)\leq t\bigr\}\, \cup \bigl\{\Span^\circ(T, \bv)\bigr\},
\end{equation}
and the distance on $r_t(T, \bv)$  is given by the restriction of the
distance $d$. We shall simply write $r_t(T, \bv)$ for $(r_t(T, \bv), d,
\bv)$. (Notice that for $t\geq t_T=\max_{i\in \{0,
  \ldots, n\}} d(\root, v_i)$ the truncated operations defined
by~\eqref{eq:def-rt} and~\eqref{eq:def-rt0} coincide.)

If $(T,\bv)$ and $(T', \bv')$ are in the same equivalence class of $\TKn$, so are
$r_t(T, \bv)$ and $r_t(T', \bv')$. Thus the function $r_t$ can be seen as a map
from $\TKn$ to itself. When $n=0$, we shall simply write $r_t(T)$ for
$r_t(T, \root)$.
The next lemma is about the continuity of $r_t$.

\begin{lem}[Continuity of $r_t$]
   \label{lem:cont-rt}
Let $n\in \N$.    For $s,t\geq 0$ and $(T, \bv), (T', \bv')\in \TKn$, we have:
   \begin{equation}
  \label{eq:cont-rt-gh}
    \dghn \left(r_t(T, \bv), r_{t+s}(T', \bv')\right) \leq  4 \,
    \dghn\bigl((T, \bv), (T', \bv')\bigr) +s.
\end{equation}
The map $\bigl(t, (T, \bv)\bigr) \mapsto r_t(T, \bv)$ is continuous from $\R_+\times
\TKn$ to $\TKn$.
\end{lem}

\begin{proof}
  Let  $(T, d, \bv),(T', d', \bv')$  be  two compact  rooted  $n$-pointed  trees. Firstly, notice that $\dghn\bigl(r_{t+s}(T,\bv), r_t(T,\bv)\bigr)\leq  s$.
  Secondly, recall Definition \eqref{eq:def-pv} of the  projection $p_\bv$ on
  $\Span^\circ(T,\bv)$.   For $y\in  T$, we  also define  the projection
  $p_t(y)$ of $y$ on  $r_t(T, \bv)$  as the  only point  of $r_t(T,  \bv)$ such
  that:
 \[
   \lb\root,y\rb\cap r_t (T,\bv)=\lb\root,p_t(y)\rb.
 \]
 We first prove the analogue of Lemma~\ref{lem:dist-proj}.  Let $\cR$ be
 a    correspondence   between    $(T,\bv)$   and    $(T',\bv')$.    Let
 $(x,x')\in \cR$ with  $x'\in r_t(T', \bv')$.  By  construction, we have
 $p_t(x)\in \lb p_\bv(x),  x \rb$.  If $x'\in \Span(T',  \bv')$, then we
 deduce           from          Lemma~\ref{lem:dist-proj}           that
 $d\bigl(x,p_t(x)\bigr)\le   d\bigl(x,p_\bv(x)\bigr)\le   \frac{3}{2}    \dist(   \cR)$.    If
 $x'\in   r_t(T',  \bv')\setminus   \Span(T',  \bv')$,   then  we   have
 $H(x')\leq t$  and thus $H(x)=d(\root, x) \leq  d'(\root', x')+ \dist
( \cR) \leq   t +  \dist (\cR)$, which  implies that
 $d\bigl(x,p_t(x)\bigr)\le     \dist    (\cR)$.      In    conclusion,     we    get
 $d\bigl(x,p_t(x)\bigr)\le  \frac{3}{2} \dist  (\cR)$. Now, arguing  as in  the proof  of
 Lemma~\ref{lem:Lip-Span},   we  deduce   that
 $\dghn   \bigl(r_t(T,   \bv),   r_{t}(T',   \bv)\bigr)   \leq   4   \,
 \dghn\bigl((T,\bv), (T', \bv')\bigr)$. This gives
  the result.
\end{proof}

A rooted $n$-pointed tree $(T, d, \bv)$ is locally compact if $r_t(T, \bv)$ is a
compact rooted tree for all $t\geq 0$. Following \cite{adh}, we set
for two locally compact rooted $n$-pointed  trees $(T, \bv)$ and $(T', \bv')$:
\[
  \dlghn((T, \bv),(T', \bv'))=\int_0^\infty \expp{-t} \rd t \,   \left(1
    \wedge \dghn \bigl(r_t(T, \bv), r_t(T', \bv')\bigr)
  \right).
\]
Furthermore, we  have that $\dlghn\bigl((T,\bv), (T',  \bv')\bigr)=0$ if
and only  if there exists an  isometric bijection  from  $(T, d)$ to
$(T', d')$ which  preserves the distinguished vertices  (this can easily
be proved  with similar  arguments as  for \cite[Proposition~5.3]{adh}).
The  relation  $\dlghn\bigl((T,  \bv), (T',  \bv')\bigr)=0$  defines  an
equivalence relation.  Arguing as in~\cite{adh}  where $n=0$, we get the
following result.  Following the notations in~\cite{adh},  for $n=0$, we
simply write $\TL$ and $\dlgh$ for $\TLn$ and $\dlghn$.

\begin{prop}[$\TLn$ is Polish]
  \label{prop:TLn-Polish}
The set
$\TLn$ of  equivalence classes of  locally compact rooted $n$-pointed  trees endowed
with $\dlghn$ is a metric  Polish space.  Furthermore, the set $\TKn$ of
equivalence classes of compact rooted  $n$-pointed trees  is an
open  dense subset  of $\TLn$. 
\end{prop}

We first provide a short proof for the following  inequalities.

\begin{lem}[Inequalities for $\dghn$ and $\dlghn$]
   \label{lem:cont-lgh-gh}
Let $n\in \N$. For $(T, \bv), (T', \bv')\in \TKn$, we have:
\begin{equation}
  \label{eq:cont-lgh-gh}
  \dlghn\bigl((T, \bv), (T', \bv')\bigr) \leq  1 \wedge 4 \,\dghn\bigl((T,\bv), (T', \bv)\bigr).
\end{equation}
For $(T, \bv), (T', \bv')\in \TLn$ and $s,t\geq 0$, we have:
\begin{align}
  \label{eq:cont-rt-lgh}
  \dlghn \bigl(r_t(T, \bv), r_{t+s}(T', \bv')\bigr)
  &\leq  4 \, \dlghn\bigl((T,\bv),(T', \bv')\bigr)+s,\\
  \label{eq:cont-rt-lgh-gh}
  \dghn \bigl(r_t(T, \bv), r_t(T', \bv')\bigr)
  &\leq  4 \expp{t}\,  \dlghn\bigl((T, \bv), (T', \bv')\bigr).
  \end{align}
The map $\bigl(t, (T, \bv)\bigr) \mapsto r_t(T, \bv)$ is continuous from $\R_+\times
\TLn$ to $\TLn$ (and to $\TKn$).
\end{lem}

\begin{proof}
  Equation   \eqref{eq:cont-lgh-gh}   is   a   direct   consequence   of
  \eqref{eq:cont-rt-gh}  with  $s=0$  and the  definition  of  $\dlghn$.
  Equation \eqref{eq:cont-rt-lgh} follows  from similar arguments, using
  also  that  $r_{t'}  \circ  r_u=   r_u\circ  r_{t'}=r_{t'  \wedge  u}$.   For
  $t\leq                  s$,                   we                  have
  $   4^{-1}\,  \dghn\bigl(r_t(T,\bv),   r_t(T',\bv')\bigr)\leq  \dghn\bigl(r_s(T,\bv),
  r_s(T',\bv')\bigr)$. Integrating with respect  to $\expp{-s}\, \rd s$ gives
  \eqref{eq:cont-rt-lgh-gh}.     The     continuity    of     the    map
  $(t, (T,\bv)) \mapsto r_t(T)$ is a direct consequence
  of~\eqref{eq:cont-rt-lgh}.
 \end{proof}

We deduce from \eqref{eq:cont-lgh-gh} and \eqref{eq:cont-rt-lgh-gh} that
all  the measurable  sets  of  $(\TKn, \dghn)$  are  measurable sets  of
$(\TLn, \dlghn)$, and  that a converging sequence in  $(\TKn, \dghn)$ is
also  converging   in  $(\TLn,  \dlghn)$.    We  also  we   deduce  from
\eqref{eq:cont-lgh-gh} that  the restriction  to $\TKn$ of  a continuous
function   defined   on  $(\TLn,   \dlghn)$   is   also  continuous   on
$(\TKn, \dghn)$.

\medskip

Removing from $\bv$ some of the distinguished  vertices (but the
root) is continuous, see the next lemma. For $(T, \bv=(v_0=\root, \ldots, v_n))\in
\TLn$ and $0\in A\subset \{0, \ldots, n\}$, we set:
\begin{equation}
   \label{eq:def-Pi0}
  \Pi_n^{\circ, A}(T, \bv)=(T, \bv_A)
  \quad\text{with}\quad
  \bv_A=(v_i, i\in A).
\end{equation}
For simplicity, we shall write $\Pi_n^\circ$ for $\Pi_n^{\circ, A}$ when
$A$ is reduced to $\{0\}$, so that $\Pi_n^\circ$ corresponds to removing
all the distinguished vertices but the root.
\begin{lem}[Removing some distinguished vertices is continuous]
  \label{lem:erase-cont}
  Let $n\in \N$ and $0\in A\subset\{0, \ldots, n\}$.
  The map $\Pi_n^{\circ, A} $ from $\TLn$ to $\TLk$, with $k$ the
  cardinal of $A$,  is
  $1$-Lipschitz continuous.
\end{lem}

\begin{proof}
First, notice that the
   equivalence class  of $(T, \bv_{A})$ in  $\TLk$ does not depend  of the
   choice of $(T,\bv)$ in its equivalence class in $\TLn$.  Thus the map
   $\Pi_n^{\circ, A} $  is well  defined from  $\TLn$ to
  $\TLk$.  It is  clearly  $1$-Lipschitz continuous since  a correspondence
   between the trees $(T,\bv)$  and $(T',\bv')$  is also a  correspondence between
   $(T,\bv_A)$ and $(T,\bv'_A)$.
\end{proof}

We give an immediate consequence on the continuity of the maps $\Span $
and $\Span^\circ$.

\begin{lem}[Continuity of the maps $\Span$ and $\Span^\circ$]
  \label{lem:cont-span}
  Let  $n\in\N$.   The  map  $(T,  \bv) \mapsto  \Span  (T,  \bv)$ and
$(T,  \bv) \mapsto  \Span ^\circ (T,  \bv)$ are  $4$-Lipschitz continuous from
  $\TLn$ to $\TLn$ and to $\TL$ respectively.
\end{lem}
\begin{proof}
    Notice                                                            that
  $\dlghn\bigl(\Span(T, \bv), \Span(T',  \bv)\bigr)=\dghn\bigl(\Span(T, \bv), \Span(T',
  \bv)\bigr)$,  and   thus  the  map   $\Span$  from  $\TLn$  to   $\TLn$  is
 $4$-Lipschitz continuous, thanks     to     Lemma~\ref{lem:Lip-Span}.
 Then use Lemma~\ref{lem:erase-cont} on the continuity of $\Pi_n^\circ$
 and the fact that $\Span^\circ=\Pi_n^\circ \circ \Span$ to conclude.
\end{proof}

Next, we check  that
 rerooting  or   reordering  the   distinguished  vertices  is   a  continuous
 operation.  For a  vector $\bv=(v_0,  \ldots, v_n)$  and a  permutation
 $\pi$       of       $\{0,       \ldots,       n\}$,       we       set
 $\bv^\pi=(v_{\pi(0)}, \ldots, v_{\pi(n)})$.

 \begin{rem}
 One can see that the map $(T, \bv)\mapsto (T, \bv^\pi)$  is an isometry on $\TKn$. The next lemma is an extension to locally compact case.
 \end{rem}

\begin{lem}[Permuting the distinguished vertices is continuous]
  \label{lem:rerooting}
Let $n\in \N$ and  let $\pi$ be a permutation  on $\{0, \ldots, n\}$. The  map
$(T, \bv)\mapsto (T, \bv^\pi)$  defined on $\TLn$ is  continuous.
\end{lem}

\begin{proof}
   First notice that if $(T, \bv)$ and $(T', \bv')$ are rooted
   $n$-pointed trees belonging to the same equivalence class of $\TLn$,
   so do $(T, \bv^\pi)$ and $(T', \bv'{}^\pi)$. Thus, the map $(T,
   \bv)\mapsto (T, \bv^\pi)$ is indeed well-defined on $\TLn$.
We shall use the following notation: we denote by $r_t^\circ$ the truncation $r_t$
when one forgets about the distinguished vertices (but the root): $r_t^\circ=\Pi_n^\circ\circ r_t$.
(Take care that $\Pi_n^\circ\circ r_t\neq  r_t\circ\Pi_n^\circ$.)
To prove the continuity of the map, we consider two cases.

  \emph{1st case}: No rerooting, $\pi(0)=0$.
  In that case, for every $t\ge 0$ and every $(T,\bv)\in \TLn$, we have
  that  $r_t^\circ(T,\bv)=r_t^\circ(T,\bv^\pi)$ and thus we get that:
  \[
  \dlghn\bigl((T,\bv^\pi),(T',\bv'{}^\pi)\bigr)=\dlghn\bigl((T,\bv),(T',\bv')\bigr).
  \]
This trivially implies the continuity of the map.

  \emph {2nd case}:  With rerooting, $\pi(k_0)=0$ for some  $k_0\ne 0$.  Let
  $(T,\bv),\            (T',\bv')\in\TLn$, with $\bv=(v_0=\root, \ldots,
  v_n)$ and $\bv'=(v'_0=\root', \ldots, v'_n)$,  such            that
  $\dlghn\bigl((T,\bv),(T',\bv')\bigr)<1/2$.   As  $v_{k_0}$  and  $v'_{k_0}$  are
  always in correspondence as well as $\root$ and $\root'$, we have, for
  every $t\ge 0$ that:
\[
\bigl|H(v_{k_0})-H(v'_{k_0})\bigr|\le 2\dghn\bigl(r_t(T,\bv),r_t(T',\bv')\bigr).
\]
Multiplying  by $\expp{-t}$ and integrating yields:
\[
1\wedge \bigl|H(v_{k_0})-H(v'_{k_0})\bigr|\le 2\dlghn \bigl((T,\bv),(T',\bv')\bigr)<1,
\]
and hence:
\[
H(v'_{k_0})\le H(v_{k_0})+1.
\]
We set $h_0= H(v_{k_0})+1$.
Then, for every $t\ge 0$, we have:
\[
r_t^\circ(T,\bv^\pi)\subset r^\circ_{t+h_0}(T,\bv)\quad\text{and thus }\quad
r_t(T,\bv^\pi)=r_{t}\Bigl(r^\circ_{t+h_0}(T,\bv),\bv^\pi\Bigr),
\]
and the same holds for $T'$. Consequently, applying Lemma \ref{lem:cont-rt}, we have:
\begin{align*}
  \dlghn\bigl((T,\bv^\pi),(T',\bv'{}^\pi)\bigr)
  & \le 4\int_0^{+\infty}\rd t \expp{-t} \left(1\wedge
    \dghn\Bigl(\bigl(r^\circ _{t+h_0}(T,\bv), \bv^\pi\bigr),\bigl(r^\circ _{t+h_0}(T',\bv'),
   \bv'{}^\pi\bigr)  \Bigr) \right)\\
  & = 4\int_0^{+\infty}\rd t \expp{-t} \left(1\wedge \dghn\bigl(r_{t+h_0}(T,\bv),r_{t+h_0}(T',\bv')\bigr)\right)\\
& \le 4\expp{h_0}\dlghn\bigl((T,\bv),(T',\bv')\bigr),
\end{align*}
where we used for the second inequality that $\dghn\bigl((\tilde T,\bv^\pi),
(\tilde T',
\bv'{}^\pi)\bigr)=\dghn\bigl((\tilde T,\bv), (\tilde T, \bv')\bigr)$ for $(\tilde T,
\bv), (\tilde T', \bv')\in \TKn$.
The continuity of the map follows.
\end{proof}

We shall also consider the set of trees whose root is not a branching
vertex:
\begin{equation}
  \label{eq:del-TL-br}
  \TLnb=\bigl\{(T, \bv)\in \TLn\, \colon\, \root\not\in \mathrm{Br}(T)\bigr\}.
\end{equation}
We shall simply write $\TLb$ for $\TLnb$ when $n=0$.
 \begin{lem}
   \label{lem:root-no-branching}
   The set $\TLnb$ is a Borel subset of $\TLn$.
 \end{lem}
\begin{proof}
For a rooted tree $T$, we define its diameter by $\diam(T)=\sup \{d(x,y)\,
\colon\, x, y\in T\}$. Notice that $H(T)\leq  \diam(T)\leq  2 H(T)$.
Clearly the function $\diam$ is constant on all equivalent classes of
$\TKn$ and thus of $\TLn$. If $\diam(T)=2H(T)<+\infty $, then we deduce
that the root is a branching vertex. Recall $\Pi_n^\circ$ for \eqref{eq:def-Pi0}.  More generally, we get that:
\[
  \TLnb=\bigcup_{n\in \N^*} D_{1/n}
  \quad\text{with}\quad
  D_t=\Bigl\{T\in \TLn\, \colon\, \diam\bigl(r_{t}\circ\Pi_n^\circ(T)\bigr)=2t\Bigr\}.
\]
Since the functions $\diam$,  $r_t$ and $\Pi_n^\circ$ are continuous, we deduce that
$D_t$ is closed, and hence $\TLnb$ is a Borel subset of $\TL$.
 \end{proof}

 We now   define  the set  of  discrete
 trees.  We say  that  a rooted  $n$-pointed  tree $(T,  d,  \bv)$ is  a
 discrete  tree if  $T$ is  equal  to the  tree spanned  by its distinguished
 vertices: $T=\Span^\circ(T,  \bv)$. We  define the set  of (equivalence
 classes of) discrete trees with at most $n$ leaves as:
  \begin{equation}
   \label{eq:def-TDn}
   \TDn=\bigl\{(T, \bv)\in \TLn\, \colon\, (T, \bv)=\Span(T,  \bv)\bigr\}.
 \end{equation}
 As a direct consequence of the continuity of the map
 $\Span$ we get the following result.
 \begin{lem}
   \label{lem:discrete=closed}
   Let $n\in \N$. The set of discrete trees $\TDn$ is  a closed subset  of $\TKn$  and of
 $\TLn$.
 \end{lem}
We end this section with partial measurability result on the number of
vertices at a given height of a tree.

  \begin{rem}
    \label{rem:cont-T-d}
    It     is      immediate     to      check     that      the     map
    $(T,  \bv)  \mapsto  \bigl(d(v_i,v_j),\,  i,j\in \{0,  \ldots,  n\}\bigr)$  is
    injective  $1/2$-Lipschitz   continuous  from  $(\TLn,   \dlghn)$  to
    $\R^{(n+1)\times (n+1)}$ endowed with the supremum norm (\emph{i.e.}
    the  maximum  of the  distances  between  coordinates). It  is  also
    bi-measurable  thanks  to  Lusin's  theorem.
\end{rem}

Let   $\TLnl$ be  the set of
trees with no leaves:
\[
  \TLnl=\bigl\{T\in \TL\, \colon\, \leaf(T)=\emptyset\bigr\}.
\]
For $T\in \TLnl$ and $t\geq 0$, let $\tilde
N_t(T)$ denotes the number of vertices at height $t$ of $T$:
\begin{equation}
   \label{eq:def-tildeN}
  \tilde N_t(T)=\Card\Bigl(\bigl\{ x\in T\, \colon\,
      H(x)=t\bigr\}\Bigr).
\end{equation}
It is easy to prove (and left as an exercise to the reader)  that
$\tilde N_t(T)$ is finite using that $T$ is 
locally compact without leaves. 
We have the following result.
\begin{lem}[Measurability of $\tilde N_t$]
  \label{lem:Nt-meas}
The set $\TLnl$ is a Borel subset of $\TL$ and the map $(t, T)\mapsto \tilde N_t(T)$ is
measurable from $\R_+ \times\TLnl$ to $\N$.
\end{lem}

\begin{proof}
  Let $t\geq 0$ and let $ \Theta_n(t)$ be the set of discrete trees such that all the
distinguished vertices (but the root) are leaves at height $t$:
\[
\Theta_n(t)  =\bigl\{T\in \TDn\, \colon\, d(\root, v_i)=t \text{ and } d(v_i,
v_j)> 0 \text{ for all } i,j\in \{1, \ldots, n\}\bigr\}.
\]
Thanks to Remark~\ref{rem:cont-T-d},  $\Theta_n(t)$ is a Borel set of $\TDn\subset\TKn\subset\TLn$.
For $T\in \TL$, we get that $\bigl\{T'\in \TDn\, \colon\, \Pi_n^\circ
(T')=T\bigr\}$ is finite. We deduce from  Lusin's  theorem that  $\Pi_n^\circ$ restricted to $\TDn$ is
bi-measurable. This implies  that the set
$\Pi_n^\circ(\Theta_n(t))$ is a Borel subset of $\TL$. We deduce that the set of
trees with no leaves, $\TLnl$,  which is formally defined by:
\[
 \TLnl=\bigcap _{k\in \N^*} \bigcup
  _{n\in \N} r_k^{-1} \Bigl(\Pi_n^\circ\bigl(\Theta_n(k)\bigr)\Bigr),
\]
is a Borel subset of $\TL$.
We also get that  $\bigl\{T\in \TLnl\, \colon\, \tilde N_t(T)=n\bigr\}= r_t^{-1}
\Bigl(\Pi_n^\circ\bigl(\Theta_n(t)\bigr)\Bigr)$; this implies that the map $\tilde N_t$ is
measurable. Since $t\mapsto \tilde N_t(T)$ is non-decreasing and
left-continuous, we deduce that the map $(t, T)\mapsto \tilde N_t(T)$ is
measurable from $\R_+ \times\TLnl$ to $\N$.
\end{proof}

 \subsection{Grafting a discrete tree on another one}
 \label{sec:graft1}

 We    define,   in    a    slightly   more    general   context    than
 Section~\ref{sec:graft-int}, the  grafting of a locally  compact rooted
 tree at  a distinguished  vertex of an  another locally  compact rooted
 tree.     For
 $(T,    \bv)     \in    \TLn$     and    $(T',\bv')\in     \TLk$    and
 $i\in      \{0,      \ldots,      n\}$,     with      $n,k\geq      0$,
 $\bv=(v_0=\root, \ldots, v_n)$  and $\bv'=(v'_0=\root', \ldots, v'_k)$,
 we  define the  tree $T\circledast_i  T'$ by~\eqref{eq:graft-2dis-tree}
 and the distance $d^\circledast$ by~\eqref{eq:graft-2dis-dist} with $x$
 replaced   by   $v_i$,   and  consider   the   distinguished   vertices
 $ \bv\circledast\bv'=(v_0=\root, \ldots, v_n, v'_1, \ldots, v'_k)$.

\begin{lem}[Continuity of the grafting map]
  \label{lem:graft-meas}
Let $n, k\in \N$ and $i\in \{0, \ldots, n\}$. The  map
$\bigl((T, \bv),(T',\bv')\bigr)\mapsto (T\circledast_i
T',\bv\circledast\bv')$,  is continuous from $\TLn\times \TLk$ to $\TLnk$.
\end{lem}

\begin{proof}
  Let      $(T_1,\bv_1),(T'_1,\bv'_1)\in       \TLn$       and
  $(T_2,  \bv_2),(T'_2,  \bv'_2)\in   \TLk$.  Set  $T=T_1\circledast_{i}T_2$, $T'=T_1'\circledast_{i}T'_2$,
  $\bv=\bv_1\circledast\bv_2$,
   and
  $\bv'=\bv'_1\circledast\bv'_2$.
  \medskip

  First     suppose  that  the     trees     are     compact,    that     is
  $(T_1,\bv_1),(T'_1,\bv'_1)\in                 \TKn$
  and
  $(T_2, \bv_2),(T'_2, \bv'_2)\in \TKk$. Let $\cR_1$ be a correspondence
  between (elements  of the  classes) $(T_1,\bv_1)$ and  $(T'_1,\bv'_1)$ and
  let  $\cR_2$ be  a correspondence  between (elements  of the  classes)
  $(T_2, \bv_2)$  and  $(T'_2, \bv'_2)$. We  set  $\cR=\cR_1\cup\cR_2$  with $\root_2$  and
  $\root'_2$ replaced respectively by $v_i$ and  $v'_i$.  It defines a correspondence
  between               $(T,\bv)$              and
  $(T',\bv')$.              For            every
  $(x,x'),(y,y')\in \cR$, we have:
\[
\bigl|d^\circledast(x,y)-d'{}^\circledast (x',y')\bigr|\\
=\begin{cases}
\bigl|d_1(x,y)-d'_1(x',y')\bigr|\le \dist(\cR_1) & \text{if }(x,x'),(y,y')\in\cR_1,\\
\bigl|d_2(x,y)-d'_2(x',y')\bigr|\le \dist(\cR_2) & \text{if }(x,x'),(y,y')\in\cR_2,
\end{cases}
\]
and if $(x,x')\in\cR_1$ and $(y,y')\in\cR_2$, we have:
\begin{align*}
  \bigl|d^\circledast(x,y)-d'{}^\circledast (x',y')\bigr|
  &=\bigl|d_1(x,v_i)+d_2(\rho_2,y)-d'_2(\rho'_2,y')-d'_1(x',v'_i)\bigr|\\
&\le \bigl|d_1(x,v_i)-d_1(x',v'_i)\bigr|+\bigl|d_2(\root_2,y)-d'_2(\root'_2,y')\bigr|\\
&\le \dist(\cR_1)+\dist(\cR_2).
\end{align*}
This gives:
\begin{equation}
\label{eq:dgh-circledast}
  \dghnk\bigl((T,\bv),
  (T',\bv')\bigr) \\
  \le\dghn\bigl((T_1,\bv_1),(T'_1,\bv'_1)\bigr)
  +\dghk\bigl((T_2,\bv_2),(T'_2, \bv'_2)\bigr).
\end{equation}

Now consider  $(T_1,\bv_1),(T'_1,\bv'_1)\in       \TLn$
and  $(T_2,  \bv_2),(T'_2,  \bv'_2)\in   \TLk$.
Without loss of
generality we assume that $H(v'_i )\ge
H(v_i)$. Remark that, for every $t\ge 0$, we have, with $a_+=\max(a, 0)$:
\[
r_t(T,\bv)=r_t(T_1,\bv_1)\circledast_{i}r_{(t-H(v_i))_+}(T_2, \bv_2).
\]
Therefore, we have:
  \begin{multline*}
\dlghnk\bigl((T,\bv),(T',\bv')\bigr) \\
  \begin{aligned}
& =\int_0^{+\infty} \!\! \rd t \expp{-t} \left(1\wedge
\dghnk\bigl(r_t(T,\bv),r_t(T',\bv')\bigr)\right)\\
& =\int_0^{+\infty} \!\!\rd t \expp{-t}\left(1\wedge \dghnk
\left(r_t(T_1,\bv_1)\circledast_{i}r_{(t-H(v_i))_+}(T_2, \bv_2),
  r_t(T'_1,\bv'_1)\circledast_{i}r_{(t-H(v'_i))_+}(T'_2, \bv'_2)\right)\right)\\
& \le \int_0^{+\infty} \!\! \rd t \expp{-t}\left(1\wedge
\dghn\left(r_t(T_1,\bv_1),r_t(T'_1,\bv'_1)\right)\right)\\
& \hspace{4cm} +\int_0^{+\infty} \!\! \rd t \expp{-t}\left(1\wedge
\dghk\left(r_{(t-H(v_i))_+}(T_2, \bv_2),r_{(t-H(v'_i))_+}(T'_2, \bv'_2)\right)\right)\\
& \le \dlghn\bigl((T_1,\bv_1),(T'_1,\bv'_1)\bigr)+4\expp{-H(v'_i)}\dlghk\bigl((T_2,
\bv_2),(T'_2, \bv'_2)\bigr)+H(v'_i)-H(v_i)\\
& \le 3\, \dlghn\bigl((T_1,\bv_1),(T'_1,\bv'_1)\bigr)+4\dlghk\bigl((T_2,
\bv_2),(T'_2, \bv'_2)\bigr),
  \end{aligned}
\end{multline*}
where we used Equation \eqref{eq:dgh-circledast} for the first
inequality and Lemma \ref{lem:cont-rt} for the second one. This completes the proof.
\end{proof}


We shall use a version of the grafting procedure where, instead of
grafting on $v_i$, we shall graft on the branch $\lb \root, v_i\rb$ at
height $h$ provided that $H(v_i)\geq h$.
Let  $n\in  \N$  and  $i\in  \{0,   \ldots,  n\}$  be  given.  For  $h\in \R_+$ and $(T,\bv)\in\TKn$,  we denote by $x_{i,h}$ the unique vertex of $T$ that satisfies $x_{i,h}\in \lb \root, v_i\rb$ and $H(x_{i,h})= H(v_i) \wedge h$. Then, the  map
$\bigl(h,(T, \bv)\bigr) \mapsto \bigl(T,  (\bv,x_{i,h})\bigr)$
is clearly continuous from $\R_+\times \TLn$ to $\TLnn$.  We then define
the grafting map $\circledast_{i,h}$ by:
\begin{equation}
   \label{eq:map-graft-ih}
  \bigl(h,(T, \bv),(T',\bv')\bigr)\mapsto T\circledast_{i,h}
  T'= (T\circledast_{i,h}
  T',\bv\circledast\bv'),
\end{equation}
as  the composition  of
\begin{itemize}
\item[] [adding the vertex $x_{i,h}$]:  $\bigl(h,(T,  \bv)\bigr) \mapsto  (T,  \tilde \bv)$  with
$\bv=(v_0=\root,              \ldots,             v_n)$              and
$\tilde \bv=(\bv,x_{i,h})=(\tilde v_0=\root,  \ldots, \tilde v_n=v_n, \tilde
v_{n+1}=x_{i,              h})$,\\
\item[] [grafting]:
$\bigl((T,   \tilde \bv),    (T',\bv')\bigr)   \mapsto   (T\circledast_{n+1}
T',\tilde \bv\circledast \bv')$ and
\item[] [removing the $(n+1)$-th distinguished vertex]:    $(T''=T\circledast_{n+1}
T',\tilde \bv\circledast \bv')  \mapsto (T'', \bv\circledast \bv')$.
\end{itemize}
 Since all those maps
 are continuous, we get the following result.

 \begin{lem}[Continuity of the grafting map $\circledast_{i,h}$]
  \label{lem:graft2-cont}
Let $n,k\in \N$, $i\in \{0, \ldots, n\}$.   The map $\bigl(h, (T,\bv),(T',\bv')\bigr) \mapsto T\circledast_{i,h}
  T'$ is continuous from
$\R_+\times \TLn\times \TLk$ to $\TLnk$.
\end{lem}

\subsection{Set of (equivalence classes of) marked trees}
\label{sec:TnSn}
We shall consider trees with a marked infinite branch; for this reason
we introduce the notion of marked trees. In this part, we do not record
an order on the marked vertices as in the $n$-pointed trees.

We say that $(T, S, d, \root)$ is a marked rooted
tree if $(T, d, \root)$ is  a  rooted tree and
 the set of  marks $S$ is a subtree of $T$ with the same root (that is $\root \in S$)
 endowed with the restriction of the distance $d$.
 A correspondence between two compact marked rooted  trees $(T,S,d,\root)$ and $(T', S', d',
 \root')$ is a set $\cR\subset T\times T'$ such that $\cR$ is a correspondence
 between $(T, d, \root)$ and $(T', d', \root')$ and $\cR\cap
 (S\times S')$ is also  a correspondence
 between $(S, d, \root)$ and $(S', d', \root')$. Then, we set:
 \[
   \dghd\bigl((T,S),(T',S')\bigr)=\inf \, \inv{2}\, \dist(\cR)  ,
\]
where the  infimum is taken  over all the correspondences  $\cR$ between
$(T, S,  d, \root)$ and  $(T', d', S',  \root')$.  An easy  extension of
\cite{adh}   gives  that   $\dghd$  is   a  pseudo-distance,   and  that
$\dghd(T, T')=0$ if and only if there exists an isometric one-to-one map
$\varphi$ from $(T, d)$ to $(T', d')$ which preserves the root and which
is    also    one-to-one   from    $S$    to    $S'$.    The    relation
$\dghd((T,S),  (T',S'))=0$ defines  an  equivalence  relation.  The  set
$\TKd$   of  equivalence   classes  of   compact  marked   rooted  trees
$(T, S, d,\root)$ endowed with $\dghd$ is then a metric Polish space. We
simply  write $(T,  S)$ for  $(T, S,  d, \root)$,  and unless  specified
otherwise, we shall denote also  by $(T,S)$ its equivalence class. Since
\begin{equation}
  \label{eq:majo-pour-P}
 \dgh(T,  T') \vee  \dgh(S, S')\leq  \dghd((T,S), (T',  S')),
\end{equation}
we deduce
that the  map $(T,S)  \mapsto (T,S)$ from  $\TKd$ to  $(\TK)^2$ (endowed
with  the  maximum distance  on  the  coordinates) is  continuous.
For
$t\geq 0$,  we define  the truncation function  $r_t^{[2]}$ of  a marked
rooted   tree   $(T,S,   d,   \root)$  as   the   marked   rooted   tree
$r_t^{[2]}(T,S)=(r_t(T), r_t(S),  d,\root)$, where we recall that
$r_t(T)=\{x\in T\, \colon\, H(x)\leq  t\}$.
 If $(T,S)$ and $(T',S')$ are in the
same  equivalence   class  of   $\TKd$,  so  are   $r_t^{[2]}(T,S)$  and
$r_t^{[2]}(T',S')$; thus  the function  $r_t^{[2]}$ can be seen  as a
map from $\TKd$ to itself.  Similarly to \eqref{eq:cont-rt-gh}, we have
for $t,s\geq 0$ and $(T,S), (T', S')\in \TKd$:
\begin{equation}
  \label{eq:cont-rt-gh2}
    \dghd \left(r_t^{[2]}(T,S), r_{t+s}^{[2]}(T', S')\right) \leq  4 \,
    \dghd\bigl((T,S), (T',S')\bigr)+s.
\end{equation}
This implies
that the map $(t, (T,S)) \mapsto r_t^{[2]}(T,S)$ is continuous from $\R_+\times
\TKd$ to $\TKd$.

A marked rooted tree $(T, S,d, \root)$ is locally compact if $r_t^{[2]}(T,S)$ is a
compact marked rooted tree for all $t\geq 0$. Following \cite{adh}, we consider
for two locally compact marked rooted trees $(T,S)$ and $(T',S')$:
\begin{equation}
   \label{eq:def-dlghd}
  \dlghd((T,S),(T',S'))=\int_0^\infty \expp{-t} \rd t   \left(1 \wedge
    \dghd \left(r_t^{[2]}(T,S), r_t^{[2]}(T',S')\right)
  \right).
\end{equation}
Furthermore,  we have  that $\dlghd((T,S),  (T',S'))=0$ if  and only  if
there  exists an  isometric one-to-one  map $\varphi$  from $(T,  d)$ to
$(T',  d')$ which  is  one-to-one from  $S$ to  $S'$  and preserves  the
roots.   Thus the   relation   $\dlghd((T,S),  (T',S'))=0$   defines   an   equivalence
relation, see \cite[Proposition 5.3]{ADH14}. The set $\TLd$ of equivalence classes of locally compact
marked rooted
trees  $(T, S,d,  \root)$ endowed  with $\dlghd$  is then  a metric  Polish
space. Furthermore,  $\TKd$ is  an open dense  subset of  $\TLd$.
Combining~\eqref{eq:majo-pour-P} and the definition of $\dlghd$, we get
the elementary following result.

\begin{lem}[Regularity of the projection]
  \label{lem:projo-cont}
  The map ${\rm P} \, \colon\, (T,S) \mapsto T$  from  $\TLd$ to  $\TL$
  is 1-Lipschitz. 
\end{lem}

Similar
equations   to    \eqref{eq:cont-lgh-gh},   \eqref{eq:cont-rt-lgh}   and
\eqref{eq:cont-rt-lgh-gh} holds with $\dlghn$ and $\dghn$ replaced by
$\dlghd$ and $\dghd$. For future use, let us give the equations
corresponding to~\eqref{eq:cont-rt-lgh}
and~\eqref{eq:cont-rt-lgh-gh}. For  $(T,S), (T', S')\in
\TLd$ and $s,t\geq 0$, we have:
\begin{align}
  \label{eq:cont-rt-lgh-2}
  \dlghd \left(r_t^{[2]}(T, S), r_{t+s}^{[2]}(T', S')\right)
  &\leq  4 \, \dlghd\bigl((T,S),(T', S')\bigr)+s,\\
  \label{eq:cont-rt-lgh-gh-2}
    \dghd \left(r_t^{[2]}(T, S), r_t^{[2]}(T',S')\right)
&\leq  4 \expp{t}\,  \dlghd\bigl((T, S), (T', S')\bigr).
\end{align}
We also we have the following result consequences of
\eqref{eq:cont-rt-gh2} and
\eqref{eq:cont-rt-lgh-2}.

\begin{lem}[Continuity of the truncation map]
   \label{lem:cont-rt-loc}
Let $n\in \N$.
The map $\bigl(t, (T,S)\bigr) \mapsto r_t^{[2]}(T, S)$ is continuous from $\R_+\times
\TKd$ to $\TKd$ and from $\R_+\times
\TLd$ to $\TLd$ (and to $\TKd$).
\end{lem}

We give in the next lemma an example of a $\TKd$ and $\TLd$ valued function.
\begin{lem}[Continuity of $\Span^\circ$]
  \label{lem:cont-span-loc}
  Let             $n\in            \N$.              The            map
  $(T, d, \bv)  \mapsto \bigl(\Pi_n^\circ(T), \Span^\circ(T, \bv), d,  \root\bigr)$ from $\TLn$ to $\TLd$
  (resp. from $\TKn$
  to $\TKd$) is  injective, bi-measurable and
  $16$-Lipschitz (resp.  $4$-Lipschitz) continuous.
\end{lem}

\begin{proof}
  We first consider the compact  case.  Let $(T,\bv)$ and $(T',\bv')$ be
  rooted $n$-pointed  compact trees and let $\cR$  be a correspondence
  between them.   Recall the definition of  $p_\bv$ in \eqref{eq:def-pv}
  as  the  projection  on $\Span^\circ(T,\bv)$  and  the  correspondence
  $\tilde      \cR$       from~\eqref{eq:def-Tilde-R}.       We      set
  $\cR^{[2]}=\cR \cup  \tilde \cR$.  By  construction $\cR^{[2]} $  is a
  correspondence    between     $\bigl(T,    \Span^\circ(T,     \bv)\bigr)$    and
  $\bigl(T',   \Span^\circ(T',   \bv')\bigr)$.    From    the   proof   of   Lemma
  \ref{lem:Lip-Span},              we              get              that
  $\dist(\cR^{[2]}) \leq 4 \, \dist(\cR)$. This directly implies that:
 \begin{equation}
   \label{eq:double-rt-span}
   \dghd\Bigl(\bigl(T, \Span^\circ (T, \bv)\bigr), \bigl(T',  \Span^\circ (T',
   \bv')\bigr)\Bigr)\leq  4\, \dghn\bigl((T, \bv), (T', \bv')\bigr).
 \end{equation}
  This gives that the map $(T, d, \bv) \mapsto \bigl(T, \Span^\circ(T, \bv), d, \root\bigr)$ from
  $\TKn$ to $\TKd$ is $4$-Lipschitz   continuous.
  \medskip

We now consider the locally compact case. Let $(T,\bv)$ and $(T',\bv')$
belong to $\TLd$. We have:
\begin{multline*}
  \dlghd \Bigl(\bigl(T, \Span^\circ (T, \bv)\bigr), \bigl(T',  \Span^\circ
    (T',\bv')\bigr)\Bigr)\\
\begin{aligned}
  &= \int_0^\infty \expp{-t} \rd t   \left(1 \wedge
    \dghd \left(r_t^{[2]}\bigl((T, \Span^\circ (T, \bv)\bigr), r_t^{[2]}\bigl(T',
    \Span^\circ (T',\bv')\bigr)\right)   \right)  \\
  &\leq  4 \int_0^\infty \expp{-t} \rd t   \left(1 \wedge
    \dghd \Bigl(\bigl(r_t(T, \bv), \Span^\circ (T, \bv)\bigr), \bigl(r_t(T',\bv'),
    \Span^\circ (T',\bv')\bigr)\Bigr)   \right)  \\
  &\leq  16 \int_0^\infty \expp{-t} \rd t   \left(1 \wedge
    \dghn \left(r_t(T, \bv),r_t(T',\bv')\right)\right)   \\
  &=  16 \, \dlghn \left((T, \bv),(T',\bv')\right),
\end{aligned}
\end{multline*}
where we used~\eqref{eq:cont-rt-gh2}  (with
$T$ and $S$ replaced respectively by $r_t(T, \bv)$ and $ \Span^\circ (T,
\bv)$ and similarly for $T'$ and $S'$) for the first inequality,
and~\eqref{eq:double-rt-span}  (with $(T, \bv)$ replaced
by $r_t(T, \bv)$) as well as  the relation
$\Span^\circ(r_t(T, \bv))=\Span^\circ(T, \bv)$ for the second.    This gives that the map $(T, d, \bv) \mapsto (T, \Span^\circ(T, \bv), d, \root)$ from
$\TLn$ to $\TLd$ is $16$-Lipschitz   continuous.

Clearly those maps are injective  and thus bi-measurable thanks to   Lusin's theorem.
\end{proof}

  \begin{rem}
  \label{rem:incoherent}
  Let us  stress that for  $(T,\bv)$ a rooted $n$-pointed  compact tree,
  the                            rooted                            tree
  $r_t^{[2]}\bigl(T,  \Span^\circ(T,  \bv)\bigr)= \Bigl(  r_t(T),  r_t\bigl(\Span^\circ(T,
  \bv)\bigr)\Bigl)$ and the rooted tree $ \Bigl(r_t(T), \Span^\circ\bigl(r_t(T,
  \bv)\bigr)\Bigr)=\bigl(r_t(T), \Span^\circ(T, \bv)\bigr)$ differ if
  and only if $t$ is smaller than the height of $\Span^\circ(T, \bv)$.
\end{rem}

\medskip
Let $(T, S, d,  \root)$ be a marked locally
compact rooted tree. To simplify, we shall only write $(T,S)$ for $(T, S, d,
\root)$. We define the projection of  $z\in T$ on $S$,
$p_S(z) \in S$, as the element of $S$ uniquely defined by:
  \[
    \lb  \root, p_S(z)\rb=  \lb  \root,  z \rb  \cap  S.
  \]
  Now, we  consider the truncation of  a marked tree at  a given height,
  say   $t$,   of    the   marked   subtree.   For    $t\geq   0$   and
  $\varepsilon\in  \{-, +\}$,  we  set:
\begin{equation}
   \label{eq:def-r2+}
  r_{t}^{[2], \varepsilon}(T, S)=\left( r_{t,1}^{[2], \varepsilon}(T,S),  r_{t}(S)\right)
\end{equation}
with:
\begin{align*}
     r_{t,1}^{[2],+}(T,S) & =\Bigl\{x\in T\, \colon\, H\bigl(p_S(x)\bigr)\leq t \Bigr\},\\
   r_{t,1}^{[2], -}(T,S) & = \Bigl\{x\in T\, \colon\,  H\bigl(p_S(x)\bigr)< t
                           \Bigr\}\cup\bigl\{x\in S\,\colon\,H(x)=t\bigr\}.\\
\end{align*}

See Figure~\ref{fig:r2T} for an instance of $ r_{t}^{[2],
  \varepsilon}(T, S)$, where $S$ is an infinite branch.
For    $\varepsilon\in     \{+,    -\}$,     we    also     denote    by
$r_{t}^{[2],\varepsilon}(T,     S)$    the     marked    rooted     tree
$\bigl(r_{t}^{[2],\varepsilon}(T, S), d, \root\bigr)$ endowed with the restriction
of  the  distance   $d$  and  the  root  $\root$.
Furthermore, if  $(T,S)$ and $(T',  S')$ belong to the  same equivalence
class of $\TLd$ or $\TKd$,  then so do $r_{t}^{[2],\varepsilon}(T, S)$
and     $r_{t}^{[2],\varepsilon}(T',     S')$.      Thus     the     map
$\bigl(t, (T,S)\bigr) \mapsto r^{[2],\varepsilon}_{t} (T,S)$ is a well defined map
from $\R_+\times \TLd$ to $\TLd$ for $\varepsilon\in \{+, -\}$.

\begin{figure}[H]
\includegraphics[scale=0.5]{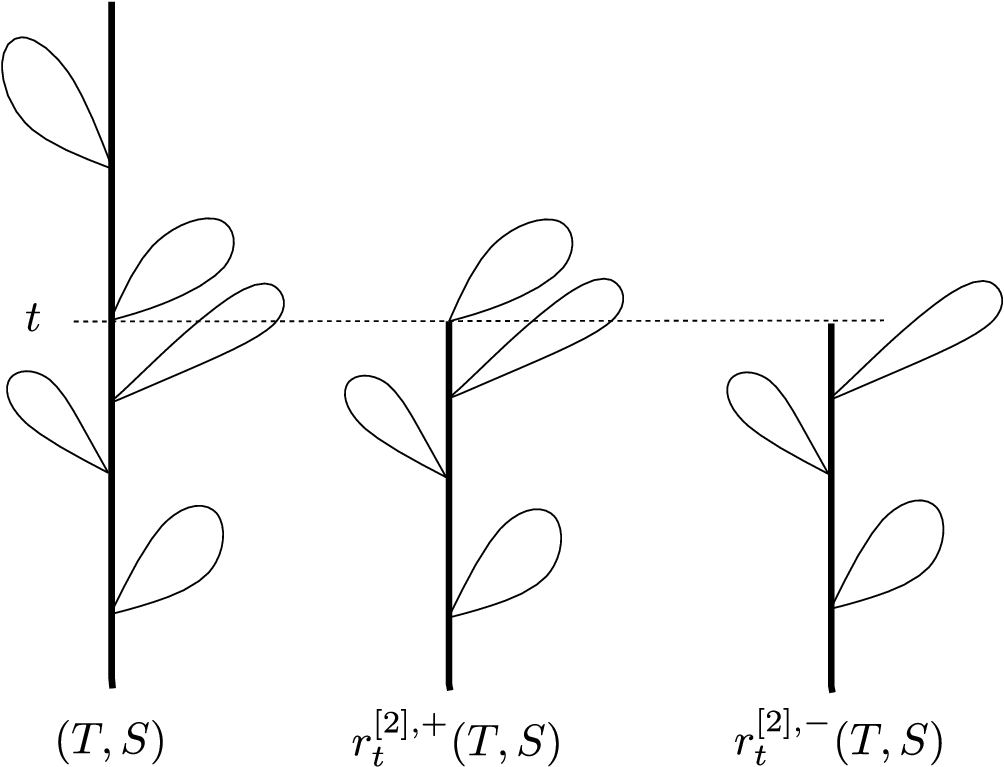}
\caption{Example of restrictions of a tree $T$ with a marked spine $S$
  (in bold).}
\label{fig:r2T}
\end{figure}

\begin{rem}[Examples]
   \label{rem:exple-rt2}
We give  elementary examples. For
$\varepsilon\in    \{+,     -\}$    and    $t>0$,    we     have    that
$r_{t}^{[2],\varepsilon}\bigl(T,\{\root\}\bigr)=\bigl(T,         \{\root\}\bigr)$        and
$r_{0}^{[2],-}\bigl(T,\{\root\}\bigr)=\bigl(\{\root\},    \{\root\}\bigr)$   as    well   as
$r_{0}^{[2],+}\bigl(T,\{\root\}\bigr)=\bigl(T,   \{\root\}\bigr)$.    We   also   have   for
$t\in   \R_+$   that  $r_{t   }^{[2],\varepsilon}(T,T)=\bigl(r_t(T),r_t(T)\bigr)$.
\end{rem}

\begin{rem}[The map $r_t^{[2], \varepsilon}$ is not continuous]
   \label{rem:cont-r2}
 Let $\varepsilon\in \{+, -\}$ and $t>0$. The function $r_{t}^{[2],\varepsilon}$ is
 not continuous from $\TLd$ to itself. Indeed take $t=1$ without loss of
 generality and consider
 $T=[0, 2]$ and $S_\delta=[0, \delta]$, with $\delta\in [0, 2]$,  $\root=0$ and
 the Euclidean distance. Notice that $\bigl([0, 1], [0, 1]\bigr)=(S_1, S_1)\neq (T, S_1)$. Then we have that
 $\lim_{\delta\rightarrow 1} \dghd\bigl((T,S_\delta), (T, S_1)\bigr)=0$,
 $ r_{1}^{[2],\varepsilon}(T, S_\delta)=(T,S_\delta)$ for $\delta<1$,
 $ r_{1}^{[2],\varepsilon}(T, S_\delta)=(S_1, S_1)$ for $\delta>1$,
 $ r_{1}^{[2],-}(T, S_1)=(S_1, S_1)$
 and
$ r_{1}^{[2],+}(T, S_1)=(T, S_1)$.
\end{rem}

We  have  the  following  measurability   result.

\begin{lem}[Measurability  of some truncation maps]
  \label{lem:rt2}
Let $\varepsilon\in \{+, -\}$.  The map $\bigl(t, (T,S)\bigr) \mapsto
r^{[2], \varepsilon}_{t} (T,S)$ is measurable from
  $\R_+\times \TLd$ to $\TLd$.
\end{lem}

\begin{proof}
  Let $a>0$. For  a marked tree $(T,S)=(T, S, d,  \root)$, we define its
  partial dilatation  $R_a(T, S)=(T,S, d_a,  \root)$ as the  marked tree
  with   $d_a(x,y)=a   d\bigl(x,p_S(x)\bigr)+   d\bigl(p_S(x),   p_S(y)\bigr)   +   a   d\bigl(y,
  p_S(y)\bigr)$ if $p_S(x)\neq p_S(y)$ and
 $d_a(x,y)=a d(x,y)$ if $p_S(x)= p_S(y)$.
Intuitively  the distances  on  $T$  are multiplied  by  $a$
  outside $S$.  The equivalence class of $R_a(T, S) $ in $\TLd$ does not
  depend of the choice of $(T,S)$ in its equivalence class in $\TLd$; so
  the map $R_a$ is well defined on $\TLd$ to itself. Notice that the map
  $R_a$  is continuous  and one-to-one  with inverse  $R_{1/a}$.  It  is
  immediate to check that, for $t\geq 0$:
\[
  r^{[2],-}_{t}=\lim_{a\rightarrow 0+ } R_{1/a} \circ r_t^{[2]} \circ R_a.
\]
This and Lemma \ref{lem:cont-rt-loc} imply the
measurability of the map  $\bigl(t, (T,S)\bigr) \mapsto r^{[2], -}_{t} (T,S)$.
Then, notice that $\lim_{s \downarrow t} r^{[2], -}_{s}=r^{[2],+}_{t}$ to
get the
measurability of the map  $(t, (T,S)) \mapsto r^{[2],+}_{t} (T,S)$.
\end{proof}

We end  this section  by proving (in  a very similar  way) that  the map
$ r_{*}^{[2]}$ below,  which consists in cleaning the root, that is, in erasing the bushes  at the root
of a marked  tree is measurable.  For $(T,S)=(T, S,  d, \root)$ a marked
locally compact rooted tree, we set:
\begin{align}
\label{eq:r*[2]}
  r_{*}^{[2]}(T, S)=\left( r_{*,1}^{[2]}(T,S),  S\right)
  \quad\text{with}\quad
  r_{*,1}^{[2]}(T,S)=\bigl\{x\in T\, \colon\, p_S(x)\neq \root\bigr\}\cup\{\root\}.
\end{align}
We    also     denote    by
$r_{*}^{[2]}(T,     S)$    the     marked    rooted     tree
$\bigl(r_{*}^{[2]}(T, S), d, \root\bigr)$ endowed with the restriction
of  the  distance   $d$  and  the  root  $\root$.
Furthermore, if  $(T,S)$ and $(T',  S')$ belong to the  same equivalence
class of $\TLd$,  then so do $r_{*}^{[2]}(T, S)$
and     $r_{*}^{[2]}(T',     S')$.      Thus     the     map
$ r^{[2]}_{*}$ is well-defined
from $\TLd$ to $\TLd$.
\begin{lem}[Measurability  of the root cleaning  map]
  \label{lem:rt*}
  The map $r_*^{[2]}$  is measurable from
  $\TLd$ to $\TLd$.
\end{lem}

\begin{proof}
  Let $a>0$. For  a marked tree $(T,S)=(T, S, d,  \root)$, we define its
  partial dilatation $R'_a(T, S)=(T,S, d'_a,  \root)$ as the marked tree
  with     $d'_a(x,y)=F_a(t)d(x,y)$      if     $p_S(x)=p_S(y)$     with
  $t=H\bigl(p_S(x)\bigr)$,                 and                 otherwise
  $d'_a(x,y)=F_a(t)    d\bigl(x,    p_S(x)\bigr)+    a    d\bigl(p_S(x),
  p_S(y)\bigr)+      F_a(s)       d\bigl(y,      p_S(y)\bigr)$      with
  $t=H\bigl(p_S(x)\bigr)$,  $s=H\bigl(p_S(y)\bigr)$,  and  the  function
  $F_a$ defined for $t\geq 0$ by $  F_a(t)= t \wedge a + a^{-2}(a-t)_+ $
  if $a\leq  1$, and $F_a(t)=1/F_{1/a}(at)$  if $a>1$.  Notice  that for
  $x\in T\setminus\{\root\}$,  we have, as  $a$ goes down to  $0$, that:
  $d'_a(x,    \root)    \sim    a     d(x,    \rho)$    as    well    as
  $d'_{1/a}(x, \root) \sim a^{-1} d(x, \rho)$ if $p_S(x)\neq \root$; and
  $d'_a(x,    \root)   \sim    a^{-1}   d(x,    \rho)$   as    well   as
  $d'_a(x, \root) \sim a d(x, \rho)$ if $p_S(x)= \root$.

 The  equivalence  class  of
  $R'_a(T, S) $ in $\TLd$ does not depend of the choice of $(T,S)$ in its
  equivalence  class in  $\TLd$; so  the map  $R'_a$ is  well defined  on
  $\TLd$  to  itself.  Notice  that  the map  $R'_a$  is  continuous  and
  one-to-one with inverse $R'_{1/a}$.  It is immediate to check that for
  $t>0$:
\[
  r^{[2]}_{*}=\lim_{a\rightarrow 0+ } R_{1/a} \circ r_t^{[2]} \circ R_a.
\]
This and Lemma \ref{lem:cont-rt-loc} imply the
measurability of the map  $ r^{[2]}_{*}$.
\end{proof}

\subsection{Set of (equivalence classes of) trees with one infinite marked branch}
 \label{sec:TLs}
Let us denote by $\Tz=(\root, \{\root\})$ the rooted tree reduced to its root.  Notice that
$r_{0}^{[2],+}(T,     S)=\{(\Tz,\Tz)\}$      if     and      only     if
$\lb \root ,x \rb \cap S  = \{\root\}$ implies $x=\root$.    Let
$T_1=([0, \infty ),  d, 0)$ be the tree consisting  of only one infinite
branch. We  consider the
set (of  equivalence classes) of  locally compact rooted trees  with one
infinite marked branch and its subset of trees whose  root is not a branching vertex:
\begin{align}
  \label{eq:def-spine}
\TLs&=\left\{(T, S)\in \TLd\, \colon\, S=T_1 \text{ in } {\mathbb T}_{\text{loc-K}}\right\},\\
  \label{eq:def-spine-br}
  \TLsb
 & =\left\{(T, S)\in \TLs\, \colon\,
   \root \not\in \Br(T) \right\}.
\end{align}

\begin{lem}
  \label{lem:mes-T*}
  The sets $\TLs$ and  $\TLsb$ are  Borel subsets of $\TLd$.
\end{lem}

\begin{proof}
  Consider the projection $\tilde  \Pi:(T,S) \mapsto S$ from $\TLd$
  to    $\TL$,     which    is  by construction  $1$-Lipschitz and thus
  continuous.  As
  $\TLs=\tilde \Pi^{-1}\bigl(\{T_1\}\bigr)$, we get that $\TLs$ is Borel.

  Notice that for $(T, S)\in \TLs$, then, by definition of $r^{[2],
    +}_{t}$, we get that the root is not a branching vertex of $(T, S)$ if and only if
$r_{0}^{[2],+}(T, S)=(\Tz,
  \Tz)$.
  Then, the set  $\TLsb=\TLs\cap  (r_{0}^{[2],+})^{-1}\Bigl(\bigl\{(\Tz,
  \Tz)\bigr\}\Bigr)$ is Borel as the map $r^{[2],  +}_{0}$ is
  measurable according to Lemma~\ref{lem:rt2}.
\end{proof}

We shall be mainly consider elements of $\TLsb$ in what follows.
For simplicity, we shall write $T^*=(T,S)$ for an element of $\TLsb$.
For $t\geq 0$ and $T^*=(T,S)$ in $\TLsb$, we have $r^{[2], +}_t
(T^*)=\bigl(r^{[2], +}_{t,1}(T), r_t(S)\bigr)$ where the rooted tree $r_t(S)$ is given by $\bigl(\lb\root,x
\rb, \root\bigr)$ with $x\in S$ uniquely characterized by  $d(\root, x)=t$.
We shall consider a slight modification of $r^{[2], +}_t$ on $\TLsb$, say
$\tilde r^{[2], +}_t$,
where one keeps track only of $(\root,x)$ instead of $r_t(S)$:
\begin{equation}
   \label{eq:def-tilde-rt2}
  \tilde r^{[2], +}_t(T^*)=\bigl(r^{[2], +}_{t,1}(T), (\root,x)\bigr).
\end{equation}
It is  left to the reader  to check that  $\tilde r^{[2],  \varepsilon}_t$ is
defined   on    $\TLsb$   and    $\TLu$-valued.   Similarly to
Lemma~\ref{lem:rt2},   we get the following result.

\begin{lem}
   \label{lem:rt2*}
   The  function $(t,  T^*) \mapsto  \tilde r^{[2],  +}_{t} (T^*)$  from
   $\R_+\times \TLsb$ to $\TLu$ is measurable.
\end{lem}

\subsection{Another representation for discrete trees}
\label{sec:def-vA-wA}
Let    $n\in     \N$    be     fixed.     Let    $(T,     \bv)$,    with
$\bv=(v_0=\root, \ldots, v_n)$, be  a locally compact rooted $n$-pointed
tree. We will decompose the tree $\Span(T,\bv)$ as a sequence of edges.
To    do   so,   we   introduce    some   notations.  Let
$A  \subset  \{0, \ldots,  n\}$ be non-empty. We  set  $\bv_A=(v_i, i\in  A)$. We denote by $v_A$
the most recent common ancestor  of $\bv_A$, which is the only
element of $T$ such that:
\begin{equation}
   \label{eq:def-vA}
  \lb \root, v_A\rb=  \bigcap _{k\in A} \lb \root, v_k\rb.
\end{equation}
Notice that  $v_{\{i\}}=v_i$.  Recall  that for $x\in  T$, $T_x$  is the
subtree of $T$ above  $x$ and rooted at $x$.  Let  $\cp^+_n$ be the set
of all subsets  $A\subset\{1, \ldots, n\}$ such  that $A\neq \emptyset$.
For                  $A\in                 \cp^+_n$,                  if
$T_{v_A}   \cap  \Span^\circ(T,   \bv_{A^c})\neq   \emptyset$ with $A^c=\{0,1,2,\cdots,n\}\setminus A$,  we   set
$w_A=v_A$, otherwise  we define $w_A\in  \lb \root, v_A\rb$ as  the only
element of $T$ such that:
\begin{equation}
  \label{eq:def-wA}
 \lb \root, w_A\rb=   \Span^\circ(T, \bv_{A^c}) \cap  \Span^\circ\bigl(T,
 (\root, \bv_{A})\bigr).
\end{equation}
Equivalently $w_A$ is the only element  in $\lb \root, v_A\rb$ such that
$w_A=v_{A\cup \{k_0\}}$ for some $k_0\in A^c$ and for all $k\in A^c$, we
have   $v_{A\cup   \{k\}}\in   \lb    \root,   w_A\rb$.    Notice   that
$w_{\{1,  \ldots, n\}}=\root$.  We  also record  the
  lengths of all the branches $\lb w_A, v_A\rb$:
\begin{equation}
  \label{eq:bL=lA}
\bL_n(T, \bv)=\bigl(\ell_A(T, \bv), A\in
\cp^+_n\bigr)
\quad\text{with}\quad \ell_A(T, \bv)=d(w_A, v_A).
\end{equation}

\begin{figure}[ht]
  \includegraphics[scale=0.5]{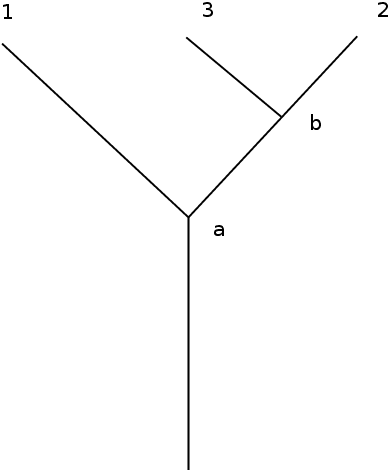}
  \caption{A discrete trees spanned by the leaves $\{1, 2, 3\}$.}
    \label{fig:123}
\end{figure}

\begin{table}[ht]
\caption{Quantities of interest for the discrete tree from
  Figure~\ref{fig:123}.}
\label{tab:123}
   \begin{tabular}
   {|c||c|c|c|c|c|c|c|}
\hline
$A\subset\cp^+_3$ & $\{1\}$ & $\{2\}$ & $\{3\}$ & $\{1,2\}$ & $\{1,3\}$
     & $\{2,3\}$ & $\{1,2,3\}$ \\
\hline
\hline
$v_A$ & 1& 2& 3& $a$ & $a$ &$b$ &$a$\\
\hline
$w_A$ &$a$& $b$& $b$& $a$ & $a$ &$a$ &$\root$\\
\hline
 $\ell_A$ & $d(a,1)$ & $d(b,2)$ & $d(b,3)$ & 0 & 0 & $d(a,b)$ & $d(\root,a)$\\
\hline
\end{tabular}
\end{table}

For instance, we record the quantity of interest in
Table~\ref{tab:123} for the discrete tree spanned by the leaves $\{1,2,3\}$
from Figure~\ref{fig:123}.
We can see that each branch of the discrete tree appears (through their
length) once and only
once in  $\bL_3(T,\bv)$.
\medskip

Set
$\hat \bv =\bigl(\hat \bv_0=\root, ( v_A, A\in \cp^+_n)\bigr) \in T^{2^{n}}$, so
that $(T, \hat  \bv)$ is a locally  compact rooted $(2^{n}-1)$-pointed
tree with  the same root  $\root$ as $T$. Notice  that all the  vertices in
$\bv$ appear in $\hat \bv$ (possibly more than once), and that $w_A$
also appears in $\hat \bv$ for all $A\in \cp_n^+$.  Recall the set of
discrete trees defined at the end of Section \ref{sec:root-Tn}. The next
lemma states  that $\bL_n$  encodes discrete  trees continuously.  Set
$\Im (\bL_n)\subset \R_+^{\cp^+_n}$ (with $\R_+^{\cp^+_n}=\R_+^{2^n-1}$) for the image of $\bL_n$.

\begin{lem}[Regularity of the branch lengths as a function of the  tree]
  \label{lem:bL-cont}  Let  $n\in \N^*$. The  map $(T, \bv)  \mapsto (T, \hat  \bv)$ is
  well defined from $\TLn$  to $\T^{(2^{n}-1)}_{\mathrm{loc-K}}$, and it
  is continuous.  The function  $\bL_n$ is well  defined from  $\TLn$ to
  $\Im  (\bL_n)\subset\R_+^{\cp^+_n}$  and is  continuous.  Furthermore,
  $\Im (\bL_n)$ is closed and  $\bL_n$ is a one-to-one bi-measurable map
  from $\TDn$ to $\Im (\bL_n)$. 
\end{lem}

\begin{proof}
If $(T, \bv)$ and $(T', \bv')$ belong to  the same equivalence class in
$\TLn$, then we deduce from \eqref{eq:def-vA} and \eqref{eq:def-wA} that
$(T, \hat \bv) $ and $(T', \hat \bv')$ belong also to the same equivalence
class. This implies that the function $(T, \bv) \mapsto (T,\hat \bv) $
is well defined from $\TLn$ to $\T_{\mathrm{loc-K}}^{(2^{n}-1)}$.
We deduce from  \eqref{eq:def-vA} and
\eqref{eq:def-wA} that this function is in fact continuous on $\TLn$. We
also get that the function $\bL_n$ is well defined from  $\TLn$ to
$\R_+^{\cp^+_n}$.
\medskip

We shall now precise the image of the function $\bL_n$ and prove its
continuity.
 Recall $x_+=\max(x,
0)$ denotes the positive part of $x\in \R$.
We define the
function $L$ from $\R_+^{(n+1)\times(n+1)}$ to $\R_+^{\cp^+_n}$
by, for $\rd=(\rd_{ij}, 0\leq i,j\leq n)$ and $A\in \cp^+_n$:
\[
  L_A(\rd)=
\inv{4} \inf \left\{ \left( \rd_{ii'} + \rd_{ij'}+ \rd_{ji'}+d_{jj'} - 2
    \rd_{ij} - 2\rd_{i'j'}\right)_+\,\colon\, i,j\in A \text{ and } i',j'\in
  A^c\right\},
\]
where $A^c=\{0, \ldots, n\}\setminus A$.
We also define the
function $D$ from  $\R_+^{\cp^+_n}$ to $\R_+^{(n+1)\times(n+1)}$
by, for $\ell=(\ell_A, A\in \cp^+_n)$ and $i,j\in \{0, \ldots, n\}$:
\begin{equation}
   \label{eq:def-Dd}
  D_{ij}(\ell)=\sum_{A\in \cp^+_n} \ell_A \left(
\ind_{\{i\in A, j\not \in
    A\}} + \ind_{\{i\not\in A, j \in
    A\}}
  \right).
\end{equation}
The functions $L$ and $D$ are  continuous.   Consider the closed subset
$\cq^{(n)}$ of $\R_+^{(n+1)\times(n+1)}$ satisfying the so-called four-point
condition, that is the set of all $(\rd_{ij}, 0\leq i,j\leq n)\in
\R_+^{(n+1)\times(n+1)}$ such that:
\[
    \rd_{ij}+\rd_{i'j'}\leq  \max(\rd_{ii'}+ \rd_{jj'}, \rd_{ij'}+ \rd_{ji'})
    \quad\text{for all}\quad
    i,j, i', j'\in \{0, \ldots, n\}.
  \]
  Notice that the four-point condition is also used to characterize metric spaces which are real trees, see \cite{ev08}.
  Then, one can check  that the  function $L$ is one-to-one
  from $\cq^{(n)}$ to $L(\cq^{(n)})$ with inverse $D$.  We also get that
  $L(\cq^{(n)})$ is closed (indeed if $(\ell ^k=L(\rd^k), k\in \N)$ is a
  sequence  of elements  of $L(\cq^{(n)})$  converging to  a limit,  say
  $\ell$, then it is bounded and thus  the sequence $(\rd ^k, k\in \N)$ is
  also bounded. Hence there is  a converging sub-sequence, and denote by
  $\rd$ its  limit which belongs to  $\cq^{(n)}$ as this set  is  closed. Since
  $L$ is continuous,  we get that $L(\rd)=\ell$ and  thus $\ell$ belongs
  to $L(\cq^{(n)})$, which gives  that $L(\cq^{(n)})$ is closed).  Since
  for       $(T,      \bv)\in       \TLn$,       we      have       that
  $\bL_n(T, \bv)=L\bigl(d(v_i, v_j),  0\leq i, j\leq n\bigr)$, we  deduce that the
  function  $\bL_n$   is  continuous  from  $\TLn$   to  $L(\cq^{(n)})$.
  \medskip

  We  now prove  that  $\Im(\bL_n)= L(\cq^{(n)})$  and  that $\bL_n$  is
  one-to-one      from      $\TDn$     to      $L(\cq^{(n)})$.       Let
  $\ell=(\ell_A, A\in \cp^+_n) \in L(  \cq^{(n)})$. Thus, there exists a
  sequence   $\rd=(\rd_{ij},  0\leq   i,j\leq  n)\in   \cq^{(n)}$  which
  satisfies the four-point condition  and such that $L(\rd)=\ell$. Since
  $\rd$ satisfies the  four-point condition, we get that  there exists a
  discrete tree $(T,  d, \bv)\in \TDn$ such  that $d(v_i, v_j)=\rd_{ij}$
  for   all   $i,   j\in   \{0,  \ldots,   n\}$.    This   proves   that
  $\Im(\bL_n)=  L(\cq^{(n)})$.  Then  use  that $L$  is one-to-one  from
  $\cq^{(n)}$ to $L(\cq^{(n)})$  with inverse $D$ and  that two discrete
  trees $(T,  d,\bv)$ and $(T',  d', \bv')$ are  equal in $\TDn$  if and
  only if $d(v_i, v_j)=d'(v'_i,v'_j)$ for all $i, j\in \{0, \ldots, n\}$
  to deduce that $\bL_n$ is one-to-one from $\TDn$ to $L(\cq^{(n)})$ and
  thus bi-measurable thanks to Lusin's theorem.
\end{proof}

\subsection{The splitting operator for a pointed tree}
\label{sec:split}

We want now to decompose the pointed tree $(T,\bv)$ along the branches of $\Span^\circ(T,\bv)$.
We keep notations from Section~\ref{sec:def-vA-wA}.

Let  $(T,  \bv)$, with  $\bv=(v_0=\root,  \ldots,  v_n)$, be  a  locally
compact rooted $n$-pointed tree.  Recall Definition \eqref{eq:def-pv} of
the projection $p_\bv$ on  $\Span(T,\bv)$.  For $A\in \cp^+_n$, consider
the rooted $1$-pointed  tree:
\begin{equation}
  \label{eq:def-hat-T_A}
  \hat T_A(T,\bv)=\bigl(T_A(T, \bv), (\root_A, v_A)\bigr)\in \TLu,
\end{equation}
with root $\rho_A=w_A$ and
\[
T_A(T, \bv)=\bigl\{x\in T\, \colon \, p_\bv(x)\in \rb w_A, v_A \rb\bigr\}\cup\{w_A\}.
\]
By construction, we have that $ \ell_A(T, \bv)=d(\root_A, v_A)$.
\medskip

Notice that $\ell_A(T, \bv)=0$ if and only if
$\hat T_A(T, \bv)$ is reduced to its root, that is, $\bigl(\{\root_A\}, (\root_A,
\root_A)\bigr)$. Notice also that $\ell_A(T, \bv)>0$ implies that $\hat T_A$
belongs to $\TLunnb$, the set of trees in $\TLu$ such that the root is
not a branching point (see Definition~\eqref{eq:del-TL-br}).
We also define the rooted $1$-pointed tree $\hat T_{\{0\}}(T, \bv)\in \TLu=\left(T_{\{0\}}(T, \bv), (\root, \root)\right)$ by:
\[
T_{\{0\}}(T, \bv)=\bigl\{x \in T\, \colon \, \rb \root, x\rb \cap \Span^\circ(T,
  \bv)=\emptyset\bigr\},
\]
with root  $\root$ and distinguished vertex also $\root$.
If $(T, \bv)$ and $(T', \bv')$ belong to  the same equivalence class in
$\TLn$, then we get that $\hat T_A(T, \bv)$ and $\hat T_A(T', \bv')$
belong also to the same equivalent class in $\TLu$ for $A\in \cp_n=\cp_n^+\cup\{\{0\}\}$. Thus,  the map
$\Split_n$ defined on $\TLn$ by:
\begin{equation}
   \label{eq:def-split}
 \Split_n(T, \bv)=\left(
  \hat T_A(T, \bv), A\in \cp_n\right)
\end{equation}
takes values in $ \left(
  \TLu\right)^{2^n}$. We give an instance of the function $\Split_n$ in
Figure~\ref{fig:split}.

  \begin{figure}[ht]
  \includegraphics[scale=0.5]{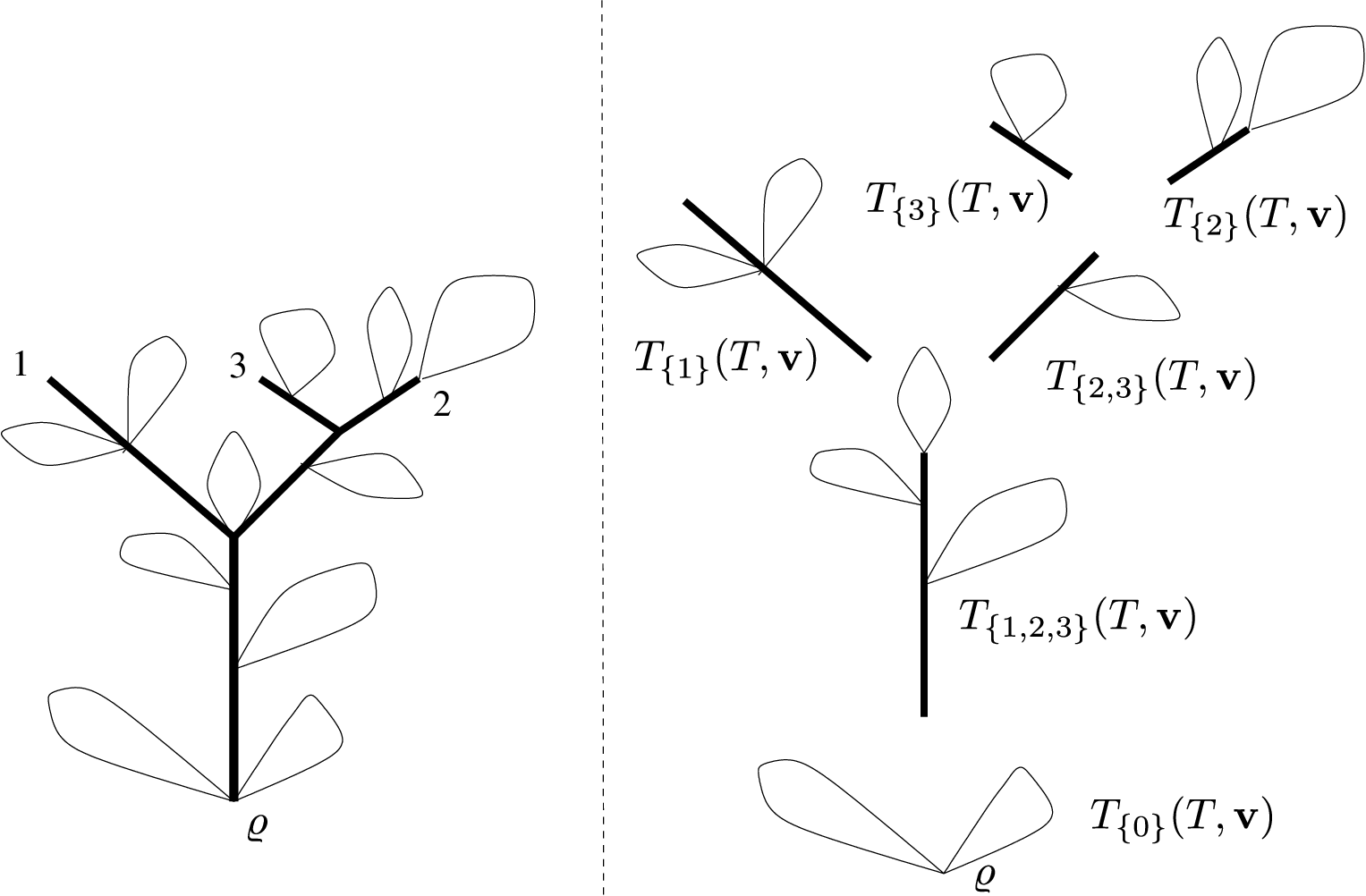}
  \caption{The splitting of the left hand tree with respect to
    $\bv=\{\root,1,2,3\}$. In this instance, $T_{\{1,2\}}$ and $T_{\{1,3\}}$ are reduced to their own root.}
  \label{fig:split}
  \end{figure}

\begin{lem}[Measurability of the splitting map]
  \label{lem:cnt-cck}
Let $n\in \N^*$.   The map $ \Split_n$ from $\TLn$ to $\left(\TLu\right)^{2^n}$ is
measurable.
\end{lem}

\begin{proof}
The proof is divided into three steps.

\emph{Step   1}:  The   map   $\hat  T_{\{0\}}$   is  measurable.    Let
$(T,    \bv)\in    \TLn$.     By     construction,    we    have    that
$ r_0^{[2],  +}\bigl(T, \Span^\circ  (T, \bv)\bigr)=\bigl(T_{\{0\}}(T, \bv),  \Tz\bigr)$. We
deduce   from    Lemma   \ref{lem:rt2}    on   the    measurability   of
$r_t^{[2],          \varepsilon}$,         that          the         map
$(T, \bv)\mapsto \hat T_{\{0\}}=\bigl(T_{\{0\}}(T,  \bv), (\root, \root)\bigr)$ is
measurable.  \medskip

   \emph{Step 2}: A measurable truncation  function. Let $n\geq 1$.  Let
   $(T, \bv)$  be a rooted  $n$-pointed tree.  Recall the  definition of
   $\hat   T_A(T,    \bv)$   from   \eqref{eq:def-hat-T_A}.     We   set
   $q(T,  \bv)=\hat T_{\{1,2,\ldots,n\}}(T,\bv)$  so that  $q$ is  a map
   from $\TLn$ to $\TLu$. Recall the measurable truncation functions $r_t^{[2],+}$
   and $r_*^{[2]}$ from \eqref{eq:def-tilde-rt2} and \eqref{eq:r*[2]}, respectively. 
   
We set:
   \[
     q'(T, \bv)= r^{[2]}_* \circ r^{[2], +}_{d(\root, w_{\{1, \ldots,
         b\}})}\,  \bigl(T, \Span^\circ(T, \bv)\bigr).
   \]
   Thanks      to       Lemma~\ref{lem:cont-span-loc},      the      map
   $(T, \bv) \mapsto \bigl(T, \Span^\circ(T, \bv)\bigr)$ is continuous from $\TLn$
   to     $\TLd$.     Thanks      to     Lemma~\ref{lem:bL-cont}     and
   Remark~\ref{rem:cont-T-d},      we     get      that     the      map
   $(T,\bv) \mapsto  d(\root, w_{\{1, \ldots, b\}})$  is continuous from
   $\TLn$ to $\R_+$. Then, use Lemmas~\ref{lem:rt2} and~\ref{lem:rt*} on
   the  measurability of  $r_t^{[2],  \varepsilon}$  and $r_*^{[2]}$  to
   conclude that the map $q'$ from $\TLn$ to $\TLd$ is measurable and it
   has        the       same        image        as       the        map
   $\bigl(T,  (\root,  v)\bigr)\mapsto  \bigl(  T,  \lb\root,  v\rb\bigr)$  from  $\TLu$  to
   $\TLd$. According  to Lemma~\ref{lem:cont-span-loc} (with $n=1$),  this latter  map is
   injective and measurable. Hence the map $q$, which is the composition of
   $q'$ and this latter map, is measurable.  \medskip

\emph{Step  3}:   Conclusion.   Let  $A\subset  \{1,   \ldots,  n\}$  be
non-empty. Notice  that $\hat  T_A$ is  the image of  $(T, \bv)$  by: the
expansion procedure $(T, \bv)  \mapsto (T, \hat  \bv)$  from the first  part of Lemma \ref{lem:bL-cont}, the
rerooting  at $w_A$ from Lemma~\ref{lem:rerooting}, the
reducing procedure from Lemma~\ref{lem:erase-cont} where
one forgets about all $w_{A'}$ and $v_{A'}$ for $A'\subset A^c$, and then the
function   $q$  from   Step   2.   This   implies   that  the   function
$(T, \bv) \mapsto \hat T_A(T, \bv)$ is measurable from $\TLn$ to $\TLu$.
\end{proof}

\subsection{The grafting procedure}
\label{sec:graft}
Let $n\in \N^*$.
Let $\ell=(\ell_A,  A\in \cp^+_n) \in  \Im (\bL_n)$. According  to Lemma
\ref{lem:bL-cont},  there exists  a  unique (up to the
equivalence in $\TKn$) rooted $n$-pointed  discrete
tree  $(S, \bv)$ (that is    $S=\Span^\circ(S,    \bv)$)    such    that
$\bL_n(S,   \bv)=\ell$.
Recall     $v_A$      and      $w_A$     defined      in
Section~\ref{sec:def-vA-wA} for $A\in \cp^+_n$ so that:
\begin{equation}
   \label{eq:def-S-SA}
  S=\bigcup_{A\in \cp^+_n}  \lb w_A, v_A\rb,
\end{equation}
where the sets $(\rb w_A, v_A\lb, A\in \cp^+_n)$ are pairwise disjoint.

Recall that $\TLsb$  denotes the set (of equivalence  classes) of locally
compact  rooted trees  with  one  infinite marked  branch  such that
the root is not a branching vertex.
Let $T^*=(T_A ^* , A\in \cp^+_n)$ be a family of elements of equivalence
classes     in     $\TLsb$.      Then,     we     define     the     tree
$(T,\bv)=\graft_n (\ell,  T^*)$,  where $T$  is the
tree $S$ with that the branches $\rb  w_A, v_A\rb$ are replaced by the trees
given by the first component of $r^{[2],+}_{\ell_A} ( T_A^*)$ (where the
second component has been identified to $\lb w_A, v_A\rb$).
\medskip

We     now     provide     a     more     formal     construction     of
$\graft_n (\ell, T^*)$. Let  $\ell\in \Im (\bL_n)$,
and     consider     the     rooted    $n$-pointed     discrete     tree
$(S, \bv)=\bL_n^{-1} (\ell)\in \TDn$ and $\bv=(v_0=\root, \ldots, v_n)$.
Set
$\hat \bv  =(\hat \bv_0=\root, (  v_A, A\in \cp^+_n))  \in T^{2^{n}}$,
with $v_A$ the most recent common ancestor of $(v_i, i\in A)$ defined in
\eqref{eq:def-vA}.   Thus, we  get that  $(S, \hat  \bv)\in \TDnn$  is a
rooted $(2^{n}-1)$-pointed discrete  tree with the same  root $\root$ as
$S$.
\medskip

In a  first step,  we build  by a  backward induction  an ``increasing''
sequence                of                 discrete                trees
$\bigl((S_k,  \bv_k),   k\in  \{0,  \ldots,  2^n-1\}\bigr)$   such  that
$(S_k,    \bv_k)\in     \TDk$    with    root    $\root$.      We    set
$(S_{2^n -1}, \bv_{2^n -1})=(S, \hat \bv)$. Recall that $x$ is a leaf of
a tree $T$ with root $\root$  if $x\in \lb \root, y\rb\subset T$ implies
$y=x$.   Assume   that  $(S_{k+1},  \bv_{k+1})$  is   defined  for  some
$k\geq 0$.  We consider the  lexicographical order on the non-empty sets
of $\N$ defined recursively as follow:  for $A, B \subset \N$ non empty,
we write $A< B$: if $\min A < \min B$; or if $\min A =\min B$ and $A$ is
a singleton  but not $B$; or  if $\min A =\min  B$, $A$ and $B$  are not
singletons and $A'< B'$ where  $A'=A \setminus \{\min A\}$ and similarly
for $B'$. Notice this order is total. We set:
\[
  A_{k+1}=\max   \bigl\{A\in\cp_n^+,v_{A}\in   \bv_{k+1}\   \mbox{and}\   v_A\
  \mbox{is a  leaf of}\  (S_{k+1}, \bv_{k+1})\bigr\}.
\]
Then, we define $\bv_k$ as  the sequence $\bv_{k+1}$ where $v_{A_{k+1}}$
has been removed (notice that the  first element of $\bv_k$ is still the
root $\root$),  and we set  $(S_k, \bv_k)=\Span(S, \bv_k)\in  \TDk$.  We
also  set   $B_k=\max\{B\in  \cp_n\,  \colon\,   v_B=w_{A_{k+1}}\}$.  By
construction, $v_{B_k}=w_{A_{k+1}}$ belongs to  the sequence $\bv_k$ and
is therefore  an element  of $\bv$  for some index,  and, with  a slight
abuse of notation, we simply denote this index by $B_k$.  We have, using
the grafting operation from Section~\ref{sec:graft1} that:
\begin{equation}
  \label{eq:Sk-k+1}
  (S_{k+1}, \bv_{k+1})=(S_k, \bv_k)\circledast_{B_k} [0, \ell_{A_{k+1}}],
\end{equation}
where the equality  holds in $\TLkk$ (and in $\TDkk$)  and by convention
$[0, t]$  denotes the discrete  $1$-pointed tree $\bigl([0, t],  (0,t)\bigr)$ with
root   $0$.    Notice   that   $\ell_{A_{k+1}}=0$   if   and   only   if
$ \Span^\circ(S, \bv_k)=\Span^\circ(S,  \bv_{k+1})$.  Eventually, notice
that $(S_0, \bv_0)=\bigl(\{\root\}, \root\bigr)$ is the rooted tree reduced to its
root  $\root=v_{\{0\}}$   and  $B_0=\{0\}$.  Let  us   stress,  that  in
Section~\ref{sec:graft1}, the  vector $\bv_{k+1}$  is obtained  by adding
the distinguished  vertex $\ell_{A_{k+1}}$ of $[0,  \ell_{A_{k+1}}]$ to $\bv_k$.
However     here     we      identify     $[0,     \ell_{A_{k+1}}]$     with
$\lb  v_{B_k}=w_{A_{k+1}},  v_{A_{k+1}}\rb$  and add  the  distinguished  vertex
$v_{A_{k+1}}$ to $\bv_k$ in order to obtain $\bv_{k+1}$.

For instance, we give in Table \ref{tab:deconstruction} the sequences $(A_k,1\le k\le 2^n-1)$ and $(B_k,0\le k\le 2^n-2)$ for the tree of Figure \ref{fig:123}.

\begin{table}[ht]
\caption{The sequences $(A_{k+1},0\leq k\leq  6)$,  $(B_k,0\leq k\leq
  6)$ and $(\ell_{A_{k+1}},0\leq k\leq  6)$ for the tree of Figure \ref{fig:123}.}
\label{tab:deconstruction}
   \begin{tabular}
   {|c||c|c|c|c|c|c|c|c|}
\hline
$k$ & $0$ & $1$ & $2$ & $3$ & $4$ & $5$ & $6$ \\
\hline
$A_{k+1}$ & $\{1,2\}$ & $\{1,2,3\}$ & $\{1,3\}$ & $\{1\}$ & $\{2,3\}$ & $\{2\}$ & $\{3\}$\\
\hline
$B_{k}$ &$\{0\}$& $\{1,3\}$& $\{1,2\}$& $\{1,2,3\}$ & $\{1,2,3\}$ & $\{2,3\}$ &$\{2,3\}$ \\
\hline
     $\ell_{A_{k+1}}$ &$d(\root, a)$& $0$& $0$& $d(1, a)$ & $d(a,b)$
                                  & $d(2, b)$ &$d(3, b)$ \\
\hline
\end{tabular}
\end{table}

\begin{rem}
   \label{rem:Ak-Bk}
   The  family  $\bigl\{A_k,k\in  \{1,  2^n  -1\}\bigr\} $  is  exactly  equal  to
   $\cp_n^+$. Furthermore  the sequence  $\ell\in \Im (\bL_n)\subset \R_+^{2^n-1}$ provides  implicitly two
   unique ordered sequences $\ca(\ell)=\bigl(A_k,k\in \{1, 2^n -1\}\bigr)$ (of all
   elements of  $\cp_n^+$) and $\cb(\ell)=(B_k,k\in \{0,  2^n -2\})$ (of
   elements   of  $   \cp^n=  \cp_n^+   \cup\{\{0\}\}$),  and   an
   ``increasing'' way to built  $\bL_n^{-1}(\ell)$ recursively by adding
   at step $k\in \{0, 2^n -2\}$ a branch of length $\ell_{A_{k+1}}$ (and
   graft it on $v_{B_k}$ chosen among  $\bv_k$).  It is obvious from the
   construction  that  if  $\ell$  and  $\ell'$  are  two  sequences  in
   $\Im (\bL_n)$ with the same zeros (that is, $\ell_A=0$ if and only if
   $\ell'_A=0$),    then    we     have    $\ca(\ell)=\ca(\ell')$    and
   $\cb(\ell)=\cb(\ell')$.  Thus, the  sets $\ca(\ell)$  and $\cb(\ell)$
   are implicitly coded by the zeros of $\ell$.
\end{rem}
\medskip

In   a   second   step,   given   $\ca(\ell)$   and   $\cb(\ell)$   from
Remark~\ref{rem:Ak-Bk} and  a sequence $T^*=(T^*_A, A\in \cp_+^n)$ in $\TLsb$,   we   build   by   a   forward   induction   an
``increasing''    sequence    of    marked   locally    compact    trees
$\bigl((T_k, \bv_k),  k\in \{0,  \ldots, 2^n-1\}\bigr)$  such that  $(T_k, \bv_k)$
belongs to  $\TLk$, has root  $\root$, and  the components of  the vector
$\bv_k$ can  be ranked as the root $\root=v_{\{0\}}$ and $(v_{A_i},
1\leq i\leq k)$.  Recall  also the
truncation     function    $\tilde     r^{[2],     +}_t$    given     in
\eqref{eq:def-tilde-rt2}.  We set  $(\Tz, \bv_0)=\bigl(\{\root\}, \root\bigr)$ and
for $k\in \{0, 2^n -2\}$:
 \begin{equation}
  \label{eq:Tk-k+1}
  (T_{k+1}, \bv_{k+1})=(T_k, \bv_k)\circledast_{B_k} \tilde r^{[2],
    +}_{\ell_{A_{k+1}}} (T^*_{A_{k+1}}),
\end{equation}
where the
distinguished vertex of
$\tilde r^{[2], +}_{\ell_{A_{k+1}}} (T^*_{A_{k+1}})$ is  identified with
$v_{A_{k+1}}$ (and its root with $v_{B_k}$). Then, we set:
\begin{equation}
  \label{eq:def-graftn}
  \graft_n (\ell, T^*)=(T_{2^n -1}, \bv)
\quad\text{with}
\quad \bv=(v_{\{k\}}, 0\leq  k\leq  n).
\end{equation}

\begin{figure}[ht]
\includegraphics[scale=0.4]{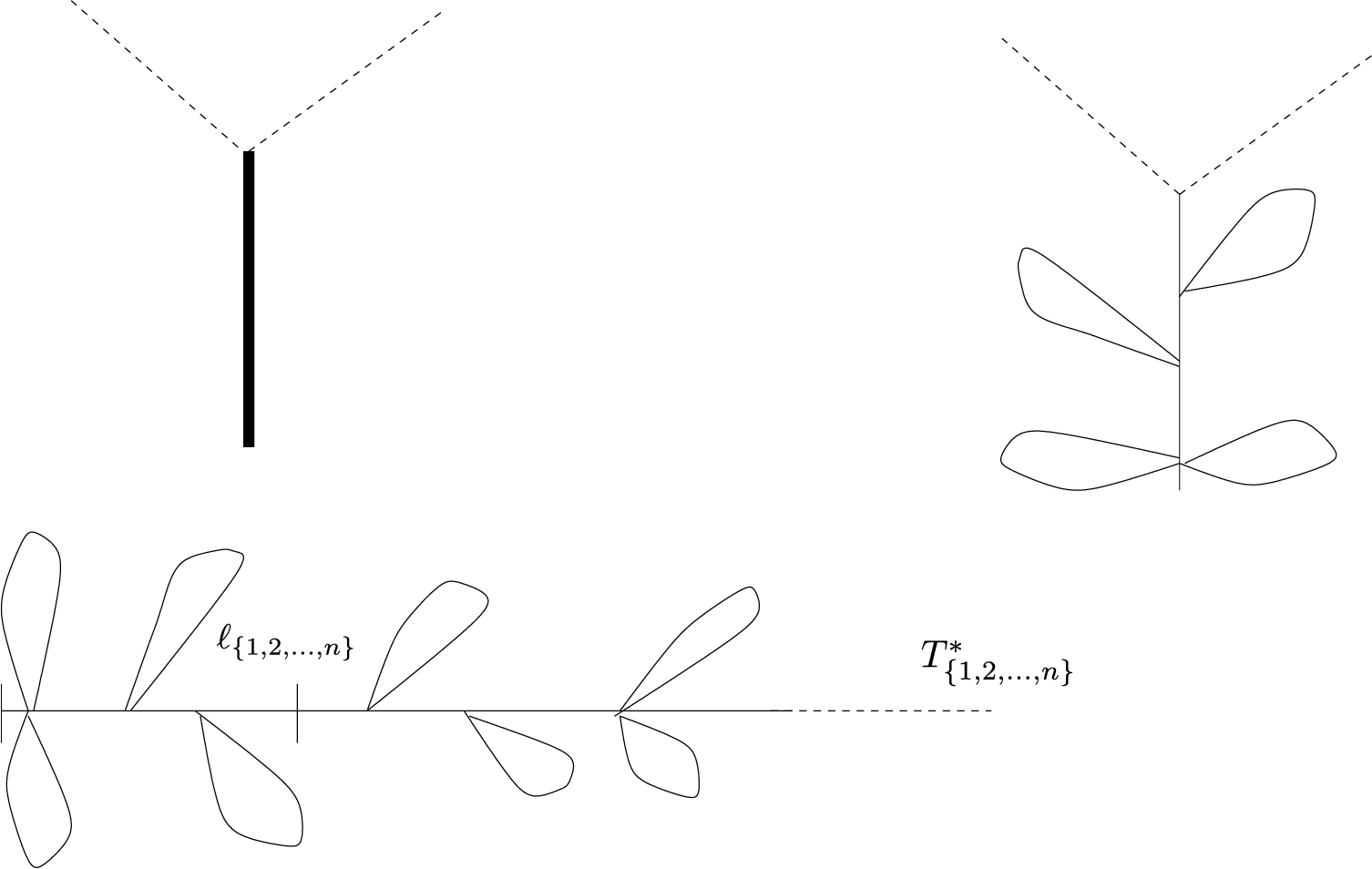}
\caption{Example of a replacement of the branch $\rb w_{\{1,\ldots,n\}},v_{\{1,\ldots,n\}}\rb$. \protect\\\hspace{\textwidth}
Upper left: The tree $S$ with the branch $\rb
w_{\{1,\ldots,n\}},v_{\{1,\ldots,n\}}\rb$ in bold.
\protect\\\hspace{\textwidth}
Upper right: The  branch $\rb w_{\{1,\ldots,n\}},v_{\{1,\ldots,n\}}\rb$
replaced by the first component of the marked tree
$r_{\ell_{\{1,\ldots,n\}}}^{[2],+}(T^*_{\{1,\ldots,n\}})$.
\protect\\\hspace{\textwidth}
Lower: The tree $T^*_{\{1,\ldots,n\}}$ with its marked infinite branch.
}
\label{fig:graft-dis}
\end{figure}

It is easy
to       check       that       the      equivalence       class       of
$(T_{2^n -1},\bv)$  in $\TLn$
does  not  depend  on  the  choice   of
$T^*=(T_A ^*   , A\in \cp^+_n)$ in their own equivalence class. Thus, the
map $\graft_n$ defined by:
\[
   \left(\ell, T^*\right)  \mapsto \graft_n  (\ell,  T^*)
\]
is well  defined from $ \Im(\bL_n) \times \left(\TLsb\right)^{\cp^+_n}$  to $\TLn$.
   The  main result  of  this  section is  the
measurability of the map $\graft_n$.

\begin{lem}[Measurability of the grafting  map]
  \label{lem:graftn-meas}
  Let     $n\in     \N^*$.      The     map     $\graft_n$     from
  $   \Im(\bL_n)\times   \left(\TLsb\right)^{\cp^+_n}$   to   $\TLn$   is
  measurable.
\end{lem}

\begin{proof}
  For           $J\subset          \cp_+^n$,           we          write
  $I_J=\bigl\{\ell \in \Im (\bL_n)\, \colon\,  \ell_A=0 \text{ if and only if
  } A\in  J\bigr\}$. Thus, the  closed set $\Im (\bL_n)$  of $\R_+^{\cp_+^n}$
  can be  written as  the union  of $I_J$  over all  the subsets  $J$ of
  $\cp_+^n$. Furthermore,  the sets  $(I_J, J\subset\cp_+^n) $  are Borel
  sets (as $\Im (\bL_n)$ is a Borel set),
  and they are pairwise disjoint. Thanks to Remark~\ref{rem:Ak-Bk},
  the  maps  $\ell\mapsto  \ca(\ell)$ and  $\ell\mapsto  \cb(\ell)$  are
  constant over $I_J$. We deduce from Equation~\eqref{eq:def-graftn} and
  recursion~\eqref{eq:Tk-k+1},    Lemma~\ref{lem:graft-meas}   on    the
  continuity of  the grafting procedure and  Lemma~\ref{lem:rt2*} on the
  measurability of $(t,  T) \mapsto \tilde r^{[2],  +}_{t} (T)$ that
  the function $\graft_n$  from $ I_J\times \left(\TLsb\right)^{\cp^+_n}$
  to $\TLn$ is measurable (as long  as $I_J$ is not empty).  Since there
  is a  finite number  of such  sets $I_J$, we  deduce that  the function
  $\graft_n$  from  $ \Im(\bL_n)\times  \left(\TLsb\right)^{\cp^+_n}$  to
  $\TLn$ is measurable.
\end{proof}

\begin{rem}
  \label{rem:graft-Tdis}
  Since the map $\bL_n$ is continuous one-to-one from $\TDn$ to $\Im
  (\bL_n)$, we deduce that the map:
  \[
   \left(T, T^*\right)  \mapsto \graft_n
   \bigl(\bL_n(T),  T^*\bigr)
 \]
 from $\TDn\times  \left(\TLsb\right)^{\cp^+_n}$  to
  $\TLn$ is measurable. Without ambiguity, we shall simply write $\graft_n
   (T,  T^*)$ for $\graft_n
   \bigl(\bL_n(T),  T^*\bigr)$.
\end{rem}

\begin{rem}\label{rem:graft-split=id}
  Intuitively, the maps $\graft_n$ and  $\Split_n$ should be the inverse
  one of  the other. More precisely,  we have the following  result. For
  every     $(T,(\root,v))\in    \TLu$,     we    define     the    tree
  $\mathrm{Sp}(T)=(T',            S')\in            \TLsb$            by
  $T'=\Pi^\circ_1\left(T\circledast_1 [0,\infty)\right)$ with the marked
  spine
  $S=\Pi^\circ_1\left(\lb
    \root,v\rb\circledast_1[0,\infty)\right)$.  Then then  we have,  for
  every $(T,\bv)\in\TLnb$ (that  is, the root of $T$ is  not a branching
  vertex,   see  Definition~\eqref{eq:del-TL-br}),   that  the   following
  equality hold in $\TLn$:
\begin{equation}
   \label{eq:graft-split=id}
\graft_n\Bigl(\Span_n(T, \bv),
\mathrm{Sp}\bigl(\Split_n(T,\bv)\bigr)\Bigr)=(T,\bv),
\end{equation}
where $\mathrm{Sp}(T_A, A\in \cp_n)=(\mathrm{Sp}(T_A), A\in \cp_n^+)$.
\end{rem}

\subsection{A measure associated with trees in $\TLsb$ or $\TLu$}
\label{sec:def-M}

Recall  $\Tz=(\{\root\},  \root)\in \TL$  is  the  tree reduced  to  its
root. We define
\begin{align}\label{eq:T*loc-K}
  \TLz=\TL\setminus\{\Tz\}
\end{align}
endowed with the distance:
\[
 \dlghz(T, T')=\dlgh(T, T')+ \val{H(T)^{-1} - H(T')^{-1}}.
\]
Clearly $(\TLz,  \dlghz)$ is  Polish with the topology induced by the
topology on $\TL$ (as $H$ is continuous on $\TL$), and  for all  $\varepsilon>0$, the
sets $B_\TLz(\varepsilon)=\{T\in \TLz\, \colon\, H(T)\geq \varepsilon\}$
are closed  and bounded. Furthermore,  every bounded  set is a  subset of
$B_\TLz(\varepsilon)$     for    $\varepsilon>0$     small    enough.      Set
$E=\R_+\times      \TLz$       endowed      with       the      distance
$d_E((u,T), (u',  T'))=|u-u'|+ \dlghz(T, T')$,  so that $(E, d_E)$  is a
Polish   space.    Every  bounded   set   of   $E$   is  a   subset   of
$B_E(\varepsilon)=[0,  \varepsilon^{-1}]\times B_\TLz(\varepsilon)$  for
$\varepsilon>0$  small enough.   We  define $\M(E)$,  the  set of  point
measures on  $E$ which  are bounded  on bounded sets,  that is  finite on
$B_E(\varepsilon)$  for   all  $\varepsilon>0$.    We  say  that a  sequence
$(\cm_n, n\in \N)$ of elements of $\M(E)$ converges to a limit $\cm$, if
$\lim_{n\rightarrow  \infty   }  \cm_n(f)=\cm(f)$  for   all  continuous
functions    on    $E$    with    bounded    support.     According    to
\cite[Proposition~9.1.IV]{dv2} the space $\M(E)$ is Polish and the Borel
$\sigma$-field is the smallest  $\sigma$-field such that the application
$\cm\mapsto  \cm(A)$  is  measurable  for  every Borel  set  $A$  of  $E$.
\medskip

We build a tree from  a point
measure $\cm=\sum_{i\in I}\delta _{(h_i, T_i)}\in \M(E)$ by grafting $T_i$ at
height $h_i$ on an infinite spine.
Recall the infinite spine $\Tu=(\R_+, 0)$ endowed with the Euclidean distance is an
element of $\TLsb\subset \TL$.
For $T\in \TL$, let  $(\tilde T, d,  \root)$ denote a rooted locally compact tree in the equivalent
class $T$.
With obvious notation, we define the tree $T'$ as follow:
\begin{gather*}
T'=\Tut\sqcup_{i\in I} (\tilde T_i\setminus\{\root_i\}),\\
\forall x,x'\in T',\ d(x,x')=
\begin{cases}
d_i(x,x') & \text{if } x,x'\in \tilde T_i, i\in I\\
|x-x'| & \text{if }x,x'\in  \Tut,\\
d_i(x,\root_i)+ |h_i - x| & \text{if } x\in \tilde
T_i,\ x'\in \Tut, i\in I,\\
d_i(x,\root_i)+d_j(x',\root_j)+ |h_i - h_j| & \text{if } x\in \tilde
T_i,\ x'\in \tilde T_j \text{ with } i\neq j, \, i,j\in I,\\
\end{cases}
\end{gather*}
where $\sqcup$ denotes  the disjoint union. By construction
$T'$ is  a tree rooted at $\root=\root_1$, the root of $\Tut$. Because $\cm$ is finite on bounded sets
of $E$, it is not difficult  to check that $T'$ is locally compact.
  It is  easy  to  see  that  the
equivalence class  of $\tree(\cm)=(T' , \Tut)$ in  $\TLd$ does not  depend of
the choice of the representatives in  the equivalence classes of $\Tu$ and
$T_i$ for $i\in I$. Hence, identifying $\tree(\cm)$ with its equivalence
class, we get that the             map
$\tree$     is     well    defined     from
$\M(E)$  into $\TLd$.

\begin{lem}[Regularity of the map $\tree$]
  \label{lem:tree-meas}
The map $\tree$ from $\M(E)$ to $\TLd$  (or $\TLs$) is continuous.
\end{lem}

\begin{proof}
  We   only  give   the   principal  arguments   of   the  proof.    Let
  $(\cm_n, n\in \N)$ a sequence  of point measures, elements of $\M(E)$,
  which  converges to  $\cm$.  Let $\varepsilon>0$  be  fixed such  that
  $\cm(\partial  B_E(\varepsilon))=0$. For  $n$  large  enough, we  have
  $\cm_n(B_E(\varepsilon))=\cm(B_E(\varepsilon))$   and  the   atoms  of
  $\cm_n$  in  $B_E(\varepsilon)$ converge  to  the  atoms of  $\cm$  in
  $B_E(\varepsilon)$. Using  correspondence between  the representations
  of   the  atoms,   and  similar   arguments   as  in   the  proof   of
  Lemma~\ref{lem:graft-meas},  we  deduce   that  the  distance  between
  $\tree(\cm_n)$   and   $\tree(\cm)$   (in    $\TLd$)   is   small   if
  $\varepsilon>0$ is small  (to prove this statement in  detail, one can
  use       the      distance       on       $\M(E)$      given       in
  \cite[Equation~(A2.6.1)]{dv1}).         This        means         that
  $\lim_{n\rightarrow \infty }  \dlghd(\tree(\cm_n), \tree(\cm))=0$, and
  thus the map $\tree$ is continuous on $\TLd$. \medskip
\end{proof}

We shall now prove that the restriction of the map $\tree$ to a subset of
$\M(E)$ is injective and bi-measurable. For this reason, we consider the  subset of $\TL$ of (equivalence classes of) trees
not reduced to their root and such that the root is not a branching
vertex (recall Definitions~\eqref{eq:T*loc-K} and \eqref{eq:del-TL-br} with $n=0$):
\begin{equation}
   \label{eq:def-TLzb}
  \TLzb=\TLz\cap \TLb.
\end{equation}
As a direct consequence of Lemma~\ref{lem:root-no-branching}, $\TLzb$ is
a Borel subset of $\TL$ and thus of $\TLz$.
In particular,  the following subset  of $\M(E)$ is a
 Borel set (recall $E=\R_+\times      \TLz$):
\begin{equation}
   \label{eq:def-ME}
\ImM =\Bigl\{\cm\in \M(E)\, \colon\, \cm\bigl(\R_+\times (\TLzb)^c\bigr)=0\Bigr\}.
\end{equation}

We now introduce a map $\cmt$ from $\TLs$ to $\M(E)$ as follow.
Let  $T^*=(T,\Tu)$ be  a rooted  locally compact  tree with  an infinite
marked spine. In
particular, we have $\Tu\subset T$ and  $\Tu$ is equivalent to $(\R_+, d, 0)$.  Let $(T^\circ_i,i\in I)$ be the family
of the  connected components of  $T\setminus \Tu$.  For every  $i\in I$,
let us  denote by  $x_i$ the  MRCA of $T^\circ_i$,  that is,  the unique
point    of   $\Tu$    such   that    for   every    $x\in   T^\circ_i$,
$\lb\root,x\rb\cap \Tu=\lb\root,x_i\rb$.   We    then   set
$T_i=T_i^\circ\cup\{x_i\}$ viewed  as a  locally compact  tree rooted
at $x_i$. Then, we define the point measure $\cmt(T^*)$ on $\R_+\times
\TLz\subset\R_+\times
\TL $ by:
\begin{equation}
   \label{eq:def-cmt}
  \cmt(T^*)=\sum_{i\in I}\delta_{(H(x_i),T_i)}.
  \end{equation}  
As $\cmt(T^*)$ does not depends on the representatives chosen in the
equivalence class of $T^*$ in $\TLs$, we deduce that  $\cmt:  T^*\mapsto \cmt(T^*)$ is a map from $\TLs$ to $\M(E)$.
We now give the main result of this section.

\begin{prop}[Regularity of the maps $\tree$ and $\cmt$]
  \label{prop:meas(T)-mesurable}
  The  map  $\cmt$  is    bi-measurable  from $\TLs$
  to $\ImM$ with $\ImM= \mathrm{Im} (\cmt)$. The map $\tree$  is
  bi-measurable  from $\tilde \M(E)$
to $\TLs$. Furthermore, the map
  $\tree\circ \cmt$     is      the     identity     map     on
  $\TLs$  and $\cmt\circ\tree$ is the
  identity map on $\ImM$.
\end{prop}

\begin{proof}
By construction, the roots of all the
trees $T_i$ in the point measure $\cmt(T^*)$ are  not branching vertices,
so that $\cmt(T^*)$  belongs to $\tilde \M(E)\subset \M(E)$. We
also get by construction that $\tree(\cmt(T^*))=T^*$. This implies that $\cmt$ is
injective and thus bi-measurable thanks to Lusin's theorem.

We  also have  by construction  that $\cmt\circ\tree(\cm)=\cm$  for
$\cm\in       \tilde       \M(E)$.         This       implies       that
$ \mathrm{Im}  (\cmt)=\ImM$  and also  that $\tree$  restricted to
$\ImM$   is  injective   and  thus   bi-measurable  thanks   to  Lusin's
theorem.
\end{proof}

We  extend the  map $T^*\mapsto  \cmt(T^*)$  to $\TLu$  in the  following
way. For  $\bigl(T, \bv=(\root, v_1)\bigr)\in  \TLu$, we graft the  infinite spine
$\Tu$  on   $v_1$  and  consider   the  rooted  locally   compact  tree with an infinite marked spine $\mathrm{Sp}(T)\in\TLs$ defined in Remark \ref{rem:graft-split=id}. Then,  we define  $ \cmt(T,  \bv)$ as
$\cmt(\mathrm{Sp}(T))$.   From  the  continuity   of  the  grafting  procedure,  see
Lemma~\ref{lem:graft-meas}   and  the   continuity  of   $\Pi^\circ_1$,  see
Lemma~\ref{lem:erase-cont},   and   the   measurability   of   the   map
$\cmt$, we  deduce that  the map  $(T, \bv)\mapsto  \cmt(T, \bv)$,
which we  still denote by $\cmt$ is measurable. In  fact, we have
the stronger following  result.  Consider the set  of (equivalent classes
of) $n$-pointed rooted locally compact tree  such that the root is not a
branching vertex and the distinguished vertices are not equal to the root:
\begin{equation}
  \label{eq:def-T1-no-br}
  \TLnzb=\bigl\{(T, \bv)\in \TLnb\, \colon\, d(\root, v_i)>0 \text{ for all
    $i\in \{1, \ldots, n\}$} \bigr\},
\end{equation}
where  $\bv=(\root, v_1, \ldots, v_n)$.
According to Lemma~\ref{lem:root-no-branching} and Remark~\ref{rem:cont-T-d}, the set $ \TLnzb$ is a Borel subset of
$\TLn$.
Recall from~\eqref{eq:def-TLzb} that the  Borel set $\TLzb$ is the set
of (equivalence class  of) $1$-pointed rooted locally  compact trees such
that the  root is not  a branching vertex and  the distinguished vertex  is not
equal to the root.
\begin{cor}[Recovering $(T,\bv)$ from $\cmt(T, \bv)$]
  \label{cor:Span-Ma-1}
  The following map from $\TLu$ to
  $\R_+\times \M(E)$ defined  by:
\[
  (T, \bv)\mapsto \bigl( d(\root, v), \cmt(T, \bv)\bigr)
\]
is measurable and its restriction to $\TLuzb$ is injective and bi-measurable.
\end{cor}
\begin{proof}
  Set
  $\M   ^*(E)=\bigl\{\cm\in    \M(E)\,   \colon\,   \cm\bigl(\{0\}\times
  \TLz\bigr)=0\bigr\}$.  For $\cm\in \M^*(E)$,  we get that $\tree(\cm)$
  belongs   to   $\TLsb$.   Write 
  $[0, a]\in \TLu$ for the tree $[0, a]$ with root  $0$ and
  distinguished vertex  $a\geq 0$. We   define  a  map   $g$   on
  $\R_+\times                         \M^*(E)$                        by
  $g(a,  \cm)=\graft_1\bigl([0, a],  \tree(\cm)\bigr)$. Thanks  to the
  continuity of the  grafting procedure, see Lemma~\ref{lem:graftn-meas}
  and of the function  $\tree$, see Lemma~\ref{lem:tree-meas}, we deduce
  that $g$ is continuous.

  Let $(T, \bv)\in \TLuzb$. As the root  of $T$ is not a branching vertex,
  we   get  that   $\cmt(T,\bv)$  belongs   to  $\M   ^*(E)$,  and   thus
  $g\bigl(d(\root, v), \cmt(T, \bv)\bigr)$, where $\bv=(\root, v)$, is well defined
  and     in    fact     equal     to    $(T,     \bv)$    thanks
  to~\eqref{eq:graft-split=id} with $n=1$.   This   implies   that   the   map
  $ (T, \bv)\mapsto  \bigl( d(\root, v), \cmt(T, \bv)\bigr)  $ defined on
  $\TLuzb$  is  injective,  and  thus bi-measurable  by  Lusin's  theorem.
\end{proof}

We extend this result to $n$-pointed  trees.
 Recall from~\eqref{eq:def-split} that, for $(T, \bv)\in \TLn$,  we have
 $\Split_n (T, \bv)=\left(
  \hat T_A(T, \bv), A\in \cp_n\right)$ and set $\cm_A[T,
\bv]=\cmt\bigl(\hat T_A(T, \bv)\bigr)$ for $A\in \cp_n^+$.

\newcommand{\grr}{{\cm_A[T, \bv]}}
\begin{cor}[Recovering $(T,\bv)$ from the $\grr$]
  \label{cor:Span-Ma-n}
 The following map from $\TLn$ to
  $\TDn\times \M(E)^{\cp_n^+}$ defined  by:
\[
  (T, \bv)\mapsto \Bigl( \Span_n(T, \bv), \bigl( \cm_A[T, \bv], A\in
    \cp_n^+\bigr)\Bigr)
\]
is measurable and its restriction to $\TLnzb$ is injective  and bi-measurable.
\end{cor}
\begin{proof}
 Using the measurability of  the functions $\Span$ from $\TLn$ to
    $\TLn$  (see Lemma~\ref{lem:cont-span}), $\bL_n$ from $\TLn$ to
    $\R_+^{\cp_n^+}$  (see Lemma~\ref{lem:bL-cont}), $\Split_n$  from $\TLn$
    to $\left(\TLu\right)^{2^n}$  (see Lemma~\ref{lem:cnt-cck}) and  the map
    $(T,   \bv)   \mapsto   \cmt(T,   \bv)$  from   $\TLu$   to   $\M(E)$
    (see Corollary~\ref{cor:Span-Ma-1}), we  deduce that the following  map, say
    $g_1$, from $\TLn$ to  $\TLn\times \bigl( \R_+\times \M(E)\bigr)^{\cp^+_n}$ is
    measurable:
\[
g_1:  (T, \bv)\mapsto \Bigl( \Span(T, \bv), \bigl( (\ell_A(T, \bv) ,\cm_A[T, \bv]
), A\in
    \cp_n^+\bigr)\Bigr).
\]
Notice that $(T, \bv)\in \TLnzb$ implies that $\hat T_{\{0\}}$ is reduced
to its  root. Using the measurable  functions $\graft_n$ and the  map
defined in
Corollary~\ref{cor:Span-Ma-1},  we easily  deduce that  $g_1$ restricted  to
$\TLnzb$  is   injective  and  thus  bi-measurable   by  Lusin's  theorem.     Since   $\bL_n(T,    \bv)$   is    also   equal    to
$\bL_n\bigl(\Span(T, \bv)\bigr)$,  we deduce  that the  following map  $g_2$, from
$\TLn$ to $\TLn\times \M(E)^{\cp^+_n}$ is measurable:
\[
g_2:  (T, \bv)\mapsto \Bigl( \Span(T, \bv), \bigl( \cm_A[T, \bv], A\in
    \cp_n^+\bigr)\Bigr).
\]
Furthermore, its restriction to $\TLnzb$ is also  injective and thus bi-measurable.
\end{proof}

\section{Formal definitions  of the objects informally introduced in
  Section~\ref{sec:backbone-intro}}
\label{sec:appli}
In this section we check that the topological and measurability  results
obtained in the previous section allows to precisely define the objects which are
introduced in Section~\ref{sec:backbone-intro}. 

\subsection{The elementary grafting operation}
\label{sec:graft-intro}

In Section~\ref{sec:graft-int}, we considered the map:
\begin{equation}
   \label{eq:elementary-graft}
\bigl((T, (\root, x)),(T',\root')\bigr)\mapsto (T\circledast_x
 T',\rho). 
\end{equation}
\begin{lem}
  \label{lem:graft-intro}
  The map~\eqref{eq:elementary-graft} 
  from $\TLu\times \TL$ to $\TL$ is continuous.
  \end{lem}
\begin{proof}
  The   map~\eqref{eq:elementary-graft}  is   the  composition   of  the
  continuous  grafting  the  map from  Lemma~\ref{lem:graft-meas}  (with
  $n=i=1$,  $k=0$  and  $v_i=x$)  with the  map  $\Pi_1^{\circ}$
  defined in~\eqref{eq:def-Pi0} which removes $x$ from the distinguished
  vertices,    as   this    latter   map    is   also    continuous   by
  Lemma~\ref{lem:erase-cont}.
\end{proof}

\subsection{The  grafting operation~\eqref{eq:rev-graft-n}}
\label{sec:appli-graft}
In this section  we give a precise definition of  the grafting procedure
given in~\eqref{eq:rev-graft-n}.  Recall $\Tz$ is the tree reduced to
its root and the infinite spine tree  $\Tu\in \TL$
is identified as the set $\R_+$  with the usual Euclidean distance and
root  $\root=0$. We also recall that $\TLz=\TL\setminus\{\Tz\}$,
see~\eqref{eq:T*loc-K}. 

Unfortunately, it  is not  possible to prove  in general  the regularity
property  of  the  grafting   procedure  $\graft_n$  defined  informally
by~\eqref{eq:rev-graft-n}.  To stay close to this informal presentation,
we consider the case where $n=0$ and $(T, \bv)=\Tu$ is just the infinite
spine  and the  case where  $(T, \bv)$  is a  discrete tree,  element of
$\TDn$.

\subsubsection{The spine case: $(T, \bv)=\Tu$}
\label{sec:def-formal-kesten}
This  case  appear  in  the  definition  of  the  Kesten  tree
in~\eqref{eq:def-Kesten}.    Let   $\cm$   be   a   point   measure   on
$E= \R_+\times  \TLz $  (or equivalently on  $\Tu\times \TLz$)  with the
restriction that $\cm $ belongs to $\M(E)$, the set of point measures on
$E$    which   are    bounded    on   bounded    sets   introduced    in
Section~\ref{sec:def-M}.       Then      the     grafting      procedure
$\graft_0(\Tu,\cm)$ is precisely defined by:
\[
  \graft_0(\Tu,\cm)={\rm P} \circ \tree(\cm),
\]
where the reconstruction map $\tree$ is continuous, see
Lemma~\ref{lem:tree-meas}  and the projection map ${\rm P} $ is also continuous, see
Lemma~\ref{lem:projo-cont}.
More precisely, seeing $\Tu$ as a distinguished spine of
$\graft_0(\Tu,\cm)$, we also have:
\[
  \left(\graft_0(\Tu,\cm), \Tu\right)= \tree(\cm)
  \quad\text{in}\quad \TLd.
\]

It is then elementary to check that the Kesten tree is well defined.
\begin{lem}[The Kesten tree is well defined]
  \label{lem:Kesten}
  Let $\cm(\rd h, \rd T)$ be a Poisson  point measure on $\R_+\times \TL$ with
intensity  $2\beta  \ind_{\{h>0\}}\rd  h \,  \N^{\theta}[\rd  T]$. Then
the Kesten tree $ \ct^*=\graft_0(\Tu, \cm)$ is a $\TL$-valued random
variable. 
\end{lem}

\begin{proof}
   Since ${\rm P} \circ \tree$ is continuous, it is enough to check that
   a.s.\ the random variable $\cm$ belongs to $\M(E)$. Keeping the
   notations from Section~\ref{sec:def-M}, we get:
   \[
     \E\left[\cm(B_E(\varepsilon))\right]=2\beta \varepsilon^{-1}
     \N^\theta [H(T)\geq  \varepsilon]=2\beta \varepsilon^{-1} cq(\varepsilon)<+\infty .
   \]
Thus the point measure $\cm$ is a.s.\ bounded     on   bounded    sets
of $E$. 
\end{proof}

Let us notice  that $(\ct^*, \Tu)=\tree(\cm)$ is  a $\TLd$-valued random
variable, which we call the Kesten tree with its distinguished spine; by
definition~\eqref{eq:def-spine} and~\eqref{eq:def-spine-br},  it is also
a $\TLs$-valued  and a  $\TLsb$-valued random  variable.  Let  us stress
that the Kesten tree has a unique spine (which is then distinguished) if
$\theta>0$ and a countable number of  spines if $\theta<0$ with only one
of them being distinguished.

\subsubsection{The discrete case: $(T, \bv)\in \TDn$}
\label{sec:discrete-app}

For $n\geq 1$, the construction is much more technical (even though the case
$n=1$ could be still handled by hand), and we shall only
consider grafting on a discrete tree,   using the theoretical background
of Section~\ref{sec:graft}. 
First recall  the 
  measurable   application  $\cmt$  defined
in~\eqref{eq:def-cmt} which intuitively from  a locally compact rooted  tree with a
marked infinite spine $(T, \Tut)$ (in the sense of
Section~\ref{sec:TnSn}, with $\Tut$ equivalent to $\Tu$ and seen as a
subset of $T$) 
gives a point measure recording the heights $h_i$ and
the locally compact trees  $T_i\neq \Tz$ such that $(T, \Tut)$ is in
the same equivalence class  
as the infinite spine tree $\Tu$ on which the $T_i$ are grafted 
at $h_i$. See  Proposition~\ref{prop:meas(T)-mesurable} for the
measurable property
of the application $\cmt$.
From the proof of Lemma~\ref{lem:Kesten}, we deduce from
Proposition~\ref{prop:meas(T)-mesurable} that,  if $\cm(\rd h, \rd
T)$ is  a Poisson  point measure on $\R_+\times \TL$ with
intensity  $2\beta  \ind_{\{h>0\}}\rd  h \,  \N^{\theta}[\rd  T]$, then:
\[
  \cmt(\tree(\cm))=\cm.
\]
For this reason, it is natural to identify $\cm$ with the $\TLsb$-valued
random variable $(\ct^*, \Tu)=\tree(\cm)$. 

From Lemma~\ref{lem:graftn-meas} and 
Remark~\ref{rem:graft-Tdis}, we get that the map:
 \[
   \left(T, T^*\right)  \mapsto \graft_n
   \bigl(T,  T^*\bigr)
   \quad\text{with}\quad
\graft_n
   \bigl(T,  T^*\bigr)    =\graft_n
   \bigl(\bL_n(T),  T^*\bigr)
 \]
 from $\TDn\times  \left(\TLsb\right)^{\cp^+_n}$  to
  $\TLn$, which consist in replacing the branches of the discrete tree
  $T$ with truncated  part of locally compact tree with a distinguished
  spine is measurable. Now for $A\in\cp^+_n$,   identifying the locally compact tree with a distinguished
  spine  $T^*_A$ with the point measure $\cmt_A=\cmt(T^*_A)$ allow the
  following identification:
  \[
    \graft_n \left(T, (\cmt_A)_{A\in \cp^+_n}\right) =
\graft_n
\bigl(T,  T^*\bigr).
\]

We   shall    consider   the    case   where   the    random   variables
$(\cmt_A)_{A\in \cp^+_n}$  are independent Poisson point  measure on $E$
with                 the                  same                 intensity
$2\beta \ind_{\{h>0\}}\rd  h \, \N^{\theta}[\rd  T]$. In this  case, the
locally         compact          $n$-pointed         random         tree
$  \graft_n  \left(T,   (\cmt_A)_{A\in  \cp^+_n}\right)$  is  informally
obtained  by   grafting, for all $i\in I$,   on   $x_i\in  T$ the   tree  $T_i\in   \TL$,  where
$\cm'(\rd  x, \rd  \ct)=\sum_{i\in I}  \delta _{(x_i,  T_i)}(\rd x,  \rd
\ct)$   is,  conditionally   on  $T$,   a  Poisson   point  measure   on
$T            \times            \TL$           with            intensity
$2 \beta\, d\length^{T}(\rd x)\N^\theta[\rd \ct]$; and we shall 
write:
\begin{equation}
   \label{eq:def-graft(T,M)}
  \graft_n  (T,   \cm')
    \quad\text{for}\quad
    \graft_n  \left(T,   (\cmt_A)_{A\in  \cp^+_n}\right).
\end{equation}
We shall  stress here  that the  definition of  $\graft_n (T,  \cm')$ is
abusive because the  measure $\cm'$ is not clearly defined  as $T$ is an
equivalent  class  of trees  and  that  furthermore  there is  no  clear
measurability property in $T$, which is mandatory as we want to consider
$T$  a random  variable  in the  $n$-leaves generalized  decomposition  from
Theorem~\ref{thm:k-Bismut}. So in conclusion, the notation:
\begin{equation}
   \label{eq:def-graftn-intuition}
  \graft_n
  (T,   \cm')
\end{equation}
where, conditionally   on  $T$, the random measure 
$\cm'$     a  Poisson   point  measure   on
$T            \times            \TL$           with            intensity
$2 \beta\, d\length^{T}(\rd x)\N^\theta[\rd \ct]$ is an abusive shortcut
for:
\begin{equation}
   \label{eq:def-graftn-TK}
  \graft_n
  \bigl(T,  \ct^*\bigr)
\end{equation}
with $\ct^*=(\ct'_A)_{A\in  \cp^+_n}$ independent Kesten trees with their  
distinguished spine.

Thanks to the measurability property of $\graft_n$ in its two arguments
given in Lemma~\ref{lem:graftn-meas},
the discrete tree $T$ in~\eqref{eq:def-graftn-TK} can be a $\TDn$-valued
random  variable.  In  the  setting  of the  present  paper  the  random
variables $T$ and $\ct^*$ will be independent.

\subsection{Planar trees (Section~\ref{sec:planar})}
\label{sec:planar-def}
Recall $\TDn\subset \TLn$  is the closed subset of (equivalence classes of) discrete trees, that
is,    compact   trees    with   all    the   leaves    being   distinguished,
see~\eqref{eq:def-TDn}.      Let    $(\bt,     \bv)\in    \TDn$     with
$\bv=(v_0=\root, \ldots, v_n)$. (Notice that  the tree $\bt$ has at most
$n$  leaves.)  For  $k\in  \{1,  \ldots, n-1\}$,  let  $p_{k+1}$ denote  the
projection of $v_{k+1}$ on $\Span(\bt, (v_0, \ldots, v_k))$, that is the
only     point    on     $\lb    \root,     v_{k+1}\rb$    such     that
$\lb \root, p_{k+1} \rb= \lb \root, v_{k+1}\rb \cap \Span(\bt, (v_0, \ldots,
v_k))$.     The   discrete    tree   $(\bt,    \bv)$   is    planar   if
$p_{k+1}\in \lb \root, v_k\rb$ for all $k\in \{1, \ldots, n-1\}$. It is easy
to  check  this  condition  is  equivalent to  the  condition used in
Section~\ref{sec:planar}:  for  all
$x\in  \bt$, there  exists $0\leq  i_\mathrm{g}\leq i_\mathrm{d}\leq  n$
such      that     $v_i\in      \bt_x$      if      and     only      if
$i_\mathrm{g}\leq  i\leq   i_\mathrm{d}$.

Let  $\TPn\subset  \TDn$   be  the  set  of   (equivalence  classes  of)
$n$-pointed planar trees.  It is elementary to check that for a discrete
tree  $(\bt, \bv)\in  \TDn$ there  exists  a permutation  (which is  not
unique) $\pi$ such  that the discrete tree $(\bt, \bv  ^\pi)$ is planar.
Arguing as in the proof of Lemma~\ref{lem:Lip-Span}, on get that the map
$(\bt, \bv)\mapsto (\bt, \bv_k)$ with $\bv_k=(v_0, \ldots, v_n, p_k)$ is
5/2-Lipschitz  from  $\TDn$  to  $\TDn$.  Then,  since  the  application
$ (\bt, \bv_k)  \mapsto d(\root, p_k) + d(p_k, v_k)  - d(\root, v_k)$ is
clearly  continuous and  the  latter quantity  is zero  if  and only  if
$p_k\in \lb \root, v_k\rb$, we deduce  that $\TPn$ is a closed subset of
$\TDn$ and thus a closed subset of $\TKn$.

\subsection{Oriented grafting on discrete trees (Section~\ref{sec:planar})}

When considering planar trees  in Section~\ref{sec:planar-def}, we shall
also  be interested  in  a grafting  on  the  left or  on  the right  of
$i\in  \{1,   \ldots,  n\}$,   which  is  the   same  as   the  grafting
\eqref{eq:map-graft-ih}, but  for the  order of  the coordinates  of the
vector   $\bv\circledast  \bv'$.    Recall  that   for  $h\geq   0$  and
$(T, \bv)\in \TDn$,  the vertex $x_{i,h}$ the unique vertex  of $T$ that
satisfies       $x_{i,h}\in      \lb       \root,      v_i\rb$       and
$H(x_{i,h})=     H(v_i)    \wedge     h$,     see
Section~\ref{sec:graft1}.  For  $\epsilon\in \{{\rm  g}, {\rm  d}\}$, we
define      the     grafting      map     $\circledast^{\epsilon}_{i,h}$
by~\eqref{eq:def-gd-graft} with $x=x_{i,h}$ and~\eqref{eq:def-vg-graft},
\eqref{eq:def-vd-graft}    and~\eqref{eq:def-igd-graft},    using    the
convention   stated   thereafter   when    $i_{\rm   g}=0$   (that   is,
$x_{i,h}=\root$)   and   $   i_{\rm   d}=n$.    Let   us   recall   that
$  i_{\rm   g}=  \min\{j\in   \{0,  \ldots,   n\}\,  \colon   \,  v_j\in
T_{x_{i,h}}\}                         $                         (resp.
$   i_{\rm   d}=\max\{j\in  \{0,   \ldots,   n\}\,   \colon  \,   v_j\in
T_{x_{i,h}}\}$) is the left  (resp; right) most distinguished vertices
being a descendant of $x_{i,h}$.

\begin{lem}[Measurability of the left/right grafting maps]
  \label{lem:graft2-gd-cont}
  Let     $n,k\in    \N$,     $i\in     \{0,     \ldots,    n\}$     and
  $\epsilon\in      \{{\rm     g},      {\rm     d}\}$.      The     map
  $\bigl(h, (T,\bv),(T',\bv')\bigr)  \mapsto T\circledast_{i,h}^\epsilon
  T'$ is measurable from $\R_+\times \TLn\times \TLk$ to $\TLnk$.
\end{lem}

\begin{proof}
  We             recall             that             the             map
  $\bigl(h,(T,  \bv)\bigr)  \mapsto   \bigl(T,  (\bv,x_{i,h})\bigr)$  is
  continuous     from    $\R_+\times     \TLn$    to     $\TLnn$,    see
  Section~\ref{sec:graft1},     and     that    the     grafting     map
  $\circledast_{i,h}$  is  continuous,  see  Lemma~\ref{lem:graft2-cont}
  therein.   Thanks  to  the  continuity of  the  permutation  of  the
  distinguished  vertices (so  that
  $i_{\rm g}$  and $ i_{\rm d}$  play a similar role  by considering the
  permutation  $\pi$ on  $\{0,, \ldots,  n\}$ such  that $\pi(0)=0$  and
  $\pi(j)=n+1-j$  otherwise)  and  of the
  removing of distinguished  vertices (so that $x_{i,h}$  can be removed
  from  the distinguished  vertices of  $\bigl(T, (\bv,x_{i,h})\bigr)$),
  see Lemmas~\ref{lem:rerooting} and~\ref{lem:erase-cont},  we only need
  to prove  that the  map $(T,\bv)  \mapsto i_{\rm  g}$, with  $i=n$ and
  $h=0$  or equivalently  $x_{i,h}=v_n$, is  a measurable  function from
  $\TDn$ to $\{0, \ldots, n\}$ for $n\in \N^*$.
This latter result is obvious as 
$\{ i_{\rm  g} >k\}= \bigcap_{j=0}^k \{ v_j \not \in T_{v_n}\}$
  and  as $ v_j$ belongs to $ T_{v_n}$ if and only if $d(\root, v_i)=d(\root, v_n)+
  d(v_n, v_i)$ and the map $(T, \bv) \mapsto (d(v_i, v_j), 0\leq i\leq
  j\leq  n)$ is trivially continuous. 
\end{proof}

\subsection{The $\mathrm{Growth}_n$ function from~\eqref{eq:growth}}
\label{sec:growth}
Let $n\in \N^*$. We consider the function $\mathrm{Growth}_n$ defined
in~\eqref{eq:growth}, which formally is written as first attaching successively
a branch $([0, h], (0, h))\in \TDu$ simply denoted $[0,h]$ to each
distinguished vertices $\bv^*$ of $(T, \bv)$, but the root, (notice that
there is then $2n+1 $ distinguished vertices) and then forgetting 
all the $n$ distinguished vertices $\bv^*$ so that there are only $n+1$
distinguished vertices:
\[
\mathrm{Growth}_n \bigl((T,
  \bv),h\bigr)=  \Pi^{\circ, A_n}_{2n} \circ \mathrm{Growth}'_{n,n} \bigl((T,
  \bv),h\bigr),
\]
where  $  \Pi_{2n}^{\circ,A_n}$  is defined  in~\eqref{eq:def-Pi0}  with
$A_n=(0,                n+1,                \ldots,                2n)$
and for  $i=1, \ldots, n$:
\[
  \mathrm{Growth}'_{n,i} \bigl((T, \bv),h\bigr)= \mathrm{Growth}'_{n,i-1}
  \bigl((T, \bv),h\bigr) \circledast_i[0,h],
\]
with                            the                           convention
$\mathrm{Growth}'_{n,0}  \bigl((T, \bv),h\bigr)=  (T, \bv)$.   Using the
continuity of  the grafting procedure  (see Lemma~\ref{lem:graft2-cont})
and      the     continuity      of     $\Pi_{2n}^{\circ,A_n}$      (see
Lemma~\ref{lem:erase-cont}), we get the following result.

\begin{lem}[Continuity of the map $\mathrm{Growth}_n$]
  \label{lem:growth-meas}
Let $n\in \N^*$.   The map $\mathrm{Growth}_n$ is continuous from
$\TLn\times \R_+$ to  $\TLn$. 
\end{lem}

\subsection{A detail of the proof of Corollary~\ref{cor:Thmbis}}
\label{sec:detail-proof-cor}
Recall $\TLuzb$ defined in \eqref{eq:def-T1-no-br}
is the Borel subset of  $\TLu$ of the trees such that the
root is not a branching vertex and the distinguished vertex is distinct from the
root.  The map $ g\, \colon\, (T, \bv)\mapsto \bigl( d(\root, v), \cmt(T,
\bv)\bigr)$, 
with   $\bv=(\root,  v)$,   defined   on  $\TLuzb$   is  injective   and
bi-measurable,  see Corollary~\ref{cor:Span-Ma-1}.
We deduce that $ (T, \bv)$ is a measurable function of 
$( d(\root, v), \cmt(T,
\bv)\bigr)$ on the image of $\TLuzb$ by $g$. 

Furthermore the  set $\TLuzb$  is of  full measure  with respect  to the
distribution           of          $(\ct,           \bv$)          under
$\N^\theta[\rd  \ct] \,  \Lambda_t  (d v)$,  with  $\bv=(\root, v)$,  as
$\N^\theta$-a.e.   the root  of  $\ct$  is not  a  branching vertex  and
$d(\root, v)=t>0$.  Thus, as $t>0$ is fixed, we get that $(\ct, \bv)$ is a
measurable function of $\cmt(\ct, \bv)$.

\subsection{Construction of the continuum random tree $\ct^{\alpha,
    \theta}$}
\label{sec:ct-aq}

Let $\beta>0$, $\theta, \alpha\in\R_+$ and  let $S^{\alpha, \theta}$ be a
Poisson   point  measure   on   $[0,\infty)$   with  intensity   measure
$\fint(t)\, \rd t$ and $\fint$ given by~\eqref{eq:def-fint}.
We first  consider the case  $\alpha>0$.  Denote by  $(\xi_i, i\in\N^*)$
the increasing  sequence of jumping  times of the  inhomogeneous Poisson
process $(N^{\alpha, \theta}_t=S^{\alpha, \theta}([0, t]), t\geq 0)$. We
consider  the $\TDn$-valued  random  variable  $\fT_{\xi_n}$ of  Section
\ref{sec:tree_ske} for $n\geq 1$  associated to $\fint$.  In particular,
recall that, for  every $n\ge 1$, $\fT_{\xi_n}$ is a  discrete tree with
$n$ distinguished leaves, where all of them are at height $\xi_n$.

For every $n\ge  1$, let $\ct^{n,*}=(\ct_A,A\in\cp_n^+)$ be  a family of
independent Kesten  trees with parameter $(\beta,  \alpha)$, independent
of the tree $\fT_{\xi_n}$.  We define the random marked tree:
\[
  \ct^{(n)}=\left( \Pi_n^\circ  ( \tilde  \ct^{(n)}), \Span^\circ(\tilde
    \ct^{(n)})\right)
  \quad\text{with}\quad  \tilde  \ct^{(n)}=\graft_n
  \left( \fT_{\xi_n},\ct^{n, *}\right).
\]

Thanks to  Lemma~\ref{lem:cont-span-loc} and Lemma~\ref{lem:graftn-meas}
on  the  measurability   of  the  grafting  function,   we  deduce  that
$\ct^{(n)}$  is a  $\TLd$-valued  random variable.   The  family of  the
distributions of the $\TLd$-valued random trees $(\ct^{(n)} ,n\ge 1)$ is
consistent in the sense that, for every $n\ge 1$ and every $t\le \xi_n$,
$r_t^{[2]}(\ct^{(n)})\overset{(d)}{=}r_t^{[2]}(\ct^{(n+1)})$.  It  is in
particular   a   Cauchy  sequence   in   $\TLd$,   and  we   denote   by
$(\ct^{\alpha,\theta}, \fT^{\alpha, \theta})$ its  limit which is thus a
$\TLd$-valued random variable.   By construction, $\fT^{\alpha, \theta}$
and $\fTs$  have the same  distribution.  This construction is  a formal
way to  define the tree  obtained by  grafting on the  infinite discrete
tree $\fTs$ (which  serves as a backbone) at $x_i$  a tree $\ct_i$ where
$((x_i,\ct_i),i\in  I)$ are  the atoms  of  a Poisson  point measure  of
intensity  $2 \beta \length(\rd x)  \N^\theta(\rd\ct)$, where  $\length$ is  the
length measure on $\fTs$.

For           $\alpha=0$,           we           simply           define
$(\ct^{\emptyset,  \theta}, \fT^{0,  \theta})$ as  the Kesten  tree with
parameter $(\beta, \alpha)$.

\medskip
We then define the $\TLd$-valued random process
$\left((\ct_t^{\alpha,\theta}, \fT^{\alpha, \theta}_t), t\geq 0\right)$
by setting:
\[
  \ct_t^{\alpha,\theta}=r_t(\ct^{\alpha,\theta})
  \quad\text{and}\quad
  \fT^{\alpha, \theta}_t=r_t(\fT^{\alpha, \theta}_t),
\]
In particular, thanks to Lemma~\ref{lem:Nt-meas}, the random variable
$(\ct_t^{\alpha,\theta}, \tilde N_t(\fT^{\alpha, \theta}))$ is well
defined.

\section{Proof of Theorem~\ref{thm:k-Bismut}}
\label{sec:proof-k-Bismut}  
 We prove Formula \eqref{eqn:k-Bismut} by induction.
For $n=1$, as $\bT_{1}=[0, t]$ (with root $\root=0$ and distinguished
vertex $v_1=t$), this is Corollary  \ref{cor:Thmbis}.

\medskip
Let $k\in \N^*$. Recall the maps $\bL_k$, from~\eqref{eq:bL=lA}
in Section~\ref{sec:def-vA-wA},
and $\Split_k$ from~\eqref{eq:def-split} in Section~\ref{sec:split}.  For $(T,
\bv)\in \TLk$ and $A\in \cp_k^+$, we write $ \cm_A[T, \bv](\rd h, \rd \bt)$ for
the measure $\cmt(\hat T_A(T, \bv))$ on $E= \R_+\times \TLz$, where
$(\hat T_A(T, \bv), A\in \cp_k)=\Split_k(T, \bv)$ and the
measure $\cmt(T,\bv)$  is defined at the end of  Section~\ref{sec:def-M}. We
also recall the notation   $(\ell_A(T,   \bv),    A\in
\cp_k^+)=\bL_k(T,   \bv)$, and notice
that  $\ell_A(T, \bv)=0$ implies that  $\cm_A[T, \bv]=0$. Let $n\in \N^*$
and $(\Phi_A,  A\in \cp^+_n)$  be a  family of  non-negative measurable
functions defined on $E$. Let  $f$ be a bounded non-negative measurable
function defined  on $\TLn$ (or more  simply on $\TDn$). We  shall first
prove \eqref{eqn:k-Bismut}  for a  non-negative function $F$  defined on
$\TLn$ of the form:
\[
  F(T, \bv)=f(\Span(T, \bv))\, \exp\Bigl\{- \sum_{A\in \cp_n^+} \bigl\langle \Phi_A,
    \cm_A[T, \bv] \bigr\rangle\Bigr\}.
\]

Let $n\geq  2$ and  suppose that  \eqref{eqn:k-Bismut} holds  for $n-1$.
For     $k\in      \{1,     \ldots,      n\}$,     we      denote     by
$\ct^{[k]}$ the tree $\Span(\ct,\bv_k)\in  \TLk$, where  $ \bv_k=(v_0=\root,  \bv_k^*)$
and $\bv_k^*=(v_1,\ldots, v_{k})$; and we simply write $\cm_A^{[k]}$ for
$ \cm_A[\ct,  \bv_k] $ and  $\ell_A^{[k]}$ for $\ell_A(\ct,  \bv_k)$, so
that under $\N^\theta[\rd \ct]\, \ell_t^{\otimes n}(\rd \bv^*)$:
\[
  F(\ct, \bv_n)=f(\ct^{[n]})\,  \exp\Bigl\{- \sum_{A\in \cp_n^+} \bigl\langle \Phi_A,
    \cm_A^{[n]} \bigr\rangle\Bigr\}.
\]
We also write $v^{[k]}_A$ and  $w_A^{[k]}$ for $v_A$ and $w_A$ from
\eqref{eq:def-vA} and \eqref{eq:def-wA} with $(T, \bv)$ replaced by
$(\ct^{[k]}, \bv_k)$; and thus we have $\ell_A^{[k]}=d(w_A^{[k]},
v^{[k]}_A)$.
\medskip

Similarly, under $\E^{\theta,t}$,
for $k\geq 2$, we write also
 $\hat\cm_A^{[k]}$ for the measure $\cmt(T ^*_A)$
 restricted to $[0, \ell_A(\bT_{k})]\times \TLz$,
 $\hat v^{[k]}_A$ and  $\hat w_A^{[k]}$ for $v_A$ and $w_A$ from
\eqref{eq:def-vA} and \eqref{eq:def-wA} with $(T, \bv)$ replaced by
$(\bT_{k}, \bv_k)$,  and $\hat \ell_A^{[k]}=d(\hat w_A^{[k]},\hat v^{[k]}_A)=\ell_A(\bT_k)$.
For $n\geq 2$, simply writing $\bT_n$ for $(\bT_n, \bv_n)$,  we have:
\[
F\bigl(\graft_n(\bT_{n},\ct^*)\bigr)=  f(\bT_{n}) \, \exp \Bigl\{- \sum_{A\in \cp_n^+} \bigl\langle \Phi_A,
\hat\cm_A^{[n]} \bigr\rangle\Bigr\}.
\]
Using  the definition of the Kesten tree via Poisson point measures
  and the definition of the function $\graft_n$, we obtain  in particular that:
\begin{equation}
   \label{eq:EF=EF'}
  \E^{\theta, t}\Bigl[ F\bigl(\graft_n(\bT_{n},\ct^*)\bigr)\Bigr]
=\E^{\theta, t}\left[ F'(\bT_{n})\right],
\end{equation}
where
\begin{equation}
   \label{eq:def -F'}
    F'(\bT_{n})=f(\bT_{n}) \, \exp\Bigg\{- 2 \beta \sum_{A\in
        \cp_{n}^+}
      \int_0^{\hat \ell_A^{[n]}} \rd a \, \N^\theta
      \left[1-\expp{-\Phi_{A}( a, \ct)}\right]  \Bigg\}.
\end{equation}

\medskip

Recall \eqref{eq:def-pv}. Set   $p_n=p_{\bv_{n-1}}(v_n)$   for   the  projection   of   $v_n$   on
$\ct^{[n-1]}$.  Since $\N^\theta$-a.e. $p_n\neq  \root$, we deduce that there
exists $\N^\theta$-a.e.  a unique $B\in  \cp_{n-1}^+$ such that
$p_n\in \rb  w_B^{[n-1]}, v_B^{[n-1]}\rb\subset \ct^{[n-1]}$,  and write
$h_n=d(p_n,  w_B^{[n-1]})$.  Recall  the  function  $\tree$,  defined  in
Section~\ref{sec:def-M} just before Lemma~\ref{lem:tree-meas},
from  $\M(E)$  into  $\TLd$ and  the projection
$\tilde \Pi$  from  $\TLd$  to  $\TL$, defined just before
Lemma~\ref{lem:mes-T*},
which forgets about the marked subtree
defined in Section~\ref{sec:TLs}. We simply write $\tree'=\tilde\Pi\circ \tree$.
On the one hand, we have:
\begin{align}
\nonumber  \ct^{[n]}
  &=\ct^{[n-1]} \circledast_{\min B, H(p_n)} \bigl[ 0, t-H(p_n)\bigr],\\
  \ell_{B}^{[n-1]}
  \label{eq:B:ln=ln-1}
  &= \ell_{B}^{[n]}+  \ell_{B\cup\{n\}}^{[n]},\\
\nonumber   \cm_{B}^{[n-1]}
  &=  \cm_{B\cup\{n\}}^{[n]}+ \cm_{B}^{[n]}(\cdot + h_n,
    \cdot)+ \delta_{\left(h_n, \tree'\left(\cm_{\{n\}}^{[n]}\right)\right)};
\end{align}
and, to fix notation, we shall write:
  \[
   \cm_{B}^{[n-1]}=  \cm_{B}[\ct,
\bv_{n-1}]=
    \sum_{i\in I_{n-1}^B} \delta_{\rh^{[n-1],B}_i, \ct^{[n-1],B}_i} .
  \]
On the other hand, for $A\in \cp_{n-1}^+ $ and $A\neq B$, we have:
\begin{align}
\label{eq:BinA:ln=ln-1}
  B\subset  A
  &\,\Longrightarrow \, \cm_A^{[n-1]}= \cm_{A\cup\{n\}}^{[n]}, \quad
 \cm_A^{[n]} =0, \quad \ell_{A}^{[n-1]}=\ell_{A\cup\{n\}}^{[n]}
    \quad\text{and}\quad    \ell_{A}^{[n]}=0,\\
  \label{eq:AinB:ln=ln-1}
 A\cap B \in \{\emptyset, A\} 
  &\,\Longrightarrow \, \cm_A^{[n-1]}= \cm_A^{[n]}, \quad
  \cm_{A\cup\{n\}}^{[n]}=0, \quad \ell_{A}^{[n-1]}=\ell_{A}^{[n]}
    \quad\text{and}\quad    \ell_{A\cup\{n\}}^{[n]}=0,\\
  \label{eq:BA=0:ln=ln-1}
  A\cap B \not\in \{\emptyset, B, A\}
  &\,\Longrightarrow \, \cm_A^{[n-1]}
    =\cm_{A}^{[n]}= \cm_{A\cup\{n\}}^{[n]}=0
      \quad\text{and}\quad    \ell_{A}^{[n-1]}= \ell_{A}^{[n]}=  \ell_{A\cup\{n\}}^{[n]}=0.
\end{align}
It is also easy to rebuild $(\cm^{[n]}_A, A\in \cp_n^+)$ from
$(\cm^{[n-1]}_A, A\in \cp_{n-1}^+)$ and $v_n$.
\medskip

Set \[F_n  =\N^{\theta} \left[\int_{\ct^{n}}  \Lambda_t^{\otimes n}(\rd
   \bv^*_n)\, F \Big(\ct,\bv_n\Big)\right].\] Considering that
 $\ct_i^{[n-1],B}$ is a subset of $\ct$, we have:
\begin{align*}
  F_n
&    =\N^\theta\Bigg[\int_{\ct^{n-1}}\Lambda_t^{\otimes (n-1)}(\rd
        \bv^*_{n-1})\sum_{B\in \cp_{n-1}^+}\sum_{i\in I_{n-1}^B}
        \int_{\ct_i^{[n-1],B}}\ell_{t }(dv_n)\, F(\ct,\bv_n)\Bigg]\\
& =\N^\theta\Bigg[\int_{\ct^{n-1}}\Lambda_t^{\otimes (n-1)}(\rd
\bv^*_{n-1})\\
  &\hspace{2cm}
    \sum_{B\in \cp_{n-1}^+}\sum_{i\in I_{n-1}^B}
 \Gamma_B\left( \ct^{[n-1]}, H(w_B^{[n-1]}),\cm_{B,i}^{[n-1]},  H(w_B^{[n-1]})+\rh_i^{[n-1],B},\ct_i^{[n-1],B}\right)\\
&\hspace{2cm}  \times \exp\biggl\{
     -\sum_{A\in\cp_{n-1}^+\setminus \{B\}}
     \Bigl\langle \ind_{ \{B\subset A\}}\,  \Phi_{A\cup\{n\}}+ \ind_{ \{A\cap
     B=\emptyset\text{ or }A \}}\, \Phi_{A} ,\cm_A^{[n-1]} \Bigr\rangle \biggr\}
 \Bigg],
\end{align*}
where the measure  $\cm_{B,i}^{[n-1]}$ is the measure $\cm_B^{[n-1]}$
without its atom at $(\rh_i^{[n-1], B}, \ct_i^{[n-1], B})$:
\[
  \cm_{B,i}^{[n-1]}=\cm_B^{[n-1]}-
  \delta_{(\rh_i^{[n-1], B}, \ct_i^{[n-1], B})},
\]
and, for $(T,\bw)\in\TLnm$, $(T', \root')\in\TL$, $\nu\in \M(E)$ and $\rh'\geq h\geq 0$:
\begin{multline*}
\Gamma_B\bigl((T,\bw),h,\nu,  \rh', T'\bigr)=f\bigl(T\circledast_{\min B,
  \rh'}[0,t-\rh']\bigr)\, \exp\left\{-\langle \Phi_{B,\rh'-h},
  \nu\rangle \right\}  \\
\times \int_{T'}\Lambda_{t-\rh'}(dv)\,\exp\biggl\{ - \Bigl\langle
  \Phi_{\{n\}},\cmt\bigl(T',(\root',v)\bigr)\Bigr\rangle\biggr\},
\end{multline*}
with:
\begin{equation}
   \label{eq:def-Phi-h''}
  \Phi_{B,\rh''}(s, \bt)
  =\ind_{\{s\leq   \rh''\}} \Phi_{B\cup\{n\}}( s,\bt)
  + \ind_{\{s>  \rh''\}} \Phi_B(s-\rh'',\bt).
\end{equation}
\medskip

For   $B\in   \cp_{n-1}^+$,     using   the   notation
$    \hat   \cm^{[n]}_B=\sum_{i\in    \hat   I_{n-1}^B}    \delta_{(\hat
  \rh_i^{[n-1],    B},    \hat     \ct_i^{[n-1],    B})}$,   we set    for
$i\in \hat I_{n-1}^B$:
\[
  \hat   \cm_{B,i}^{[n-1]}=\hat \cm_B^{[n-1]}-
  \delta_{\left(\hat \rh_i^{[n-1], B}, \hat \ct_i^{[n-1], B}\right)}.
\]

We deduce from the induction assumption (\emph{i.e.} Equation
\eqref{eqn:k-Bismut} with $n-1$ instead of $n$) and the definition of Kesten tree, with $F_n=(n-1)!(\tilde c_t^\theta)^{2-n}\expp{-2\beta\theta
  t} G_n$ that:
\begin{multline*}
G_n=\E^\theta\Biggl[\sum_{B\in \cp_{n-1}^+}\sum_{i\in \hat I_{n-1}^B}
 \Gamma_B\left( \bT_{n-1}, H(\hat w_{[n-1],B}),\hat \cm_{B,i}^{[n-1]},  H(\hat w_{[n-1],B})+\hat
 \rh_i^{[n-1],B},\hat \ct_i^{[n-1],B}\right)\\
 \times \, \exp\biggl\{
     -\sum_{A\in\cp_{n-1}^+\setminus \{B\}}
     \Bigl\langle \ind_{ \{B\subset A\}}\,  \Phi_{A\cup\{n\}}+ \ind_{ \{A\cap
     B=\emptyset\text{ or }A \}}\, \Phi_{A} ,\hat \cm_A^{[n-1]} \Bigr\rangle
   \biggr\}\Biggr].
\end{multline*}
Since for $A\in \cp_{n-1}^+$, the random measure $\cmt(T^*_A, \hat
\ell_A^{[n-1]})(\rd h', \rd \ct')$ is conditionally given $\hat
\ell_A^{[n-1]}$ a Poisson point measure on $[0,\hat \ell_A^{[n-1]}]\times \TL$ with
intensity $2 \beta \rd h' \, \N^\theta[\rd \ct']$, we deduce from
Palm formula that:
\begin{align*}
  G_n
  &=\E^\theta\Biggl[\sum_{B\in \cp^+_{n-1}} 2\beta \int_0^{\hat
    \ell_B^{[n-1]}} \rd r\,
 \int \N^\theta[\rd \ct] \, \Gamma_{B}\left(\bT_{n-1},H(\hat w_B^{[n-1]}), \hat
 \cm_{B}^{[n-1]}, H(\hat w_B^{[n-1]})+r , \ct\right)\\
 &\hspace{2cm}\times \, \exp\biggl\{
     -\sum_{A\in\cp_{n-1}^+\setminus \{B_x\}}
     \Bigl\langle \ind_{ \{B_x\subset A\}}\,  \Phi_{A\cup\{n\}}+ \ind_{ \{A\cap
     B_x=\emptyset\text{ or }A \}}\, \Phi_{A} ,\hat \cm_A^{[n-1]} \Bigr\rangle
   \biggr\}\Biggr]\\
  &=\E^\theta\biggl[2\beta \int_{\bT_{n-1,t}}\!\!\! \length(\rd x)\,
 \int \N^\theta[\rd \ct] \, \Gamma_{B_x}\left(\bT_{n-1},H(\hat w_{B_x}^{[n-1]}), \hat
 \cm_{B_x}^{[n-1]}, H(x), \ct\right)\\
 &\hspace{2cm}\times \, \exp\biggl\{
     -\sum_{A\in\cp_{n-1}^+\setminus \{B_x\}}
     \Bigl\langle \ind_{ \{B_x\subset A\}}\,  \Phi_{A\cup\{n\}}+ \ind_{ \{A\cap
     B_x=\emptyset\text{ or }A \}}\, \Phi_{A} ,\hat \cm_A^{[n-1]} \Bigr\rangle
   \biggr\}\Biggr],
\end{align*}
where $B_x$ is the only element $B$ of $\cp_{n-1}^+$ such that $x$
belongs to the branch $B$ of $\bT_{n-1}$: $x\in \rb \hat w_{B}^{[n-1]}, \hat v_B^{[n-1]}\rb$,
 where, as $\bT_{n-1}$ is discrete,  we recall that $\Split_{n-1} (\bT_{n-1})=
\bigl(\lb \hat w_{A}^{[n-1]}, \hat v_A^{[n-1]} \rb, A\in \cp_{n-1}\bigr)$ with $\cp_{n-1}=\cp_{n-1}^+\cup\bigl\{\{0\}\bigr\}$.
Using \eqref{eqn:k-Bismut} again for $n=1$ (or Corollary \ref{cor:Thmbis})
gives:
\begin{multline*}
 \int \N^\theta[\rd \ct] \, \Gamma_B(\bT_{n-1,t},h, \nu, \rh', \ct)
 =  f\bigl(\bT_{n-1} \circledast_{\min B, \rh'} [0,t-\rh']\bigr) \expp{-\langle \Phi_{\rh'-h}, \nu\rangle} \\
    \times
\exp\biggl\{ -2\beta \theta(t-
    \rh') -  2\beta\int_0^{t-\rh'}\rd a\,
  \N^\theta \left[1-\expp{-\Phi_{\{n\}}(a, \ct)}\right] \biggr\}.
\end{multline*}
With $x$ chosen according to the length measure
$\length(\rd x)$ on ${\bT_{n-1}}$,
the tree $\bT_{n-1}  \circledast_{\min B_x, H(x)}
\bigl[0,t-H(x)\bigr]$ is obtained by  grafting a
branch of length $t-H(x)$ at $x$ on $\bT_{n-1}$ and thus will simply be denoted
as $\bT_{n-1}  \circledast_x \bigl[0,t-H(x)\bigr]$ (see also
Remark~\ref{rem:comment} for similar notation).
Therefore, we obtain:
\begin{align*}
  G_n
  & =    \E^\theta\Biggl[
    2\beta \int_{\bT_{n-1}}\!\!\! \length(\rd x)\, f\Bigl(\bT_{n-1}
    \circledast_x \bigl[0,t-H(x)\bigr]\Bigr)\,  \exp\Bigl\{-2\beta \bigl(t- H(x)\bigr) \Bigr\}\\
   & \hspace{2cm}
     \times \exp\Biggl\{
     -2 \beta \sum_{A\in\cp_{n-1}^+\setminus \{B_x\}}
   \ind_{ \{B_x\subset A\}}\, \int _0^{\hat
    \ell_A^{[n-1]}} \rd a\, \N^\theta \left[ 1- \expp{-
    \Phi_{A\cup\{n\}}( a, \ct)} \right] \Biggr\}\\
  &\hspace{2cm}
     \times \exp\Biggl\{
     -2 \beta \sum_{A\in\cp_{n-1}^+\setminus \{B_x\}}
   \ind_{ \{A\cap
     B_x=\emptyset\text{ or } A\}}\,\int _0^{\hat
    \ell_A^{[n-1]}} \rd a\, \N^\theta \left[ 1- \expp{-
    \Phi_{A}( a, \ct)} \right] \Biggr\}\\
 &\hspace{2cm}
     \times  \exp\Biggl\{ -  2\beta\int_0^{H(x)-H\left(w^{[n-1]}_{B_x}\right)}\rd a\,
  \N^\theta \left[1-\expp{-\Phi_{B_x\cup \{n\}}(a, \ct)}\right] \Biggr\}\\
 &\hspace{2cm}
     \times  \exp\Biggl\{ -  2\beta\int_0^{H\left(v^{[n-1]}_{B_x}\right)-H(x)}\rd a\,
  \N^\theta \left[1-\expp{-\Phi_{B_x}(a, \ct)}\right] \Biggr\}\\\
 &\hspace{2cm}
     \times  \exp\biggl\{ -  2\beta\int_0^{t-H(x)}\rd a\,
  \N^\theta \left[1-\expp{-\Phi_{\{n\}}(a, \ct)}\right] \biggr\}\Biggr].
\end{align*}

We deduce from Lemma \ref{lem:Tn*=Tn+1} with the density:
\[
\fd(s)=\frac{2\beta\theta\expp{2\beta\theta s}}{\expp{2\beta\theta
    t}-1}\, \ind_{[0,t]}(s)
= \tilde c^\theta _t \, \beta\, \expp{-2\beta\theta (t-s)}\, \ind_{[0,t]}(s)
\]
that for a non-negative measurable function $F''$ defined on $\TLn$ (or $\TDn$):
\[
       \E^{\theta,t}\Bigg[2 \beta \int_{\bT_{n-1}} \length (\rd x)\,
    F''\Bigl(\bT_{n-1}\circledast_x [ 0,t-H(x)]\Bigr)
    \, \expp{- 2\beta\theta (t-H(x))} \Bigg]
    = (\tilde c^\theta _t)^{-1}\, n\,  \E^{\theta,t}\left[F''(\bT_{n})\right].
  \]
Using  similar equations as \eqref{eq:B:ln=ln-1},
  \eqref{eq:BinA:ln=ln-1}, \eqref{eq:AinB:ln=ln-1} and \eqref{eq:BA=0:ln=ln-1}
  stated with $\bT_n$ instead of $(\ct, \bv_n)$ as well as an obvious
  choice of $F''$, we obtain that:
\[
  G_n= (\tilde c^\theta _t)^{-1}\, n \,
  \E^{\theta,t}\left[F'(\bT_{n})\right],
\]
where  $ F'(\bT_{n})$ is given by~\eqref{eq:def -F'}.
Then, we deduce from~\eqref{eq:EF=EF'} that:
\[
  G_n= (\tilde c^\theta _t)^{-1}\, n \,
\E^{\theta,t}\Bigl[F\bigl(\graft_n(\bT_{n},\ct^*)\bigr)\Bigr].
\]
This gives:
\begin{align*}
\N^{\theta} \left[\int_{\ct^{n}}\Lambda_t^{\otimes n}(\rd
    \bv^*_n)\, F (\ct,\bv)\right]
    =F_n    &= (n-1)!(\tilde c_t^\theta)^{2-n}\expp{-2\beta\theta
  t} G_n\\
   & =n!\, \left(\tilde c_t^{\theta}\right)^{1-n}
\expp{-2\beta\theta
     t}\E^{\theta,t}\Bigl[F\bigl(\graft_n(\bT_{n},\ct^*)\bigr)\Bigr].
\end{align*}
Thus,  Equation \eqref{eqn:k-Bismut}  holds for  the functionals  $F$ we
considered.
\medskip

Recall that $\TLnzb$ is the Borel subset of  $\TLn$ of the trees such that the
root is not a branching vertex  and the point vertices (but the root) are distinct from
the root.  The map:
\[
  (T, \bv)\mapsto \Bigl( \Span(T, \bv), \bigl( \cm_A[T, \bv], A\in
    \cp_n^+\bigr)\Bigr)
\]
defined on $\TLnzb$  is one-to-one onto its image  and bi-measurable, see
Corollary~\ref{cor:Span-Ma-n}. Furthermore the set  $\TLnzb$ is of full measure
with   respect    to   the   distribution   of    $(\ct,   \bv$)   under
$\N^\theta[\rd   \ct]  \,   \Lambda_t^{\otimes  n}   (d  \bv^*)$,   with
$\bv=(\root,  \bv^*)$, as  $\N^\theta$-a.e. the  root of  $\ct$ is  not a
branching  vertex.  Thus,  $(\ct,  \bv)$  is  a  measurable  function  of
$\left(\ct^{[n]}, \big( \hat\cm_A^{[n]}  , A\in \cp_n^+\big)\right)$. We
then    conclude     by    the     monotone    class     theorem    that
Equation~\eqref{eqn:k-Bismut}  holds  for  any  non-negative  measurable
function $F$ defined on $\TLn$.

\medskip
\subsubsection*{Acknowledgement}
    We thank the two referees for their precious work and whose comments allowed to considerably improve the presentation of the results.

\newpage
\cuthere \medskip
\begin{center}
 \textbf{Index of notation}
\end{center}
\vspace{1em}

\renewcommand\labelitemi{-}
\setlength{\columnseprule}{1pt}
\begin{multicols}{2}
\noindent
\hrulefill

\begin{center}
   \textbf{Trees and pointed trees}
\end{center}

\begin{itemize}[leftmargin=*,itemsep=0.5em]
\item $T$, $\bt$, $\bT$, $\ct$: generic notations for  trees (or class of
  equiv. trees).
\item $d$: generic distance on a tree.
\item  $\root$: generic notation for the root of trees.
\item $H(x)=d(\root, x)$: height of the vertex $x$.
\item $H(T)$: height of the tree $T$.
  \item $T_x$: subtree of $T$ above the vertex $x\in T$.
\item  $\lb x, y \rb$: the branch  joining the vertices $x$ to $y$.
\item  $\Tz$: the rooted tree reduced to its root.
\item  $\Tu$: the rooted infinite branch.
\item  $\length$ or $\length^T$: length measure on the tree $T$.
\item  $\bv=(v_0=\root, v_1, \ldots, v_n)$: generic notation for
  distinguished vertices of a tree.
\item  $(T, \bv)$ a  (or a class of
  equiv. of) rooted $n$-pointed tree.
\item $(T,
  S)=(T,S, d, \root)$ a  (or a class of
  equiv. of) marked tree  with $\root\in S\subset T$.
  \end{itemize}

 \hrulefill
  \begin{center}
   \textbf{Grafting a tree on a tree}
\end{center}

  \begin{itemize}[leftmargin=*,itemsep=0.5em]

\item    $ (T\circledast_i
T',\bv\circledast\bv')$,  also denoted by $ T\circledast_i
  T'$, is the tree obtained by grafting $T'$ on $T$ at the distinguished vertex $v_i\in T$ and
identifying  the root $\root'$ of $T'$ with $v_i$. The distinguished vertices
$\bv\circledast\bv'$ are the concatenation of the distinguished vertices $\bv$
of $T$ and the distinguished vertices $\bv'$ (but for the root) of $T'$.

\item   $ T\circledast_{i,h}
  T'$, is the tree obtained by grafting $T'$ on $T$ at level $h$ on the
  branch $\lb \root, v_i \rb$.

\item  $ T\circledast_{i, h}^\epsilon T'$, with $\epsilon \in \{\rm{g},
  \rm{d}\}$, same as above but for the distinguished vertices of $T'$ which
  are inserted on the left (if $\epsilon=\rm{g}$) or on the right of
  $v_i$  (if $\epsilon=\rm{d}$).
  \end{itemize}

  \hrulefill
  \begin{center}
   \textbf{Spanning and  truncation}
\end{center}

  \begin{itemize}[leftmargin=*,itemsep=0.5em]

\item  $\Span^\circ(T, \bv)$: the discrete rooted subtree of $T$ spanned  by the distinguished
  vertices $\bv$.
\item  $\Span(T, \bv)$: the rooted tree $ ( \Span^\circ(T, \bv), \bv)$ with the distinguished
  vertices $\bv$.
\item  The map $\Pi^\circ_n$ removes the distinguished vertices (but the root) from
  an $n$-pointed tree: $\Pi^\circ_n(T, \bv)=(T, \root)$. Thus:
  \[
    \Pi^\circ_n(\Span(T, \bv))=\Span^\circ(T, \bv).
  \]
\item $r_t(T, \bv)$: the tree $T$ truncated at level $t$ with the spanned
    tree $\Span^\circ(T, \bv)$, and the distinguished vertices $\bv$.
\item $r_t^{[2]}$,  $r_t^{[2], +}$,     $r_t^{[2], -}$,     $r_*^{[2]}$, $\tilde
  r_t^{[2], +}$: various truncation on marked trees (see
  Sect. \ref{sec:TnSn} and \ref{sec:TLs}).
 \end{itemize}

  \hrulefill
  \begin{center}
   \textbf{Splitting and grafting}
\end{center}

  \begin{itemize}[leftmargin=*,itemsep=0.5em]

\item $\bL_n(T, \bv)$ record the lengths of all the branches of the subtree
  $\Span(T,\bv)$ spanned by the $n$ distinguished vertices:
  \[
  \bL_n(T, \bv)=(\ell_A(T, \bv), A\in
  \cp^+_n),
  \]
with $\cp^+_n$ the set
of all subsets  $A\subset\{1, \ldots, n\}$ such  that $A\neq \emptyset$.
\item    $\Split_n(T, \bv)$ record the subtrees of $T$ associated to  all the branches of
  $\Span(T,bv)$:
\begin{equation}
   \label{eq:def-split-not}
 \Split_n(T, \bv)=\left(
  \hat T_A(T, \bv), A\in \cp_n\right)
  \end{equation}
  with $\cp_n=\cp_n^+\cup\{\{0\}\}$.
\item $\graft_n(T',(T_A^*,A\in\cp_n^+))$: replace the branches labeled
  by $A$,  of the
  discrete $n$-pointed tree $T'$ by the trees $T^*_A$ with a marked
  infinite branch cut at the length $\ell_A(T, \bv)$. (The discrete tree
  $(T', \bv')$ can be coded/replaced by $\bL_n(T', \bv')$.)
\item Intuitively, we have for $(T, \bv)$ a $n$-pointed tree whose root
  is not a branching vertex (see~\eqref{eq:graft-split=id}):
  \[
  (T, \bv)=\graft_n\Bigl(\Span_n(T, \bv), \Split_n(T,\bv)\Bigr).
\]
  \end{itemize}

  \hrulefill
  \begin{center}
   \textbf{Set of (equiv. classes of)  trees}
\end{center}

  \begin{itemize}[leftmargin=*,itemsep=0.5em]
\item  $\TK$ set   of (equiv. classes of) rooted  compact trees.
\item  $\TKn$ set   of (equiv. classes of) rooted  $n$-pointed compact
  trees; $\TKz=\TK$.
\item  $\dghn$ the distance on $\TKn$; $\dghz\equiv\dgh$.
\item  $\TL$ set of (equiv. classes of) rooted loc. compact trees.
\item  $\TLz=\TL\backslash \{\Tz\}$.
\item  $\TLb$ subset of $\TL$ of trees whose root is not a branching
  vertex.
\item      $\TLzb=\TLb \cap \TLz$.
\item  $\TLn$ set of (equiv. classes of) rooted $n$-pointed
      loc. compact trees;  $\TLzer=\TL$.
\item  $\dlghn$ the distance on $\TLn$; $\dlghzer\equiv\dlgh$.
\item  $\TLnb$       subset of $\TLn$ of trees whose root is not a branching
vertex.
\item $\TLnz$ subset of $\TLn$ of trees whose all distinguished vertices (but
  the root) are distinct from the root.
\item $\TLnzb=\TLnb\cap \TLnz$.
\item  $\TDn$ subset of $\TKn\subset \TLn$ of discrete trees.

\item  $\TLd$ set of (equiv. classes of) rooted loc. compact marked trees.

\item $\TLs$ subset of $\TLd$ of  marked trees $(T,S)$ such that
  $S=\Tu$, with $\Tu$  the infinite branch.

  \end{itemize}

  \hrulefill
  \begin{center}
   \textbf{Trees with a marked  branch and point measures}
\end{center}

\begin{itemize}[leftmargin=*,itemsep=0.5em]
\item $E=\R_+\times \TLz$.
\item $\M(E)$ set of point measures on $E$  which are
  bounded on bounded sets of $E$.
\item $\tree: \M(E) \rightarrow \TLs$ maps the measure $\cm=\sum_{i\in
    I} \delta_{h_i, T_i}$ to the marked tree $(T, \Tu)$, with the
rooted  tree $T$  obtained  by grafting the trees $T_i$ on the rooted
infinite branch $\Tu$ at level $h_i$.
\item $\cmt: \TLs \rightarrow \M(E)$ maps the marked tree $(T, \Tu)$
  to the measure $\sum_{i\in
    I} \delta_{h_i, T_i}$ where $T_i\backslash\{\root_i\}$ are the connected
  component of $T\backslash \Tu$ with root $\root_i\in \Tu$ and
  $h_i=d(\root, \root_i)$, where $\root$ is the common root of $T$ and
  $\Tu$.

\item $\cmt$ is  also  defined on $\TLu$.

 \end{itemize}

  \hrulefill
  \begin{center}
   \textbf{Reconstruction results}
\end{center}

\begin{itemize}[leftmargin=*,itemsep=0.5em]

\item With $\mathrm{Id}$  the identity map:
\begin{align*}
  \tree\circ \cmt&=\mathrm{Id}\quad\text{on $\TLs$},\\
    \cmt \circ
    \tree&=\mathrm{Id}\quad\text{on $\ImM= \mathrm{Im} (\cmt)$}.
\end{align*}

\item  $(T,\bv)\in \TLuzb$ can be recovered in a measurable way from
  $\left( d(\root, v), \cmt(T, \bv)\right) $.

\item $(T, \bv) \in \TLnzb$ can be recovered in a measurable way from
  $\left( \Span_n(T, \bv) , (\cm_A[T, \bv], A\in \cp_n^+)\right)
  $, where $\cm_A[T, \bv]=\cmt(\hat T_A(T, \bv))$, with $\hat
  T_A(T, \bv)\in \TLu$ defined by the splitting operation  in~\eqref{eq:def-split-not}.
\end{itemize}

  \hrulefill
\end{multicols}
\cuthere

\bibliographystyle{abbrv}
\bibliography{biblio}

 \end{document}